\documentclass[11pt]{amsart} 
\usepackage[left=2.5cm,top=2.5cm,right=2.5cm]{geometry}             
\usepackage[parfill]{parskip}    

\usepackage{graphicx,caption,subcaption} 

\usepackage{enumerate}                            
\usepackage{appendix}                               
\usepackage{tikz-cd}
\usepackage{comment}
\usepackage{pgfplots}
\usepackage{srcltx}

\usepackage{amssymb,amsmath,amsthm} 

\usepackage{hyperref}

\hypersetup{
    colorlinks=true, 
    linktoc=all,     
    linkcolor=blue,  
    urlcolor=blue,
}

\newcommand{\R}{\mathbb{R}}

\renewcommand{\;}{\, ; \,}

\newcommand{\eps}{\epsilon}
\newcommand{\tbl}[1]{#1} 
\newcommand{\tblu}[1]{\textcolor{black}{#1}}

\newcommand{\norm}[1]{\lVert #1 \rVert}
\newcommand{\abs}[1]{\lvert #1 \rvert} 
\newcommand{\N}{\mathbb{N}}
\newcommand{\supp}[1]{\text{supp}\,(#1)}
\newcommand{\codim}{\textrm{codim}}
\renewcommand{\dim}{\textrm{dim}}
\newcommand{\f}[1]{{}}
\newcommand{\s}{\hspace{0.5pt}}

\newcommand{\PP}{\mathcal{P}}
\newcommand{\XX}{\mathcal{X}}
\def\IX{{\mathcal{X}^{-1}}} 

\newcommand{\COLOP}{\mathcal{Q}}

\newcommand{\sslimit}{\mathcal{S}}
\newcommand{\ccdot}{\,\cdot\,}

\newcommand{\p}{\partial}
\renewcommand{\;}{\hspace{0.5pt} ; \hspace{0.5pt}}
 
\newcommand{\antticomm}[1]{{\color{black}#1}}
\def\gin{\eta_{1}}
\def\ginp{\tilde{\eta}_{1}}
\def\gout{\gamma_{1}}
\def\dgin{\dot\eta_{1}}

\def\dginp{\dot{\tilde{\eta}}_{1} }
\def\tgin{\eta_{2}}
\def\tginp{\tilde{\eta}_{2}}

\def\gino{\sigma_{1}}
\def\ginpo{\tilde{\sigma}_{1}}

\def\nnn{{n}}
\def\OVS{{\overline{\mathcal{P}}^+ }}
\def\SP{{\mathcal{P}}^+ }

\def\IX{{\mathcal{X}^{-1}}} 

\usepackage{color,xcolor} 
\usepackage[normalem]{ulem}                        
\definecolor{traceyedit}{RGB}{123, 25, 164}
\newcommand{\tjb}[1]{{#1}}

\newcommand{\antti}[1]{{#1}}

\makeatletter
\newtheorem*{rep@theorem}{\rep@title}
\newcommand{\newreptheorem}[2]{%
\newenvironment{rep#1}[1]{%
 \def\rep@title{#2 \ref{##1}}%
 \begin{rep@theorem}}%
 {\end{rep@theorem}}}
\makeatother

\newtheorem{theorem}{Theorem}[section]
\newreptheorem{theorem}{Theorem}

\newtheorem{lemma}[theorem]{Lemma}
\newreptheorem{lemma}{Lemma}

\newtheorem{proposition}[theorem]{Proposition}
\newreptheorem{proposition}{Proposition}

\newtheorem{corollary}[theorem]{Corollary}
\newreptheorem{corollary}{Corollary}

\newtheorem{definition}[theorem]{Definition}
\newreptheorem{definition}{Definition}

\theoremstyle{definition}
\newtheorem{remark}[theorem]{Remark}

\title{An Inverse Problem for the Relativistic Boltzmann Equation}

\author[Balehowsky]{Tracey Balehowsky}
\address{Department of Mathematics and Statistics, University of Calgary}
\email{tracey.balehowsky@ucalgary.ca}

\author[Kujanp\"a\"a]{Antti Kujanp\"a\"a}
\address{Department of Mathematics and Statistics, University of Helsinki}

\author[Lassas]{Matti Lassas}
\address{Department of Mathematics and Statistics, University of Helsinki}
\email{matti.lassas@helsinki.fi}

\author[Liimatainen]{Tony Liimatainen}
\address{Department of Mathematics and Statistics, University of Jyv\"askyl\"a, \newline \noindent Department of Mathematics and Statistics, University of Helsinki}
\email{tony.t.liimatainen@jyu.fi}

\date{\today}                                         

\begin{document}

\maketitle

\begin{abstract}
We consider an inverse problem for the Boltzmann equation on a globally hyperbolic Lorentzian spacetime $(M,g)$ with an unknown metric $g$. We consider measurements done in a neighbourhood $V\subset M$  of a timelike path $\mu$ that connects a point $x^-$ to a point $x^+$. The measurements are modelled by a source-to-solution map, which maps a source supported in $V$ to the restriction of the solution to the Boltzmann equation to the set $V$. We show that the source-to-solution map uniquely determines the Lorentzian spacetime, up to an isometry, in the set $I^+(x^-)\cap I^-(x^+)\subset M$. The set $I^+(x^-)\cap I^-(x^+)$ is the intersection of the future of the point $x^-$ and the past of the point $x^+$, and hence is the maximal set to where causal signals sent from $x^-$ can propagate  and return to the point $x^+$. The proof of the result is based on using the nonlinearity of the Boltzmann equation as a beneficial feature for solving the inverse problem.

\end{abstract}

\tableofcontents

%
%

\section{Introduction}



In this paper we study what information can be recovered from indirect measurements of a system governed by the \tjb{Boltzmann equation.} The Boltzmann equation describes nonlinear particle dynamics
%
which arise in many areas of \tjb{physics, such as} atmospheric chemistry, cosmology and condensed matter physics. For example, in cosmology, the Boltzmann equation describes how radiation is scattered by matter such as dust, stars and plasma on an Einstein spacetime. In condensed matter physics, the Boltzmann \tjb{equation may describe the transportation} of electrons or electron-phonon excitations in a media, which can be a metal or a semiconductor. In such situations, the geometry of the spacetime or resistivity of the media may be described by a Lorentzian manifold. 
%
In particular, we investigate the inverse problem of recovering the corresponding Lorentzian manifold of a system behaving according to the Boltzmann equation by making measurements in a confined, possibly small, area in space and time. 

In the kinetic theory we adopt, 
particles travel on a Lorentzian manifold $(M,g)$ along trajectories defined by either future-directed timelike geodesics (for positive mass particles) or future-directed lightlike geodesics (in the case of zero mass particles). In the absence of collisions and external forces, the kinematics of a particle density distribution $u\in C^\infty (TM)$ is captured by the \emph{Vlasov equation} \cite{Ringstrom}, \cite{Rendall} (or Liouville-Vlasov equation \cite{Choquet-Bruhat})  
\[
\XX u(x,p) =0 \quad \text{for } (x,p)\in \OVS M.
\]
Here $u:TM\to \R$ defines a density distribution of particles with position and velocity components $(x,p)\in TM$. Lightlike particles with position and velocity $(x,p)$ are defined by $g(p,p)=0$, timelike particles are defined by $g(p,p)<0$, and spacelike particles are defined by $g(p,p)>0$. The set $\OVS M$ above is the subset of $TM$ of the future directed causal (lightlike and timelike) velocity vectors and $\mathcal{X}: C^\infty (TM) \rightarrow C^\infty (TM)$ is the geodesic vector field. 
In terms of the Christoffel symbols $\Gamma_{\lambda\mu}^\alpha$, the latter is given as
\[
\mathcal{X} = \sum_{\alpha,\lambda,\mu}  p^\alpha  \frac{\partial}{\partial x^\alpha}  -  \Gamma^\alpha_{\lambda \mu}   p^\lambda  p^\mu  \frac{\partial}{\partial p^\alpha },
\]
where $\alpha, \lambda,\mu \in \{ 0,1,2,\ldots, \text{dim}(M)-1\}$. The behaviour of binary collisions is characterized by a \emph{collision operator} 
 \begin{equation}\label{collision_operator_formula_intro}
\COLOP [u, v] (x,p)  
=\int_{\Sigma_{x, p }}  \left[u (x,p) v (x,  q )  - u (x,p' ) v (x,  q' ) \right] A(x,p,q, p', q') dV (x,p \; q,p',q'),
\end{equation}
where   $u,v \in C^\infty( \OVS M)$.
The volume form $ dV (x,p \; q,p',q') $ is a smooth volume form induced on the submanifold
\[
\Sigma_{x,p} := \{p\}\times  \{ (q,p',q')\in ( \overline{\mathcal{P}}_x M )^3  \,:\, p+q = p'+q'\} \subset (T_x M)^4
\]
from a choice of volume form on $(TM)^4$ (for example, one could choose the Leray form \cite[p. 328]{Choquet-Bruhat}).
Above $\overline{\mathcal{P}} M\subset TM$ denotes the bundle of all past and future directed nonzero causal vectors on $M$. We also write $\mathcal{P} M \subset TM$ for the set of all nonzero past and future directed timelike vectors on $M$. The submanifold $\Sigma_{x,p}$ defines the set of particle collisions where conservation of 4-momentum is satisfied \cite[p. 328]{Choquet-Bruhat}, \cite{israel1963relativistic}. 
\begin{remark}
\antticomm{The momentum is often defined in the literature as a covector in $T^*M$ instead of a tangent vector in $TM$. This is  natural in the Hamiltonian context, where motion is described by the canonical flow along the level sets of the given Hamiltonian 
on the symplectic manifold $T^*M$. In our setting, 
 the link between the two pictures on $TM$ (4-position, 4-velocity) and $T^*M$ (4-position, 4-momentum) is given by the isomorphism $(x,p) \to (x,-p^\flat )$. This isomorphism implies that the manifold $\Sigma_{x,p}$ defines conservation of momentum.} 
\end{remark}
For each $(x,p)\in \OVS M$ the function
\[
A(x,p, \cdot,\cdot,\cdot )\in C^\infty(\Sigma_{x, p }) 
\]
is called a \emph{collision kernel} (or a shock cross-section). 

We assume that $(M,g)$ is a globally hyperbolic $C^\infty$ smooth Lorentzian manifold (see Section~\ref{prelim}). Global hyperbolicity allows us to impose an initial state for the particle density function $u$ by using a Cauchy surface $\mathcal{C}\subset M$. Write $\mathcal{C}^\pm$ for the causal future ($+$) or the causal past ($-$) of $\mathcal{C}$. Given an initial state of no particles in $\mathcal{C}^-$ and a particle source $f$ supported in $\mathcal{C}^+$, the kinematics of a distribution of particles $u$ is given by the \emph{relativistic Boltzmann equation} \cite{Choquet-Bruhat}, \cite{Ringstrom}, \cite{Rendall},
\begin{equation}\label{boltzmann}
 \begin{split}
\mathcal{X}u(x,p)-\COLOP[u,u](x,p)&= f(x,p), \quad  (x,p)\in
\OVS M  \\
u(x,p)&= 0, \quad\quad\quad \ \  (x,p)\in \OVS \mathcal{C}^-.
 \end{split}
\end{equation}
The space $\OVS \mathcal{C}^\pm\subset\OVS M$ denotes the set of future directed causal vectors on the causal future ($+$) or past ($-$) of the Cauchy surface $\mathcal{C}$. 
Though we do not consider it in this paper, Boltzmann's H-Theorem (which states that the entropy flux is nonincreasing in time) can be shown to hold for the relativistic Boltzmann kinematic model \eqref{boltzmann} when the collisions are reversible ($A(x,p,q, p', q') = A(x,p',q', p, q)$). We refer the reader to the works of \cite{Choquet-Bruhat}, \cite{Ringstrom}, or \cite{Rendall} for more details on this matter. While equation \eqref{boltzmann} is defined on the entire space $\OVS M$ (which contains particles of all masses), in this paper we will consider \eqref{boltzmann} for particles with mass contained in a finite interval which includes zero. Here zero mass particles are lightlike particles and particles with nonzero mass correspond to timelike particles.



We study an inverse problem where we make observations in an open neighbourhood $V$ of a timelike geodesic in $M$. We assume that $V$ has compact closure without further notice. We denote the bundle of all lightlike future directed vectors with base points in $V$ by $L^+V\subset TV$. The observations are captured by the 
\emph{source-to-solution map for light observations},
\begin{equation}\label{light-source-soln}
 \Phi_{L^+V} : B \rightarrow C_b( L^+V ) ,\quad \Phi_{L^+V}(f) := \Phi(f)|_{L^+V}.
\end{equation}
Here $\Phi$ is the \emph{source-to-solution map} for the relativistic Boltzmann equation
\begin{equation}\label{supposedtolooklikethis}
\Phi : B \rightarrow 
C_b( \OVS M ), \quad \Phi(f) = u,
\end{equation}
where $u$ solves \eqref{boltzmann} with a source $f$. The set $B$ is a neighbourhood of the origin in the function space $C_K( \OVS \mathcal{C}^+):= C^0_K(\OVS \mathcal{C}^+)$, 
where $K\subset \OVS \mathcal{C}^+$ is a fixed compact set.
For integers $k\ge 0$ we define  
 \[
 C_K^k(\OVS \mathcal{C}^+) := \{f\in C^k(\OVS \mathcal{C}^+)\,:\, \text{supp}(f)\subset K\}
 \]
 equipped with the $C^k$ norm\footnote{We define the $C^k$ norm of a function in $C_K^k(\OVS \mathcal{C}^+)$ by fixing a partition of unity and summing up the $C^k$ norms of the local coordinate representations of the function.} and
 \[
 C_b(\OVS \mathcal{C}^+) := \{f\in C(\OVS \mathcal{C}^+)\,:\,f \text{ is bounded}\}
 \]
 equipped with the sup-norm. Note that these spaces above are Banach spaces.
Loosely speaking, the operator $\Phi_{L^+V}$ corresponds to measuring the photons received in $V\subset M$ from particle interactions governed by the Boltzmann equation \eqref{boltzmann}.

Known uniqueness and existence results to \eqref{boltzmann} depend inextricably on the properties of the collision kernel $A$. For general collision kernels, existence of solutions to \eqref{boltzmann} is not known. Complicating the analysis is the fact that it is not completely known what are the physical restrictions on the form of the collision kernel \cite[p. 155]{Ringstrom}. Further, in the case of Israel molecules \cite{israel1963relativistic} where one has a reasonable description of what the collision kernel should be, the collision operator can not be seen as a continuous map
%
%
between weighted $L^p$ spaces~\cite[Appendix F]{Ringstrom}. It is not always clear what is the relationship between conditions on the collision operator $\COLOP$ and the induced conditions on its collision kernel $A$. 

To address the well-posedness of \eqref{boltzmann}, we consider collision kernels of the following type. Here we denote
\[
 \Sigma := \bigcup_{(x,p) \in \overline{\PP}M} \Sigma_{x,p} \subset (TM)^{ 4}.
\]
\begin{definition}[Admissible Kernels]\label{good-kernels}
We say that $A:\Sigma\subset (T M)^4\to\R$ is an \emph{admissible collision kernel} with respect to a  relatively compact open set 
 $W\subset M$ if $A$ satisfies
\begin{enumerate}
     \item\label{conditiony} $A\in C^\infty(\Sigma  )$. 
     
     \item\label{conditionyyyyy} The set $\pi (\text{supp} A)$ is compact and contains $W$ as a subset. \footnote{As we will take $W$ to be the largest domain of causal influence for the set where we take measurements (see \eqref{W-set}, interactions outside $W$ do not influence our observations. Thus for simplicity we assume that $A$ is compactly supported. Here and below $\pi$ denotes the projections $\pi:(TM)^4\to M$ and $\pi:TM\to M$ to the base point.}
     
     \item\label{condition4}  $A(x, p, q ,p',q')>0$, for all $(x,p,q,p',q') \in \Sigma \cap (L^+W \times \overline{\PP} W \times \PP W \times  \PP W)$. 
     
     %
     
     %
     \item\label{conditionyy} There is a constant $C>0$ such that for all $(x,p) \in \OVS M$,
     \[
      \|A(x,p, \cdot,\cdot,\cdot)\|_{L^1(\Sigma_{x,p})} :=   \int_{\Sigma_{x,p}} | A(x,p, q, p',q')  | dV(x,p \; q,p',q')\leq C.
     \]
%
%
     \item\label{conditionyyyy} For every $(x,p) \in \OVS M$ the function $F_{x,p}(\lambda) :=  \|A(x,\lambda p, \cdot, \cdot, \cdot)\|_{L^1(\Sigma_{x,\lambda p})}$, \tbl{$\lambda\in \R$}, is continuously differentiable at $\lambda=0$ and satisfies $F_{x,p}(0) = 0$. 
 \end{enumerate}
\end{definition}
If there is no reason to emphasize the set $W$ we just use the term \emph{admissible collision kernel}. The most important case is when $W$ is a time-like diamond as in the main theorem, see Figure~\ref{fig:W_pic}. The reader may take this as an assumption, although many of the steps in our main proof can be carried through for more general $W$. 

The condition \eqref{condition4} physically means that the collision of particles produces electromagnetic radiation, which propagate along rays of light. We expect that it might be possible to prove similar results for the inverse problem as the ones in this paper with this condition relaxed. However, in that case one needs to construct the (conformal class of the) Lorentzian manifold by a different method than in \cite{Kurylev2018}. We use the method of \cite{Kurylev2018}, which is based on observing rays of light resulting from the nonlinearity of the model. This is why we impose condition  \eqref{condition4}.  

The conditions
(\ref{conditionyy}) and (\ref{conditionyyyy}) are imposed to have control of collisions of relatively high and low momenta particles. These two conditions are required for our proof of the well-posedness of the Cauchy problem for the Boltzmann equation \eqref{boltzmann}. Roughly speaking, these conditions are valid when collisions happen mostly for mid-energy particles. While these conditions arise quite naturally in our proof of well-posedness, there are also different assumptions under which the well-posedness can be proven (at least in Euclidean spaces). We refer to the works \cite{diperna1989cauchy, villani2002review} for examples and references for such cases. (More references to works around the subject will also be given below.) We emphasize that our solution method for the inverse problem is valid whenever the forward problem is well-posed in suitable function spaces. Regarding the inverse problem, the conditions (\ref{conditionyy}) and (\ref{conditionyyyy}) can be replaced by other conditions, which guarantee well-posedness of the Boltzmann equation for small data.

\begin{theorem}\label{boltz-exist}

Let $n\ge3$ and $(M, g)$ be a globally hyperbolic $C^\infty$-Lorentzian $n$-manifold. Let also $\mathcal{C}$ be a Cauchy surface of $M$ and $K\subset \OVS \mathcal{C}^+$ be compact. 
Assume that $A: \Sigma \to \mathbb{R}$ is an admissible collision kernel in the sense of Definition \ref{good-kernels}.
\antticomm{Moreover, assume that $\pi(\text{supp} A) \subset \mathcal{C}^+$}.\footnote{
\antticomm{That is; $\mathcal{C}$ is ``far enough'' in the past. Notice that the set $\pi(\text{supp}A)$ is compact for an admissible $A$.}}

Then, there are open neighbourhoods $B_1 \subset C_K(\OVS \mathcal{C}^+ )$ and $B_2\subset C_b( \OVS  M )$ of the respective origins such that if $f \in B_1$, the relativistic Boltzmann problem \eqref{boltzmann} with source $f$ has a unique solution $u\in B_2$. Further, there is a constant $c_{A,K}>0$ such that~\f{Norm symbols changed.}
\[ \tjb{\|u\|_{C( \OVS M )}\le c_{A,K}\|f\|_{C(\OVS M)}.}\]

\end{theorem}

The well-posedness of~\eqref{boltzmann} has also been addressed in the following works. For exponentially bounded data, and an $L^1$-type bound on the cross-section of the collision kernel, it was shown in \cite{Bichteler} that a short-time unique solution to \eqref{boltzmann} exists in the setting of a $4$-dimensional, globally hyperbolic spacetime. Also for globally hyperbolic geometries, under conditions which require the collision operator to be a continuous map between certain weighted Sobolev spaces, it was proven in \cite{Bancel} that a unique short-time solution exists to \eqref{boltzmann} in arbitrary dimension. If the geometry of $(M,g)$ is close to Minkowski in a precise sense, the collision kernel satisfies certain growth bounds and the initial data satisfied a particular form of exponential decay, then unique global solutions exist to the Cauchy problem for the Boltzmann equation~\cite{Glassey}. We refer to~\cite{Choquet-Bruhat}, \cite{Ringstrom} and \cite{Rendall} for more information about the well-posedness of~\eqref{boltzmann}. Also, the recent paper~\cite{lai-uhlmann-yang}, related to our work, contains a well-posedness result in Euclidean spaces.

We are now ready to present our main theorem. In our theorem, measurements of solutions to the Boltzmann equation are made on a neighborhood $V$ of a timelike geodesic $\hat\mu$ $ : [-1,1] \to  M$ in $(M,g)$.  We prove that measurements made on $V$ both determines a subset $W$ of $M$ and the restriction of the metric to $W$ (up to isometry). The subset $W$ is naturally limited by the finite propagation speed of light. Indeed, given an initial point $x^-\in \hat \mu (-1,1)$ and an endpoint $x^+\in  \hat \mu (-1,1)$ in the future of $x^-$, the set $W$ can be expressed as 
\begin{align}
 W :=  I^-(x^{+}) \cap I^+(x^{-}),\label{W-set}   
\end{align}
where the sets
\begin{align*}
 I^+(z) &=\{y\in M\,:\, \text{ there is a future directed timelike geodesic from } z \text{ to } y \text{ and } z\ne y\}, \\
 I^-(z) &=\{y\in M\,:\,  \text{ there is past directed timelike geodesic from } z \text{ to } y \text{ and } z\ne y \}
\end{align*}
denote respectively the chronological future and past of a point $z\in M$. We call the set $W$ given as above the domain of causal influence. The situation is illustrated in Figure~\ref{fig:W_pic} below. 
 \begin{figure}[ht!]
    \centering
    \includegraphics[scale=0.4]{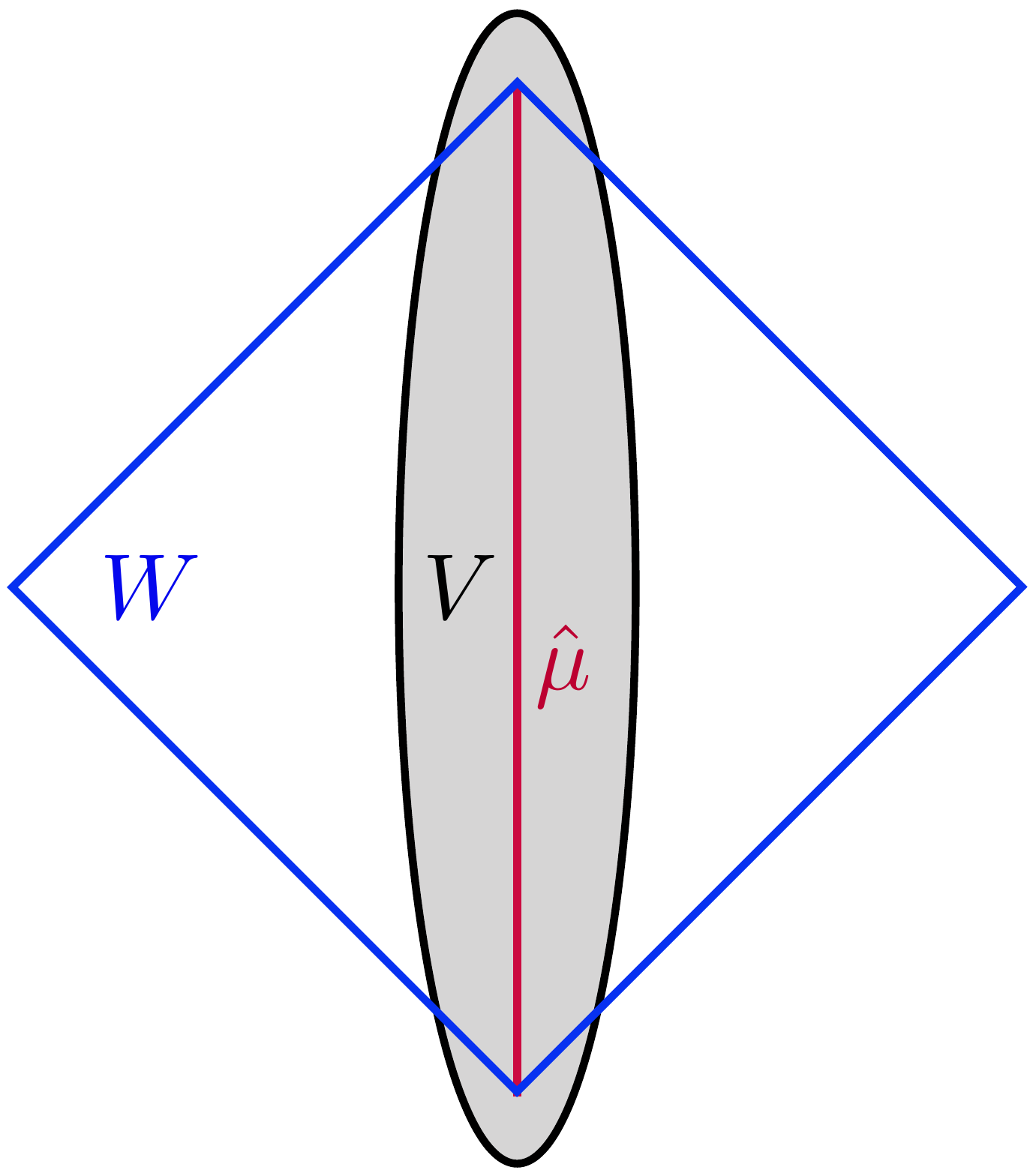}
    \caption{Illustration of the timelike geodesic $\hat\mu$, the known set $V$, and the unknown set $W$.}
    \label{fig:W_pic}
\end{figure}

Before we continue, we introduce our notation. We consider source-to-solutions maps of the Bolzmann equation \eqref{boltzmann} defined on sources, which are supported on different compact sets $K\subset \OVS M$. It follows from Theorem \ref{boltz-exist} that for each such $K$ the source-to-solution map $B_1\to B_2\subset C_b(\OVS M)$ of the Boltzmann equation is well defined, where $B_1\subset C_K(\OVS \mathcal{C}^+)\subset C_c(\OVS \mathcal{C}^+)$. In this case, we denote $B_{1}=B_{1,K}$. In the next theorem we consider source-to-solution maps of Boltzmann equations defined on
\begin{equation}\label{union_domain}
 \mathcal{B}:=\bigcup_{\{K:\s \pi(K)\subset \tblu{V}\}} B_{1,K} \subset C_c(\OVS \mathcal{C}^+).
 \end{equation}
\tblu{We assume that $V$ is fixed from now on.}


\begin{theorem}\label{themain}
Let $(M_1, g_1)$ and $(M_2,g_2)$ be globally hyperbolic $C^\infty$-Lorentzian manifolds of dimension $n \geq 3$.~\f{Dimension added.} Assume that $V$ is a mutual open subset of $M_1$ and $M_2$ and that $g_1|_V=g_2|_V$. Let $\mu : [-1,1] \rightarrow V$ be a smooth timelike curve. Let also $-1< s^{-} < s^{+} < 1$, and define $x^{\pm}:= \mu(s^{\pm})$ and 
\[
W_j :=  I_j^-(x^{+}) \cap I_j^+(x^{-}), \quad j=1,2. 
\]


Suppose that $A_1$ and $A_2$ are admissible kernels in the sense of Definition \ref{good-kernels}. Assume that the source-to-solution maps  of the Boltzmann equation on $(M_1,g_1)$ and $(M_2,g_2)$ with kernels $A_1$ and $A_2$ respectively agree,
\[
 \Phi_{2,\,L^+V}(f)  = \Phi_{1,\,L^+V}(f), \qquad \text{ for all } f\in \mathcal{B},
\]
where $\mathcal{B}$ is as in \eqref{union_domain}.
Then there is an isometric $C^\infty$-diffeomorphism $F: W_1 \rightarrow W_2$,
\[
 F^*g_2=g_1 \text{ on } W_1.
\]
%
\end{theorem}

\tbl{Note that 
alternative to $V$ being a mutual open set of $M_1$ and $M_2$, we could instead assume that there are open sets $V_1\subset M_1$ and $V_2\subset M_2$ and an isometry $\mathcal{I}:V_1\to V_2$ so that the corresponding source-to-solution maps $\Phi_{j,L^+V_j}$, $j=1,2$, satisfy $\Phi_{1,L^+V_1}\mathcal{I}^*=\mathcal{I}^*\Phi_{2,L^+V_2}$. In this case there would be an isometry $W_1\to W_2$. We work with the assumption that $V$ is a mutual open set for simplicity.}

Inverse problems have been studied for equations with various nonlinearity types. Many of the earlier works rely on the fact that a solution to a related linear inverse problem exists. However, it was shown in Kurylev-Lassas-Uhlmann~\cite{Kurylev2018} that the nonlinearity can be used as a beneficial tool to solve inverse problems for nonlinear equations. They proved that local measurements of the scalar wave equation with a quadratic nonlinearity on a Lorentzian manifold determines \tjb{topological, differentiable, and conformal structure of the Lorentzian manifold}. Our proof of Theorem~\ref{themain}, which we explain shortly in the next section, builds upon this work ~\cite{Kurylev2018}. 

Recently, Lai, Uhlmann and Yang \cite{lai-uhlmann-yang} studied an inverse problem for the Boltzmann equation in the Euclidean setting. With an $L^1$ bound and symmetry constraint on the collision kernel, they show that one may reconstruct the collision kernel from boundary 
measurements. 
%
%
Linear equations such as Vlasov, radiative transfer (also called linear Boltzmann), or generalized transport equations model kinematics of particles which do not undergo collisions. We list a very modest selection of the literature on inverse problems for these equations next.  In Euclidean space, Choulli and Stefanov \cite{CS-1} showed that one can recover the absorption and production parameters of a radiative transfer equation from measurements of the particle scattering. They also showed that an associated boundary operator, \tjb{called the} Albedo operator, determines the absorption and production parameters for radiative transfer equations in \cite{CS-2} and \cite{CS-3}. Other results for the recovery of the coefficients of a radiative transfer equation in Euclidean space from the Albedo operator have been proven by Tamasan \cite{Tamasan}, Tamasan and Stefanov \cite{TamasanStefanov}, Stefanov and Uhlmann \cite{StefanovUhlmann}, Bellassoued and Boughanja \cite{BellassouedBoughanja}, and Lai and Li \cite{LaiLin}. Under certain curvature constraints on the metric, an inverse problem for a radiative transfer equation was studied in the Riemannian setting by McDowall \cite{McDowall1}, \cite{McDowall2}. A review of some of these and other inverse problems for radiative transfer and linear transportation equations is given by Bal \cite{GBal-1}. 

The mentioned work~\cite{Kurylev2018} invented a \emph{higher order linearization method} for inverse problems of nonlinear equations in the Lorentzian setting for the wave equation. Other works studying inverse problems for nonlinear hyperbolic equations by using the higher order linearization method include: Lassas, Oksanen, Stefanov and Uhlmann \cite{lassas-oksanen-stefanov-uhlmann}; Lassas, Uhlmann and Wang \cite{lassas-uhlmann-wang-semi}, \cite{lassas-uhlmann-wang-emeq}; and Wang and Zhou \cite{wang-zhou}; Lassas, Liimatainen, Potenciano-Machado and Tyni~\cite{lassas2020uniqueness}. The higher order linearization method in inverse problems for nonlinear elliptic equations was used recently in Lassas, Liimatainen, Lin and Salo~\cite{lassas2019inverse} and Feizmohammadi and Oksanen~\cite{feizmohammadi2020inverse}, and in the partial data case in Krupchyk and Uhlmann~\cite{krupchyk2020remark} and Lassas, Liimatainen, Lin and Salo~\cite{lassas2019partial}.




%
%

\subsection{Theorem \ref{themain} proof summary}\label{proof_summary_main_thm} Now we will explain the key ideas in our proof of Theorem \ref{themain}. To begin let $\hat\mu:[-1,1]\to M$ be a future directed, timelike geodesic and let $V$ be an open neighbourhood of the graph of $\hat\mu$ where we do measurements. For some $-1<s^-<s^+<1$, let $x^\pm=\mu(s^\pm)$, $W=I^-(x^+)\cap I^+(x^-)$ and $w\in W$. We adapt the method of higher order linearization introduced in~\cite{Kurylev2018} for the nonlinear wave equation to our Boltzmann setting \eqref{boltzmann}. Using this approach, we use the nonlinearity to produce a point source at $w$ for the linearized Boltzmann equation. \tblu{The method of \cite{Kurylev2018} is called the \emph{higher order linearization method} and we use the same term for our adapted method in this paper. The nonlinearity $u(x,p)u(x,q)-u(x,p')u(x,q')$ in the Boltzmann equation  depends on different variables. Consequently, our method is not a straightforward generalization of the earlier higher order linearization method that applies to nonlinearities depending only on one varible, say schematically of the form $U^k(X)$, $k\geq 2$.} 

The point source at $w$ is generated formally as follows. Let $\mathcal{C}$ be a Cauchy surface in $M$ and  $f_1,f_2\in C_c(\OVS V)$ be two sources of particles. For sufficiently small parameters $\eps_1$ and $\eps_2$, let $u_{\eps_1f_1+\eps_2f_2}$ be the solution to
\begin{equation}
\begin{split}
\mathcal{X}u_{\epsilon_{1}f_1+\epsilon_{2}f_2}-\COLOP[u_{\epsilon_{1}f_1+\epsilon_{2}f_2},u_{\epsilon_{1}f_1+\epsilon_{2}f_2}]&= \epsilon_{1}f_1 +\epsilon_{2}f_2\quad  \text{in} \quad  \OVS M \\
 u_{\epsilon_{1}, \epsilon_{2}}(x,p)&= 0 \quad\quad\quad\quad\quad \ \text{in} \quad  \OVS \mathcal{C}^- 
\end{split}
\end{equation}
By computing the mixed derivative 
\[
 \Phi^{2L}(f_1,f_2):=\frac{\p}{\p \eps_1}\frac{\p}{\p \eps_2}\Big|_{\eps_1=\eps_2=0} u_{\eps_1 f_1 + \eps_2 f_2},
\]
we show that $\Phi^{2L}(f_1,f_2)$ solves the equation
\begin{align}\label{2nd_lin_source}
\XX \big(\Phi^{2L}(f_1,f_2)\big) &=  \COLOP [ \Phi^L(f_1),  \Phi^L( f_2) ]+  \COLOP [ \Phi^L (f_2),  \Phi^L( f_1)]
\end{align}
on $\OVS M$. Here the functions $\Phi^L(f_l)$, $l=1,2$, are solutions to the linearization of the Boltzmann equation at $u=0$, which solve
\tbl{
\begin{equation}\label{transport_intro}
\begin{split}
\XX \big(\Phi^L(f_l)\big) &= f_l \s\quad \text{in}\quad \OVS M. \\
 \Phi^L(f_l)&= 0 \quad \  \text{in} \quad  \OVS \mathcal{C}^- 
\end{split}
\end{equation}}
%
The transport equation \eqref{transport_intro} is also called the \emph{Vlasov equation}.

\tblu{We mention at this point that if we consider the Boltzmann equation with a source of the form $\eps f$, where $f$ is a fixed source, and linearize at $\eps=0$, the resulting equation would have been \eqref{transport_intro}. Similarly, using a source of the form $f_0+\eps f$ would produce a radiative transport equation. In both of these cases, recovering a time independent metric from local measurements of the corresponding equations are open problems. Also, the scattering kernel in the latter case might be complicated. We will next explain how using sources of the form $\eps_1f_1+\eps_2f_2$ as above we will produce the equation
\eqref{2nd_lin_source} with the right hand side being a delta distribution type source at a point. Consequently, the equation \eqref{2nd_lin_source} propagates data at the point to our measurement set. We will use that to recover the metric. This is a phenomenon not present if one uses just one $\eps$ parameter and linearizes with respect to it. This partly explains the strength of the higher order linearization method. The name of the method is based on the fact that we differentiate the Boltzmann equation with respect to several parameters.}

Next, we build particular functions $f_1$ and $f_2$ such that the right hand side of \eqref{2nd_lin_source}, 
\[
\COLOP [ \Phi^L(f_1),  \Phi^L( f_2) ]+  \COLOP [ \Phi^L (f_2),  \Phi^L( f_1)],\]
 is a point source at $w$. To do this, we note that since $w\in W$, there exists lightlike geodesics $\eta$ and $\tilde\eta$, initialized in our measurement set $V$, which have their first intersection at $w$. We choose timelike geodesics $\gamma_{(\hat{x},\hat{p})}$ and $\gamma_{(\hat{y},\hat{q})}$ with initial data $(\hat{x},\hat{p})$ and $(\hat{y},\hat{q})$ in the bundle $\SP V$ of future-directed timelike vectors and which approximate the geodesics $\eta$ and $\tilde\eta$. 
Then, in Lemma \ref{coro_renerew}, we create $\dim(M)-2$  Jacobi fields on the geodesic $\gamma_{(\hat{y},\hat{q})}$ such that the variation of $\dot\gamma_{(\hat{y},\hat{q})}$ generated by the Jacobi fields is a submanifold $Y_2\subset \SP W\subset TM$. In particular, since $Y_2$ and $Y_1:=\text{graph}(\gamma_{(\hat{x},\hat{p})})\subset TM$ are geodesic flowouts, they can be considered as distributional solutions to the linear transport equation~\eqref{transport_intro}. The Jacobi fields are also constructed so that the projection of $Y_2\subset TM$ to $M$ near the intersection points of $\gamma_{(\hat{x},\hat{p})}$ and $\gamma_{(\hat{y},\hat{q})}$ is a $\dim(M)-1$ dimensional submanifold of $M$ which intersects the other geodesic $\gamma_{(\hat{x},\hat{p})}$ transversally in $M$. This transversality condition enables us to employ microlocal techniques to show in 
Theorem \ref{sofisti} and Corollary \ref{sofisti_further} that $\COLOP [ \Phi^L(f_1),  \Phi^L( f_2) ]+  \COLOP [ \Phi^L (f_2),  \Phi^L( f_1)]$ represents a point source. 

To differentiate $u_{\eps_1f_1+\eps_2f_2}$ with respect to both $\eps_1$ and $\eps_2$, we also prove that the source to solution map of the Boltzmann equation \eqref{boltzmann} is two times Frech\'et differentiable in Lemma~\ref{deriv}. To analyze the nonlinear term $\COLOP [ \Phi^L(f_1),  \Phi^L( f_2) ]+  \COLOP [ \Phi^L (f_2),  \Phi^L( f_1)]$ when $\Phi^L(f_1)$ and $\Phi^L( f_2)$ are distributions in $\mathcal{D}'(TM)$, namely the delta distributions of the submanifolds $Y_1$ and $Y_2$ of $TM$, we view them as conormal distributions in $I^{k_l}(N^*Y_l)$ for some $k_l\in \R$ and $l=1,2$.  This allows us to compute the terms $\COLOP [ \Phi^L(f_1),  \Phi^L( f_2) ]$ and $\COLOP [ \Phi^L (f_2),  \Phi^L( f_1)]$ by using the calculus of Fourier integral operators. Slight complications to this analysis are due the fact that the distributions $\COLOP [ \Phi^L(f_1),  \Phi^L( f_2) ]$ and $\COLOP [ \Phi^L(f_2),  \Phi^L( f_1) ]$ have canonical relations in the sense of the theory of Fourier integral operators over the bundle of causal vectors, which is a manifold with boundary.  

We circumvent these complications by only measuring lightlike signals. This means that we compose the source to solution map with a 
lightlike section $P:V_e\subset V\to L^+V$, where $V_e$ open subset of $V$. 
This reduces our analysis of the term $\COLOP [ \Phi^L(f_1),  \Phi^L( f_2) ]$ in \eqref{2nd_lin_source} 
to an analysis of the term
\[
- \int_{\Sigma_{x, p}} \Phi^L(f_1) (x,p' ) \Phi^L(f_2) (x,  q' )  A(x,p,q, p', q') dV_{x,p} (q,p',q'), 
\]  
where $x$ lies in the open subset $V_e\subset M$ and $p=P(x)\in L^+V$, since the other part of the collision operator~\eqref{collision_operator_formula_intro} in this case yields zero. We analyze similarly the second right hand side term in \eqref{2nd_lin_source}. 

By the above construction, we construct a singularity at $w\in W\subset M$, which we may observe the singularity by using the source to solution map for light observations $\Phi_{L^+V}$.   
Similar to \cite{Kurylev2018}, from our analysis we may recover the so-called 
\tbl{\emph{earliest observation time functions}} from the knowledge of $\Phi_{L^+V}$. Time separation function at $w$ compute the optimal travel time of light signal received from $w$. The collection of all \tbl{earliest observation time functions} from $w$ determines the \emph{earliest light observation set} that is the set of points $y\in \mathcal{U}$ that can be connected to $w$ by lightlike geodesics which have no interior cut points. This set is denoted by $\mathcal{E}_{\mathcal{U}}(w)$. Here $\mathcal{U}$ is an open subset of $V$ and we consider $\mathcal{E}_{\mathcal{U}}(w)$ as the map \tbl{$w\in W\mapsto \mathcal{E}_{\mathcal{U}}(w)$}.  We refer to Section \ref{time_sep_fnct} or~\cite{Kurylev2018} for explicit definitions. 

We apply the above to construct and measure point sources on two different Lorentzian manifolds $(M_1,g_1)$ and $(M_2,g_2)$ to prove:
\begin{proposition} \label{ss_map_first_obs}
Let $\Phi_{1,\,L^+V}$ and $\Phi_{2,\,L^+V}$ be the source-to-solution maps for light observation satisfying the conditions in Theorem \ref{themain}.  Assume that the conditions of Theorem \ref{themain} are satisfied for the Lorentzian manifolds $(M_1,g_1)$ and $(M_2,g_2)$.
Then $\Phi_{2,\,L^+V}  = \Phi_{1,\,L^+V}$ implies
\[
\{\mathcal{E}^1_{\mathcal{U}}({w_1}):\ w_1\in W_1\} = \{\mathcal{E}^2_{\mathcal{U}}({w_2}):\ w_2\in W_2\}. 
\]
\end{proposition}

As shown in \cite[Theorem 1.2]{Kurylev2018}, the sets $\mathcal{E}^{j}_{\mathcal{U}}(w)$, $j=1,2$, determine the unknown regions $W_{j}$ and the conformal classes of the Lorentzian metrics $g_j$ on them. That is, there is a diffeomorphism $F:(W_1,g_1|_{W_1})\to (W_2,g_1|_{W_2})$ such that $F^*g_2=c\s g_1$ on $W_1$. 
After the reconstruction of the conformal class of the metric $(W,g|_W)$, we conclude the proof of Theorem \ref{themain} by showing that the source-to-solution map in $V$ uniquely determines the conformal factor $c(x)$ in $W$. This implies that the source-to-solution map in the set $V$ determines uniquely the isometry type of the Lorentzian manifold $(W,g|_W)$.  

\textbf{Paper outline.} Section \ref{prelim} contains all the preliminary information. There we introduce our notation (\ref{notation}) and recall Lagrangian distributions (\ref{lagrangian-dist}). We provide some expository material on the collisionless (Vlasov) and Boltzmann models of particle kinematics in Sections \ref{Vlasov} and \ref{Bman-model} respectively.  In Section~\ref{microlocal-collision}, we analyze the operator \begin{align*}
\COLOP_{gain} [u_1, u_2] (x,p) :=- \int_{\Sigma_{x, p }} u_1 (x,p' ) u_2 (x,  q' )  A(x,p,q, p', q') dV(x,p;q,p',q'), 
\end{align*} from a microlocal perspective. Viewing particles as conormal distributions, we describe fully the wavefront set which arises from particle collisions. Additionally, in Section \ref{travserse-collisions-section} we construct a submanifold $S_2\subset TM$, whose geodesic flowout is $Y_2\subset TM$, such that the graph $Y_1$ of a geodesic and $Y_2$ satisfies the properties discussed in Section~\ref{proof_summary_main_thm}. 
In Section \ref{main-proof}, we provide the proof of Theorem \ref{themain}, broken down into several key steps:
\begin{enumerate}
\item We show that from our source to solution map data, we can determine if our particle sources interact in $W$. 
\begin{enumerate}
\item[\textbullet] In section \ref{delta_dist_of_subm}, given choices of future-directed timelike vectors $(\hat{x},\hat{p})$ and $(\hat{y},\hat{q})$, as in Section~\ref{proof_summary_main_thm}, we construct the sources $f_1$ and $f_2$ described in Section~\ref{proof_summary_main_thm}.
\item[\textbullet] In this section we additionally construct smooth $C^\infty_c$ approximations $h_1^\epsilon$ and $h_2^\epsilon$, $\eps>0$, of the sources $f_1$ and $f_2$ respectively. This is done mainly for technical reasons, since we only consider the source-to-solution operator of the Boltzmann equation for smooth sources. 
%
\item[\textbullet] In Section \ref{sca123} we describe \antti{a particular} future-directed lightlike vector field $P_e: V_e\to L^+V_e$, $V_e\subset V$ \antti{that admits as an integral curve a fixed optimal geodesic $\gamma$. We compose it with the collision operator~\eqref{collision_operator_formula_intro} and prove
\[
\text{singsupp} (\sslimit)  =  \gamma   \cap  V_e,
\]
where $\mathcal{S}$ is known from our measurements, namely $\sslimit=\lim_{\epsilon \rightarrow 0}  \Phi^{2L} (h_1^\epsilon, h_2^\epsilon)  \circ P_e $. }
\end{enumerate}
\item In section \ref{ss_elos} we prove Proposition \ref{ss_map_first_obs}, which says that we may determine earliest light observation sets from our measurements of $\mathcal{S}$. 
\item From the work in Section \ref{ss_elos} which connects  source-to-solution map data to the earliest light observation sets, we show we may determine the conformal class of the metric.
\item  Lastly, in the section \ref{det-of-conform} we finish the proof by showing that we can identify the conformal factor of the metric.
\end{enumerate}
 Auxiliary lemmas, which include lemmas on the existence of solutions to the Cauchy problem of the Boltzmann equation with small data and the linearization of the source-to-solution map, among others, are contained in the Appendices.

\textbf{Acknowledgments.} 
The authors were supported by the Academy of Finland (Finnish Centre of Excellence in Inverse Modelling and
Imaging, grant numbers 312121 and 309963) and AtMath Collaboration project.

%
\section{Preliminaries}\label{prelim}
\subsection{Notation}\label{notation}
Throughout this paper, $(M,g)$ will be an $n$-dimensional Lorentzian spacetime with $\nnn \geq 3$. We additionally assume that $(M,g)$ is globally hyperbolic. Globally hyperbolicity (see e.g.~\cite{bernal2005}) implies that \tbl{$M=\R \times N$}, for some smooth $n-1$ dimensional submanifold $N\subset M$, and that the metric takes the form
\[
g (x)  = -g_{00}(x) dx^0 \otimes dx^0  + g_{N}(x), \ \ x=(x^0,\overline{x}), \  x^0 \in \R, \ \overline{x} \in N,
\]
where $g_{00} \in C^{\infty} (M)$ is a positive function and $g_{N} (x^0, \, \cdot  \,  )   $ is a smooth Riemannian metric on $N$, for each $x^0 \in \R$.
Global hyperbolicity implies that the manifold $(M,g)$ has a global smooth timelike vector field $\tau$. This vector field  defines the causal structure for $(M,g)$. Further, a globally hyperbolic manifold is both causally disprisoning and \antti{causally} pseudoconvex (see for example \cite{MR1216526}). \tjb{Causally disprisoning means that for each inextendible causal geodesic $\gamma: (a,b)\to M$, $-\infty\le a<b\le \infty$, and any $t_{0}\in(a,b)$, the closures in $M$ of the sets $\gamma(a,t_{0}]$ and $\gamma[t_{0},b)$ are not compact in $M$. }
The manifold $M$ is pseudoconvex if 
for every compact set of $M$, there is a geodesically convex compact set which properly contains it.  


We use standard notation for the causal structure of $(M,g)$. For $x,y\in M$ we write $x\ll y$ (respectively $x\gg y$) if  $x\ne y$ and there is a future directed timelike geodesic from $x$ to $y$ (respectively from $y$ to $x$). We write $x < y$ (respectively $x > y$) if  $x\ne y$ and there is a future directed causal geodesic from $x$ to $y$ (respectively from $y$ to $x$).

For $x\in M$ we write  $J^{+}(x):=\{y\in M\, :\, x< y \text{ or } x=y\}$ for the causal future of $x$ and $J^{-}(x):= \{y\in M\, :\, x> y \text{ or } x=y\}$ for the causal past of $x$. As stated in the introduction, the sets $I^{-}(x):= \{y\in M\, :\, x\gg y\}$ and $I^{+}(x) := \{y\in M\, :\, x\ll y \}$ denote respectively the chronological past and future of $x$. 
The set of points in $M$ which may be reached by lightlike geodesics emanating from a point $x \subset M$ is 
\[
\mathcal{L}^{\pm}(x):= \{  y \in TM : y=\exp_x (sp) , \, p \in L_x^\pm M, \, s \in [0,\infty) \}, 
\]
where $L^\pm_x M \subset TM$ is the set of future ($+$) or past ($-$) directed lightlike vectors in $T_xM$.
%
If $U\subset M$, we also write 
\begin{equation}\label{causal_spaces}
J^{\pm}(U) = \bigcup_{x\in U}  J^{\pm}(x), \quad I^{\pm}(U) = \bigcup_{x\in U}  I^{\pm}(x), \quad \mathcal{L}^{\pm}(U)= \bigcup_{x\in U}  \mathcal{L}^{\pm}(x).
\end{equation}
 
We express the elements of $TM$ as $(x,p)$ where $x\in M$ and $p\in T_xM$. Since $TM=  T\R \times TN$, each $(x,p)\in TM$ can be written in the form
\[
x=(x^{0}, \bar{x}), \ \ \text{ and } \ \ p = (p^{0},\bar{p}), 
\]
for $ x^{0} \in \R$, $\bar{x} \in N$, $p^{0}\in T_{x^{0}} \R$, and $\bar{p} \in T_{\bar{x}} N$.
Given a local coordinate frame $\partial_\alpha : U\subset M \rightarrow TM$, $\alpha = 1,\dots, \nnn$, we identify $TU \approx  U \times \R^{\nnn}$. We use $(x,p^\alpha)$ to denote this local expression.

For $(x,p)\in TM$, we denote by $\gamma_{(x,p)}$ the geodesic with initial position $x$ and initial velocity $p$. 
%
The velocity of $\gamma_{(x,p)}$ at $s\in \R$ is denoted by $\dot\gamma_{(x,p)}(s) \in T_{\gamma_{(x,p)}(s)} M$, $s\in \R$. To simplify the notation we occasionally refer also to the curve $s\mapsto ( \gamma_{(x,p)}(s), \dot\gamma_{(x,p)}(s) ) \in TM$ by $\dot\gamma_{(x,p)}(s)$. 



The appropriate phase spaces for our particles will be comprised of the following subbundles of $TM$. 
\antticomm{The bundle of time-like vectors on given $U \subset M$ is defined as
\[
\mathcal{P} U := \{ (x,p) \in TM \setminus \{0\} : \sqrt{- g( p,p) } > 0, \ \ \pi(x,p)=x  \in U \} \subset TM.
\] 
}
The mass shell of mass $m\geq 0$ on $U$ is  
\begin{equation}\label{massss}
P^m U := \{ (x,p) \in TM \setminus \{0\} : \sqrt{- g( p,p) } = m, \ \ \pi(x,p)=x  \in U \} \subset TM ,
\end{equation}
where $\pi :TM \rightarrow M$ is the canonical projection. The time-like bundle on $U$ then is the union
\[
\mathcal{P}U =  \bigcup_{m>0}P^{m}U.
\]
We denote the inclusion of the light-like bundle $LU := P^0 U$ to this union as
\[
\overline{\mathcal{P}}U  :=  \bigcup_{m\geq 0}P^{m}U
\]
which is the bundle of causal vectors on $U$. 
Notice that $\overline{\mathcal{P}}U$ excludes the zero section; that is, a zero vector is not causal in our conventions.

 In our setting, we may also write for $t \in \R$, $ x\in M$
\begin{alignat*}{2}
P^r_t U &:= (T_t\R  \times TN) \cap P^r U ,   \quad &&P^r_x U :=T_x M \cap P^r U \\
\ \ \mathcal{P}_{t} U &:= ( T_{t} \R \times  T N )\cap \mathcal{P}U , \quad  &&\mathcal{P}_{x}U := ( T_{x} M )\cap \mathcal{P}U,
\end{alignat*}
and write similarly for $\overline{\mathcal{P}}_x U$ and $\overline{\mathcal{P}}_t U$.
%

Each one of the bundles above consists of future-directed and past-directed components which are distinguished by adding ``$+$'' or ``$-$'' to the superscript.
For example, we denote the manifold of future-directed causal vectors on an open set $U \subset M$ by 
\[
\OVS U := \{ (x ,p) \in TM \setminus \{0\} : x \in U, \,  -g(p,p)|_x \geq 0 , \, g(\tau,p)<0 \} \subset \overline{\mathcal{P}} U.
\]

%
%


Let $S \subset \overline{\mathcal{P}}_t M$, $t \in \R$ be a submanifold of $TM$. The \emph{geodesic flowout} of $S$ in $TM$ is the set
\[
K_S := \{ (x,p) \in TM : (x,p)= (\gamma_{y,q}(s),\dot{\gamma}_{y,q}(s)), \ s\in (-T(y,-q),T(y,q)), \ (y,q) \in S \} \subset \overline{\mathcal{P}} M.
\]
Here the interval $(-T(y,-q),T(y,q))$ is inextendible, i.e. the maximal interval where the geodesic $\gamma_{(y,q)}$ is defined. 
\begin{remark}\label{flowout-is-manifold}
The geodesic flowout $K_S$ is a smooth manifold. To see this, note that $K_S$ is the image of the map $(s,(y,q))\mapsto \varphi_s(y,q)$, where $\varphi_s(y,q)$ is the integral curve of the geodesic vector field $\XX$ at parameter time $s$ starting from $(y,q)\in S$. Since $\XX$ is transversal to $\overline{\mathcal{P}}_t M$ (see the  proof of Lemma~\ref{uuiisok}) we have that this map is an immersion, see e.g.~\cite[Theorem 9.20]{Lee_2013}. Given that $M$ is globally hyperbolic, there are no closed causal geodesics. From this it follows that this map is also injective and thus an embedding. Consequently, its image $K_S$ is a smooth submanifold of $TM$.
\end{remark}

We denote $G_S:=\pi(K_S)$, which is the set 
\begin{equation}\label{Gees}
G_S = \{ x \in M : x = \gamma_{(y,q)} (s), \ s\in (-T(y,-q),T(y,q)), \ (y,q) \in S \} \subset M. 
\end{equation}
Unlike $K_S$, the set $G_S$ might  not  be a manifold since the geodesics describing this set might have conjugate points.

If $X$ is a smooth manifold and $Y$ is a submanifold of $X$, the conormal bundle of $Y$ is defined as
\[
 N^*Y=\{(x,\xi)\in T^*X\setminus \{0\}: x\in Y,\, \xi \perp T_xY\}.
\]
Here $\perp$ is understood with respect to the canonical pairing of vectors and covectors. 

\subsection{Lagrangian distributions and Fourier integral operators}\label{lagrangian-dist}\tjb{Here we define the classes of distributions and operators we work with in this paper. We will follow the notation in Duistermaat's book ~\cite{duistermaat2010fourier}. }
\antti{See also the original sources \cite{hormander1971fourier1, MR388464}. }

Let $X$ be a smooth manifold of dimension $n \in \mathbb{N}$. We write $\mathcal{D}'(X)$ for the set of distributions on $X$ and $\mathcal{E}'(X)$ for the set of distributions with compact support on $X$. Let $\Lambda\subset T^*X\setminus \{0\}$ be a conic Lagrangian manifold~\cite[Section 3.7]{duistermaat2010fourier}. We denote the space of symbols  \cite[Definition 2.1.2]{duistermaat2010fourier} of order $\mu \in \R$ (and type $1,0$) on the conic manifold $X \times \R^k \setminus \{0 \}$ by $S^\mu(X \times \R^k \setminus \{0 \})$. 
The H\"ormander space  of Lagrangian distributions of order $m\in \R$ over $\Lambda$  is denoted by $I^m (X ; \Lambda )$, and consists of locally finite sums $u=\sum_j u_j \in \mathcal{D}' (X)$ of oscillatory integrals
\[
u_j(x) = \int_{\R^{k_j}} e^{i\varphi_j (x, \xi) } a_j(x,\xi) d\xi , \quad x\in X.
\]
Here $a_j \in S^{m - k_j/2+ n/4} (X\times \R^{k_j}\setminus \{0\} )$. The phase function $\varphi_j$ is defined on an open cone $\Gamma_j\subset X\times \R^{k_j}$ and satisfies the following two conditions: (a) it is nondegenerate, $d\varphi_j\neq 0$ on $\Gamma_j$ and (b) the mapping 
\[
\Gamma_j \rightarrow \Lambda, \ \  (x,\xi) \mapsto (x,d_x \varphi_j (x,\xi) ) 
\]
defines a diffeomorphism between the set $\{ (x,\xi) \in \Gamma_j : d_\xi \varphi_j (x,\xi)=0 \}$ and some open cone in $\Lambda$. 
When it is clear what the base manifold is we abbreviate $I^m (\Lambda) = I^m (X ; \Lambda )$. 
\antticomm{A particularly important sub-class of Lagrangian distributions is the class of conormal distributions: A distribution conormal to a submanifold $S \subset X$ is defined as an element in $I^m (N^*S)$ (i.e. $\Lambda= N^* S$) and the class of conormal distributions refers to the union of classes $I^m (N^*S)$ over smooth submanifolds $S\subset M$ and $m \in \R$.}
\footnote{\antticomm{The notation $I^\mu(S)$ is often used in the literature. Here $\mu$ stands for the order of the symbol.}}

\tjb{Below we let $X,Y,Z$ be $C^\infty$-smooth manifolds.  }
Let $\Lambda$ be a conic Lagrangian manifold in $T^*(X \times Y) \setminus \{0\}$.  
The manifold corresponds to a \tbl{canonical} relation $\Lambda'$ defined by
\begin{align}\label{relation_defn}
 \Lambda' = \{ (x,y  \;  \xi_x,\xi_y) \in T^*(X \times Y) :  (x,y \;  \xi_x,-\xi_y) \in \Lambda \}.
\end{align}
Equivalently, one may start with a canonical relation and obtain a Lagrangian manifold. In this paper we chose to represent canonical relations as twisted manifolds $\Lambda'$ of Lagrangian manifolds. 
Considering an element in $I^m (X \times Y ; \Lambda)$ as a Schwartz kernel defines an operator
\[
F : C_c^\infty (Y) \rightarrow \mathcal{D}' (X).
\]
The class of operators of this form are called  Fourier integral operators of order $m \in \R $ associated with the relation $\Lambda'$. We denote this class of operators by 
\[
I^m (X,Y ; \Lambda').
\]
We may also identify the space $ I^m (X ; \Lambda)$, $\Lambda \subset T^*X$, with \tbl{$I^m (X,\{0\}  ;  \Lambda \times \{ (0,0)\})$.}~\f{$\Lambda_0$ has another meaning later in the text.} 

The wavefront set of $F$ is denoted by $WF'(F)$ and it is the set
\begin{align}
WF'(F)&=\{(x,\xi^x,y,\xi^y)\in (T^*X\times T^*Y)\setminus \{0\} \,:\, (x,y,\xi^x,-\xi^y)\in WF(G)\},\label{wf_set_fio}
\end{align}
where $WF(G)$ is the wavefront set of the distribution kernel  $G$ of $F$. We also define 
\begin{align}
WF'_X(F)&= \{(x,\xi^x)\in (T^*X)\setminus \{0\} \,:\, (x,y,\xi^x,0)\in WF(G)\},\label{wf_set_fio_image}\\
WF'_Y(F)&= \{(y,\xi^y)\in (T^*Y)\setminus \{0\} \,:\, (x,y,0,\xi^y)\in WF(G)\}\label{wf_set_fio_domain}.
\end{align}

\tjb{Consider two Fourier integral operators} \antti{ $u_1 \in I^{m_1} (X , Y ; \Lambda_1')$ and $u_2 \in I^{m_2} (Y , Z ; \Lambda_2')$ (i.e. Schwartz kernels in $  I^{m_1} (X \times  Y ; \Lambda_1)$ and $ I^{m_2} (Y \times  Z ; \Lambda_2)$, respectively)}
with respective orders $m_1$ and $m_2$ and relations $\Lambda_1'$ and $\Lambda_2'$. Sufficient conditions for $u_1$ and $u_2$ to form a well defined composition $u_1\circ u_2 \in I^{m_1+m_2} (X,Z ; \Lambda_1' \circ \Lambda_2')$ are described in  \antti{theorems \cite[Theorem 2.4.1, Theorem 4.2.2]{duistermaat2010fourier} which provide the rules of basic microlocal operator calculus, often referred to as transversal intersection calculus.} 
 The relation $\Lambda_1' \circ \Lambda_2'$ is defined as 
\begin{align}\label{relation_of_composition}
\Lambda_1' \circ \Lambda_2':=\{ (x,z \, ; \, \xi^x,\xi^z ) &\in T^*( X \times Z): \nonumber\\
&(x,y \, ; \, \xi^x,\xi^y) \in \Lambda_1' , \ \ (y,z \, ; \, \xi^y,\xi^z) \in \Lambda_2' , \  \text{for some}  \ (y \, ; \, \xi^y) \in T^*Y \}.
\end{align}

Lastly, we remark that products of Lagrangian distributions are naturally defined as distributions over an interesting pair of conic Lagrangian manifolds $\Lambda_0,\Lambda_1 \in T^*X\setminus \{0\}$, \tjb{ which are described next. We refer to \cite{ES-thesis, guillemin1981, melrose-uhlmann1979, greenleaf_uhlmann1990, greenleaf-uhlmann} for a thorough presentation of such distributions.}

\tjb{To begin, a pair $(\Lambda_0,\Lambda_1)$ of conic Lagrangian manifolds} $\Lambda_0,\Lambda_1 \in T^*X\setminus \{0\}$, \tjb{is called an intersecting pair if their} intersection is clean: $\Lambda_0\cap \Lambda_1$ is a smooth manifold and 
\begin{align*}
T^*_\lambda(\Lambda_0\cap \Lambda_1) = T^*_\lambda\Lambda_0 \cap T^*_\lambda\Lambda_1 \text{ for all } \lambda\in \Lambda_0\cap\Lambda_1.
\end{align*}

\tjb{Let $(\Lambda_0,\Lambda_1)$ be an intersecting pair of conic Lagrangians with $\text{codim}(\Lambda_{0}\cap \Lambda_{1})=k$, and $\mu,\nu\in \R$. For $\ell \in \mathbb{Z}_{+}$, we denote by $S^{\mu,\nu}(X\times (\R^{\ell}\setminus\{0\})\times \R^{k})$ the space of symbol-valued symbols on $X\times (\R^{\ell}\setminus\{0\})\times \R^{k}$ (see \cite{ES-thesis,greenleaf_uhlmann1990}). We say that $u\in \mathcal{D}'(X)$ is a paired Lagrangian distribution of order $(\mu,\nu)$ associated to $(\Lambda_0,\Lambda_1)$ if $u$ can be expressed as a locally finite sum of the form}
\[
u= u_0 +u_1 + v,
\]
where $u_0\in I^{\mu+\nu}(\Lambda_0)$, $u_1\in I^{\nu}(\Lambda_1)$, and 
\[
v(x) = \int_{\R^{\ell}}\int_{\R^{k}} e^{i\varphi(x;\theta;\sigma)}a(x;\theta;\sigma)\,d\theta d\sigma
\]
for $a(x;\theta;\sigma)\in S^{\tilde\mu,\tilde\nu}(X\times (\R^{\ell}\setminus\{0\})\times \R^{k})$,  some $\ell\in \mathbb{Z}_+$, $\mu = \tilde\mu + \tilde\nu + \frac{\ell+k}{2} - \frac{n}{4}$, and $\nu = -\tilde\nu - \frac{n}{2}$\f{fixed tildes}. Above, the multiphase function $\varphi(x;\theta;\sigma)$ satisfies the following three conditions: for any $\lambda_{0} \in \Lambda_{0}\cap \Lambda_{1}$, (a) there is an open conic set $\Gamma \subset X\times (\R^{\ell}\setminus\{0\})\times \R^{k}$~\f{Brackets added.} such that $\varphi(x;\theta,\sigma) \in C^{\infty}(\Gamma)$, (b) $\varphi(x;\theta,0)$ is a phase function parametrizing $\Lambda_{0}$ in a conic neighbourhood of $\lambda_{0}$, and (c) $\varphi(x;0,\sigma)$ is a phase function parametrizing $\Lambda_{1}$ in a conic neighbourhood of $\lambda_{0}$. 
\antti{In particular,
\begin{align}
u\in I^{\mu,\nu}(\Lambda_0,\Lambda_1) \Longrightarrow WF(u)\subset \Lambda_0\cup\Lambda_1.
\end{align}
}
\tjb{We denote the set of paired Lagrangian distributions of order $(\mu,\nu)$, $\mu,\nu\in \R$, associated to $(\Lambda_0,\Lambda_1)$ by $I^{\mu,\nu}(\Lambda_0,\Lambda_1)$.}

\section{Vlasov and Boltzmann Kinetic Models}\label{LB-models}
%
%

\subsection{The Vlasov model}\label{Vlasov}
We consider a system of particles on a globally hyperbolic $C^\infty$-Lorentzian manifold $(M,g)$ having positions $x\in M$ and momenta $p\in \PP_xM$ in a statistical fashion as a density distribution $u\in \mathcal{D}'( \PP M )$. We assume that the support of $u$ is contained in some compact and proper subset $K\subset \overline{\PP} M$. 
The Vlasov model describes the trajectories of a system of particles in a relativistic setting where there are no external forces and where collisions between particles are negligible. In this situation, the individual particles travel along geodesics determined by initial positions and velocities. 
On the level of density distributions, the behaviour of the system of particles is captured by the \emph{Vlasov equation} 
\begin{align}
\mathcal{X}u&=f,
\end{align}
where $f\in \mathcal{D}'( \PP M )$ is a source of particles and $\mathcal{X} : TM \rightarrow TTM$ on $TM$ is the geodesic vector field. Locally, this vector field has the expression~\f{$x$ added.}  
\[
\mathcal{X} =   p^\alpha  \frac{\partial}{\partial x^\alpha}  -  \Gamma^\alpha_{\lambda \mu}(x)   p^\lambda  p^\mu  \frac{\partial}{\partial p^\alpha }.
\]
Above $\alpha, \lambda,\mu$ sum over  $0,1,2,\ldots, \text{dim}(M)-1$. The functions $\Gamma^\alpha_{\lambda \mu}$ are the Christoffel symbols of the Lorentzian metric $g$.

The vector field  $-i \mathcal{X}:C^\infty(\PP M ) \rightarrow C^\infty(\PP M )$ may be viewed as a pseudo-differential operator with the real-valued principal symbol 
\begin{align}\label{principal_symbol_of_X}
\sigma_{-i\mathcal{X}} (x,p^\alpha \, ; \, \xi^x,  \xi^p ) =   p^\alpha \xi^x_\alpha   - \Gamma^\alpha_{\lambda \mu} (x)  p^\lambda  p^\mu  \xi^p_\alpha.  
\end{align}
Note that as $(M,g)$ is $C^\infty$, the coefficients of $\sigma_{-i\mathcal{X}}$ are $C^\infty$. Writing $\langle \ccdot, \ccdot \rangle$ for the dual paring between $T \mathcal{P}M$ and $T^*\mathcal{P}M$, $\sigma_{-i\mathcal{X}}$ takes the form
\[
\sigma_{-i\mathcal{X}} (x,p \, ; \, \xi) =  \langle \xi , \mathcal{X} \rangle |_{(x,p)} =  \big\langle \xi , \partial_s \antticomm{\big( \gamma_{(x,p)} (s) ,} \dot\gamma_{(x,p)} (s) \antticomm{\big)} \big\rangle\big|_{s=0}.
\]
Therefore, a function $\phi \in C^\infty ( \antticomm{U} )$ \antticomm{on an open $U \subset \mathcal{P}M$} is characteristic for $\mathcal{X}$, that is, satisfies $\sigma_{-i\mathcal{X}} (x,p,d\phi |_{(x,p)} ) = 0$ at every $(x,p) \in U$, if and only if $\phi$ is constant along geodesic velocity curves $s\mapsto (\gamma_{(x,p)}(s) , \dot\gamma_{(x,p)}(s))$.
\antticomm{Any characteristic vector $(x_0,p_0 ; \xi_0) \in T^*\mathcal{P}M$ can be locally extended into such a graph  $(x,p,d\phi|_{(x,p)})$ (see e.g.  \cite[Theorem 3.6.3]{duistermaat2010fourier}). In fact, the smooth function $\phi$ is locally unique, provided a smooth initial data $\phi|_{U}=\psi$ on a neighbourhood $U$ of $(x_0,p_0)$ in  $\mathcal{P}_{t_0} M \ni x_0$. 
It follows that the bicharacteristic strip\footnote{Defined as the image of an integral curve of the Hamiltonian vector field of the principal symbol.} of $\mathcal{X}$ through  $(\tilde{x},\tilde{p} \; \tilde\xi )\in \sigma_{-i\mathcal{X}}^{-1} (0)$ is the image of the smooth curve $s \mapsto (\gamma_{(\tilde{x},\tilde{p})}(s) , \dot\gamma_{(\tilde{x},\tilde{p})}(s) \ ; \ \xi(s) ) \in T^* \mathcal{P}M$, where the development of $\xi(s)$ is governed by the initial value $\xi(0)=\tilde\xi$ and local presentability in the form $\xi(s) = d \phi|_{(\gamma_{(\tilde{x},\tilde{p})} (s),\dot\gamma_{(\tilde{x},\tilde{p})} (s)) }$, 
for $\phi$ as above. 
In particular, the covector $\xi(s)$ is normal to the image of $( \gamma_{(\tilde{x},\tilde{p})} , \dot\gamma_{(\tilde{x},\tilde{p})})$.\footnote{The image of the time-like curve $( \gamma_{(\tilde{x},\tilde{p})} , \dot\gamma_{(\tilde{x},\tilde{p})})$ is a submanifold by global hyperbolicity. Hence, the conormal bundle of it is well defined in the usual sense.} } We denote by $\Lambda_\XX$ the collection of pairs 
\begin{equation}\label{lambda_XX}
 ((x,p \, ; \, \xi) ,(y,q \, ; \, \eta)) \in T^* TM \times T^*  TM
 \end{equation}
that lie on the same bicharacteristic strip of $\XX$.

 \begin{lemma}\label{uuiisok}
Let $(M,g)$ be a globally hyperbolic $C^{\infty}$-Lorentzian manifold and write it in the standard 
form $M = \R \times N$ of global time and space. Let $t \in \R$. 
The vector field $-i\mathcal{X}$ on $\PP M$ is a strictly hyperbolic operator of multiplicity $1$  with respect to the submanifold $\PP_{t}M:= (T_t\R  \times TN) \cap \PP M$. 

\end{lemma}
Recall that vector field $-i\mathcal{X}$ on $\PP M$ is a strictly hyperbolic operator of multiplicity $1$ with respect to the submanifold $\PP_{t}M$ means that all bicharacteristic curves of $-i\XX$ are transversal to $\PP_{t}M$ and there is exactly one solution to the characteristic equation for given initial data $(\tilde{x},\tilde{p} \; \tilde\xi )\in T^*\PP_t M$ (see \cite[Definition 5.1.1]{duistermaat2010fourier}). We omit the proof of the lemma, which is a straightforward verification of the conditions in the definition. 


In this section, the space $(M,g)$ is assumed to be globally hyperbolic
$C^\infty$- Lorentzian manifold, but not necessarily geodesically complete. 
 Given \tbl{the initial data that $u$ vanishes in the past of a Cauchy surface and a source $f$}, we will show that the Vlasov equation $\mathcal{X}u=f$ has a unique solution. There is a substantial amount of literature on this topic (see for example \cite{Choquet-Bruhat}, \cite{Andreasson}, \cite{Ringstrom}, and \cite{Rendall}). 
For example, Lemma \ref{uuiisok} together with standard results for hyperbolic Cauchy problems (see e.g.~\cite[Theorem 5.1.6]{duistermaat2010fourier}) demonstrates uniqueness of solutions to the Vlasov equation. 

Let us denote by $\gamma_{(x,p)}:(-T_1,T_2)\to M$ the inextendible geodesic which satisfies
\begin{equation}
\gamma_{(x,p)}(0)=x \text{ and } \dot{\gamma}_{(x,p)}(0)=p.                                                                                                                            
\end{equation}
Since $(M,g)$ is not necessarily geodesically complete, we might have that $T_1<\infty$ or $T_2<\infty$. Existence of solutions to $\XX u =f$, with initial data $u=0$, in the case where $f(x,p)$ is a smooth function on $\OVS M$ with compact support in the base variable $x\in M$, can be shown by checking that 
\begin{equation}\label{integrate_over_flow}
u(x,p) := \int_{-\infty}^0 f( \gamma_{(x,p)}(t) , \dot\gamma_{(x,p)}(t) )  dt \quad \text{on} \quad (x,p) \in \OVS M
\end{equation}
 is well-defined and satisfies $\XX u =f$. 
%
%
%
\tbl{Indeed, on a globally hyperbolic Lorentzian manifold for a given compact set $K_\pi\subset M$ and a causal geodesic $\gamma$, there are $t_1,t_2\in \antticomm{(-T_1,T_2)}$ such that for parameter times $t\notin[t_1,t_2]$ we have $\gamma(t)\notin K_\pi$.  \antticomm{(See Appendix \ref{appendix-ss}) Thus the tail of the integral is actually zero as the geodesic $s \mapsto \gamma_{(x,p)}(-s) $ eventually exits the compact set  $\pi( \text{supp}f)$ permanently.
}
Because $f$ and the geodesic flow on $(M,g)$ are smooth, the function $u(x,p)$ is smooth.  If $(M,g)$ is not geodesically complete, and if $\gamma_{(x,p)}:(-T_1,T_2)\to M$, we interpret the integral above to be over $(-T_1,0]$. 
 We interpret similarly for all similar integrals without further notice.}

For our purposes it is convenient to have explicit formulas for solutions to the Vlasov equation and therefore we give a proof using the solution formula~\eqref{integrate_over_flow}.
\begin{theorem}\label{vlasov-exist}
Assume that $(M, g)$ is a globally hyperbolic $C^\infty$-Lorentzian manifold. Let $\mathcal{C}$ be a Cauchy surface of $(M,g)$, $K\subset \OVS \mathcal{C}^+$ be compact and $k\ge 0$.
Let also $f\in C_K^k(\OVS M)$. Then, the problem
\begin{alignat}{2}\label{vlasov1}
\mathcal{X}u(x,p)&= f(x,p) \quad &&\text{on} \quad \OVS M \nonumber \\
u(x,p)&= 0 \quad &&\text{on} \quad  \OVS \mathcal{C}^- 
\end{alignat}
has a unique solution $u$ in $C^k ( \OVS M )$. We write $u=\IX(f)$ and call $\IX: C^k_K(\OVS M)\to C^k(\OVS M)$ the solution operator \tbl{to~\eqref{vlasov1}}. In particular, \tbl{if $Z\subset \OVS M$ is compact, there is a constant $c_{k,K,Z}>0$ such that  
\begin{align} \norm{u|_Z}_{C^k ( Z)}\le c_{k,K,Z}\|f\|_{C^{k}
( \OVS M )}. \label{est-vlasov1}
\end{align}
If $k=0$, the estimate above is independent of $Z$: 
\[
 \norm{u}_{C ( \OVS M)}\le c_{K}\|f\|_{C( \OVS M )}.
\]
}
\end{theorem}
We have placed the proof of the theorem in Appendix \ref{appendix-ss}. The proof follows from the explicit formula~\eqref{integrate_over_flow} for the solution. 
The source-to-solution map $\Phi^L$ of the Vlasov equation~\eqref{vlasov1} is defined as 
%
\[ \Phi^L: C_K^k (\OVS M) \to C^k  ( \OVS M ), \ \ \Phi^L(f) = u,\]
where $u$ is the unique solution to \eqref{vlasov1} with the source $f\in C^k_K(\OVS M)$ and $k\geq 0$. As we will soon see, the Vlasov equation is the linearization of the Boltzmann equation with a source-to-solution map $\Phi$.
We also consider the setting where we relax the condition that the source $f$ above is smooth and show that the solution operator to~\eqref{vlasov1} (when considering separately timelike and lightlike vectors) has a unique continuous extension to the class of distribution $f$, which satisfy $WF(f)\cap N^*(\SP \mathcal{C})=\emptyset$. For the precise statement and  proof thereof, please see Appendix \ref{appendix-ss}.

\subsection{The Boltzmann model}\label{Bman-model}

The Boltzmann model of particle kinetics in $(M,g)$ modifies the Vlasov model to take into account collisions between particles. This modification is characterized by the (relativistic) \emph{Boltzmann equation}
\begin{align*}
\mathcal{X} u(x,p) - \COLOP [u,u ] (x,p) &= 0.
\end{align*}
Here $\COLOP[\ccdot,\ccdot]$ is called the \emph{collision operator}. It is explicitly given by
\begin{align}\label{colop}
 \COLOP&:  C_c^\infty (\overline\PP M ) \times C_c^\infty (\overline\PP M)
\rightarrow C^\infty (\overline\PP M )\\
\COLOP[u_1,u_2](x,p) &=\int_{\Sigma_{x, p }} \left[ u_1 (x,p) u_2 (x,  q )  - u_1 (x,p' ) u_2 (x,  q' )\right]  A(x,p,q, p', q') dV(x,p;q,p',q'), \nonumber
\end{align}
where $ dV(x,p;q,p',q')$ for fixed $(x,p)$ is a volume form defined on 
\begin{align}\label{Sigma_Set}
\Sigma_{x,p} &= \{ (p, q, p', q' ) \in   (\overline{\mathcal{P}}_{x}M)^4 : p + q = p' + q' \} \subset (T_xM)^4,
\end{align} 
which is induced by a volume form $dV(x,p, q,p',q')$ on the manifold
\begin{align}\label{Sigma_union}
\Sigma &= \bigcup_{(x,p)\in \OVS} \Sigma_{x,p} \subset (TM)^4.
\end{align}
We call $A=A(x,p,q,p',q')$ the \emph{collision kernel} and assume that it is admissible in the sense of Definition~\ref{good-kernels}. 

Heuristically, $\COLOP$ describes the average density of particles with position and velocity $(x,p)$, which are gained and lost from the collision of two particles $u_1,u_2$. The contribution to the average density from the gained particles is
\begin{align}\label{Q-gain}
\COLOP_{gain} [u_1, u_2] (x,p) :=- \int_{\Sigma_{x, p }} u_1 (x,p' ) u_2 (x,  q' )  A(x,p,q, p', q') dV(x,p;q,p',q'), 
\end{align}
and the contribution from the lost particles is 
\begin{align}
\COLOP_{loss}[u_1, u_2] (x,p):= u_1 (x,p) \int_{\Sigma_{x, p }}   u_2 (x,  q ) A(x,p,q, p', q') dV(x,p;q,p',q'). 
\end{align}



The existence and uniqueness to the initial value problem for the Boltzmann equation has been studied in the literature under various assumptions on the geometry of the Lorentzian manifold, properties of the collision kernel and assumptions on the data, see for example~\cite{Bichteler,Bancel,Glassey}. These references consider the Boltzmann equation on $L^2$-based function spaces.

We consider the following initial value problem for the Boltzmann equation on a globally hyperbolic manifold $(M,g)$ with a source $f$ 
\begin{align*}
\mathcal{X} u - \COLOP [u,u ] &= f, \quad \text{on }\OVS M\\
 u  &= 0, \quad \text{on }\OVS \mathcal{C}^{-},
\end{align*}
where $\mathcal{X}$ is the geodesic vector field on $TM$, $\mathcal{C}$ is a Cauchy surface of $M$ and $\mathcal{C}^{\pm}$ denotes the causal future (+)/past (-) of $\mathcal{C}$. 
We assume that $(M,g)$ is globally hyperbolic and that the collision kernel of $\COLOP$ is admissible in the sense of Definition~\ref{good-kernels}. Our conditions on the collision kernel allow us to consider the Boltzmann equation in the space of continuous functions. We also assume that the sources $f$ are supported in a fixed compact set $K\subset \OVS C^+$. Note that (since $0\notin \OVS M$) this especially means that $f$ is supported outside a neighbourhood of the zero section of $TM$. We work with the following function spaces, each equipped with the supremum norm:
\begin{align*}
C_b(\OVS \mathcal{C}^+) &:= \{h\in C(\OVS \mathcal{C}^+)\,:\,h \text{ is bounded}\},
\end{align*}
and 
\[
C_K(\OVS \mathcal{C}^+) := \{h\in C(\OVS \mathcal{C}^+)\,:\, \text{supp}(h)\subset K\}.
\]

\begin{reptheorem}{boltz-exist}

\tjb{Let $(M, g)$ be a globally hyperbolic $C^\infty$-Lorentzian manifold. Let also $\mathcal{C}$ be a Cauchy surface of $M$ and $K\subset \OVS \mathcal{C}^+$ be compact. 
Assume that $A: \Sigma \to \mathbb{R}$ is an admissible collision kernel in the sense of Definition \ref{good-kernels}.}
\antticomm{Moreover, assume that $\pi (\text{supp}A) \subset \mathcal{C}^+$.} 

\tjb{There are open neighbourhoods $B_1 \subset C_K(\OVS \mathcal{C}^+ )$ and $B_2\subset C_b( \OVS  M )$ of the respective origins such that if $f \in B_1$, the Cauchy problem}
\begin{alignat}{2}\label{boltz1}
\mathcal{X}u(x,p)-\COLOP[u,u](x,p)&= f(x,p) \quad && \text{ on } 
\OVS M \nonumber \\
u(x,p)&= 0 \quad &&\text{ on } \OVS \mathcal{C}^-
\end{alignat}
\tjb{has a unique solution $u\in B_2$. There is a constant $c_{A,K}>0$ such that}
\[ \|u\|_{C( \OVS M )}\le c_{A,K}\|f\|_{C(\OVS M)}.\]


\end{reptheorem}

%
%

We give a proof of Theorem \ref{boltz-exist} in Appendix \ref{appendix-ss}.
As a direct consequence, we find:
\begin{corollary}\label{ss-corollaari}
Assume as in Theorem \ref{boltz-exist} and adopt its notation. 
The source-to-solution map
\[
\Phi : B_1 \rightarrow B_2 , \quad \Phi(f) = u
\]
is well-defined.
Here $u \in  B_2\subset C_b( \OVS M )$ is the unique solution to the Boltzmann equation~\eqref{boltz1} with the source $f\in B_1\subset C_K(\OVS M)$.
\end{corollary}

The proof of Corollary \ref{ss-corollaari} appears in Appendix \ref{appendix-ss}.

Given the existence of the source-to-solution map $\Phi$ associated to the Boltzmann equation, we now formally calculate the first and second Frech\'et differentials of $\Phi$, which will correspond to the first and second linearizations of the Boltzmann equation. 
Let $\mathcal{C}$ be a Cauchy surface in $M$, and $f_1,f_2 \in C_c^\infty(\OVS \mathcal{C}^+)$, and consider the $2$-parameter family of functions
 \[
(\epsilon_1, \epsilon_2 ) \mapsto \Phi ( \epsilon_1 f_1 + \epsilon_2 f_2  )
\]
where $\eps_1,\eps_2$ are small enough so that $\eps_1f_1+\eps_2f_2\in B_1$.
Formally expanding in $\epsilon_1$ and $\epsilon_2$, we obtain
\[
\Phi(\epsilon_1 f_1 + \epsilon_2 f_2) = \epsilon_1 \Phi^L(f_1) + \epsilon_2 \Phi^L(f_2) + \epsilon_1\epsilon_2 \Phi^{2L}(f_1,f_2)  + \text{higher order terms},
\]
where the higher order terms tend to zero as $(\epsilon_1, \epsilon_2 )\to (0,0)$ \tbl{in $C_b(\OVS M)$}. Substituting this expansion of $\Phi(\epsilon_1 f_1 + \epsilon_2 f_2)$ into the Boltzmann equation and  
differentiating in the parameters $\eps_1$ and $\eps_2$ at $\eps_1=\eps_2=0$ yields the equations
\[
\XX \Phi^L(f_j) = f_j, \quad j=1,2,
\]
and
\[
\XX \Phi^{2L}(f_1,f_2) =  \COLOP [ \Phi^L(f_1),  \Phi^L( f_2) ]+  \COLOP [ \Phi^L (f_2),  \Phi^L( f_1)] . \label{second-ss}
\]
We call these equations the first and second linearizations of the Boltzmann equation. Notice that the first and second linearizations are a Vlasov-type equation \eqref{vlasov1} with a source term.

The next lemma makes the above formal calculation precise. We have placed the proof of the lemma in Appendix \ref{appendix-ss}.

%
\begin{lemma}\label{deriv}
Assume as in Theorem \ref{boltz-exist} and adopt its notation. Let $\Phi:B_1 \to B_2 \subset C_b(\OVS M)$, $B_1\subset C_K(\OVS \mathcal{C}^+)$, be the source-to-solution map of the Boltzmann equation. 

The map $\Phi$ is twice Frech\'et differentiable at the origin of $C_K(\OVS \mathcal{C}^+)$. If $f,\, h\in B_1$, then:
\begin{enumerate}
\item The first Frech\'et derivative $\Phi'$ of the source-to-solution map $\Phi$ at the origin satisfies 
\[ \Phi'(0;f)
= \Phi^{L}(f),\]
where $\Phi^L$ is the source-to-solution map of the Vlasov equation~\eqref{vlasov1}. 
\item The second Frech\'et derivative $\Phi''$ of the source-to-solution map $\Phi$ at the origin satisfies 
\[ \Phi''(0;f,h)
= \Phi^{2L}(f,h),\]
where $\Phi^{2L}(f,h)\in C(\OVS M)$ is the unique solution to the equation 
\begin{alignat}{2}\label{second_lin_bman}
\mathcal{X}\Phi^{2L}(f,h)&= \COLOP[\Phi^L(f),\Phi^L(h)] + \COLOP[\Phi^L(h),\Phi^L(f)], \quad &&\text{ on }
\OVS M,\\
\Phi^{2L}(f,h)&= 0, \quad && \text{ on }\OVS \mathcal{C}^-.  \nonumber 
\end{alignat}
\end{enumerate}

\end{lemma}
%
%
%
%

We remark that the terms $\COLOP[\Phi^L(f),\Phi^L(h)] + \COLOP[\Phi^L(h),\Phi^L(f)]$ in~\eqref{second_lin_bman} might not have compact support in $\OVS M$, and thus unique solvability of~\eqref{second_lin_bman} does not follow directly from Theorem~\ref{vlasov-exist}. However, a unique solution to~\eqref{second_lin_bman} is shown to exist in the proof of the above lemma.

In our main theorem, Theorem~\ref{themain}, the measurement data consists of solutions to the Boltzmann equation restricted to our measurement set $V$. By Lemma~\ref{deriv} above, we obtain that the measurement data also determines the solutions to the first and second linearization of the Boltzmann equation restricted to  $V$. We will see from \eqref{second_lin_bman} that the second linearization $\Phi^{2L}$ captures information about the (singular) behaviour of the collision term $\COLOP [ \Phi^L(f_1),  \Phi^L( f_2) ]$. In the next section we will analyze the microlocal behaviour of $\COLOP [ \Phi^L(f_1),  \Phi^L( f_2) ]$. Then, in Section~\ref{main-proof}, we use this analysis to recover information about when particles collide in the unknown region $W$. From such particle interactions in $W$, we will parametrize points in the unknown set $W$ by light signals measured in $V$, which are obtained by restricting $\Phi^{2L}$ to lightlike vectors.

\section{Microlocal analysis of particle interactions}\label{microlocal-collision}

In this section we consider the gain term $\COLOP_{gain}$ of the collision operator and prove that we can extend $\COLOP_{gain}$ to conormal distributions over a certain class of submanifolds in $\PP M$. 
%
%
%
%

\antticomm{
We say that two submanifolds $Y_1,Y_2 \subset \PP M$ has the admissible intersection property at $x_0 \in \pi (Y_1) \cap \pi( Y_2)$ if there is an open neighbourhood $U \subset M$ of $x_0$, 
two submanifolds $N_1,N_2 \subset U$, and smooth time-like vector fields $\theta_k : N_k \to TN_k$, $\theta_k(x) \in T_x N_k$, of $N_k$, $k=1,2$ such that 
\begin{itemize}
    \item $N_1$ and $N_2$ are transversal,
    \item $N_1 \cap N_2 = \{x_0\}$,
    \item $Y_k\cap TU$ coincides with the graph of $\theta_k$, i.e.
    \[
    Y_k\cap TU = \{ (x,\theta_k(x) ) : x \in N_k \}.
    \]
\end{itemize} 
 In this case, $U \cap \pi Y_k  = N_k$. 
}

In our inverse problem, the submanifolds $Y_1$ and $Y_2$ will be the graphs of unions of geodesics, which have a submanifold $S_{1}\subset \PP^+ V$ of dimension $(n-2)$ and a single point space $S_{2}=\{(x,p)\}\subset \PP^+ V$ as their initial data respectively. We will choose $S_1$ and $S_2$ so that there is neighborhood $U$ of the intersection point $x_0\in \pi(Y_1)\cap \pi(Y_2)$ with the following property: to each point $x\in N_1:=U\cap \pi(Y_{1})$ there is a unique geodesic passing, which has initial data at $S_1$, that passes through $x$. The velocity vectors of the corresponding geodesics give the vector field $\theta_1:N_1\to T N_1$. (The vector field $\theta_2$ is just the velocity vector field of the geodesic with initial data $S_2$ and we set $N_2:=U\cap \pi(Y_{S_2})$.) 
 
In coordinates the definition is as follows:

\begin{definition}\label{intersection_coords}[Admissible intersection property]
We say that submanifolds 
$Y_1 \subset \PP M$ and $Y_2 \subset \PP M$ has the admissible intersection property 
at $x_0\in \pi(Y_1)\cap \pi(Y_2)$ if there exists an open neighbourhood \tbl{$U_{x_0} = U\subset M$} of $x_0$ and coordinates
\begin{equation}\label{x_splits}
x = (x',x'') : U \rightarrow \R^\nnn, \quad x' = (x^1,\dots,x^d), \quad x''=(x^{d+1}, \dots, x^\nnn)
\end{equation}
\antticomm{at $x_0 = (0,0)$} 
such that
\begin{equation}\label{turhautuminen1}
\begin{split}
Y_1 \cap TU &= \{ (x,p)   \in TU  : x' = 0  , \ p' = 0, \ p''= \antticomm{ \theta_1(x'')  }  \}    \\
Y_2 \cap TU &= \{  (x,p) \in TU : x'' = 0, \ p''= 0, \ p'= \antticomm{ \theta_2(x') } \}
\end{split}
\end{equation}
\antticomm{for some smooth $ x'' \mapsto \theta_1(x'') $ and $x' \mapsto \theta_2(x')$ 
}
in the associated canonical coordinates $(x,p) = (x',x'',p',p'')$ of $TU$. 
Moreover, for $X \subset M$ we say that the pair $Y_1$ and $Y_2$ has the admissible intersection property in $X$ if either the property holds at every $x_0 \in \pi(Y_1)\cap \pi(Y_2) \cap X$ or $ \pi(Y_1)\cap \pi(Y_2) \cap X = \emptyset $. 

Given the point $x_0$ and its neighbourhood $U$ as above we define the following conic Lagrangian submanifolds \antticomm{of $T^*U$}: 
\[
\Lambda_0 := N^*\{ x'=0, \ x''=0 \} = T_{x_0}^* M, \quad \Lambda_1 := N^*  \{ x   \in U  : x' = 0  \}, \quad \Lambda_2 := N^*  \{ x   \in U  : x'' = 0  \}. 
\]
\antticomm{In terms of the manifolds $N_k= U \cap \pi Y_k$, $k=1,2$ above, 
\[
\Lambda_k = N^* N_k,  \quad \Lambda_0 = N^* (N_1\cap N_2) = T^*_{x_0} M. 
\]
}
\end{definition}


In our inverse problem, the submanifold $S_{1}$ will be constructed  in \antti{Corollary \ref{coro_renerew}} 
so that 
the geodesic flowouts $Y_{1} := K_{S_{1}}$ and $Y_{2}:= K_{S_{2}}$, where $S_{2}$ is a single point in $\PP^+ V$, have admissible intersection property (see Figure \ref{fig:intersection}). 
In a fixed compact set, 
the intersection points $\pi(Y_1)\cap \pi(Y_2)$ will be finite and thus also \antti{discrete}. 
Moreover, due to the admissible intersection property of $Y_1$ and $Y_2$
the map $\pi :TM \to M$ defines a diffeomorphism from $TU\cap Y_j$ to $U\subset \pi (Y_j)$, $j=1,2$, where $U$ is open neighborhood of an intersection point. 
Globally the set $G_{S_1}:=\pi(Y_{S_1})$ (also $G_{S_2}:=\pi(Y_{S_2})$ if $S_2$ is not required to be a single point space) may fail to be a manifold due to caustic effects. 

The constructions of $S_1$ and $S_2$ are done so that the admissible intersection property holds on both of the manifolds $(M_j,g_j)$, $j=1,2$, simultaneously. The sets $\pi(S_1)$ and $\pi(S_2)$ will be subsets of $V$, which is the set where we do our measurements.

\begin{figure}[h]
 \includegraphics[width=0.4\textwidth]{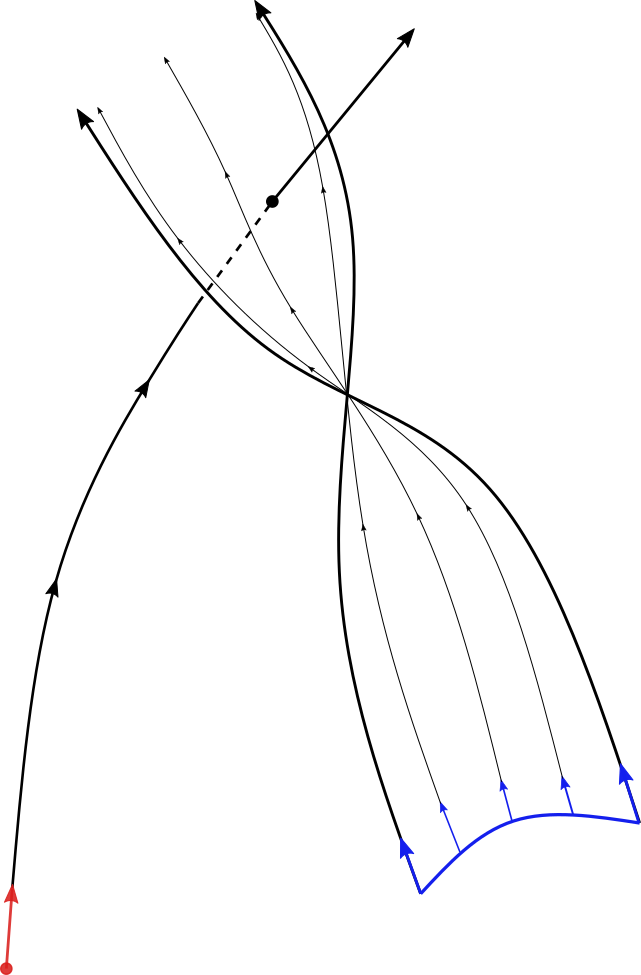}
 \caption{Schematic of an admissible intersection. The manifolds $S_1\subset \PP^+V$ and $S_2 \subset \PP^+V$ ($S_2$ a single point space) are indicated in blue and red, respectively. The geodesic flowouts $Y_1=K_{S_1}$ and $Y_2= K_{S_2}$ have the admissible intersection property. Near the intersection point, $G_{S_1}$ behaves as a manifold. For the construction of $S_1$ and $S_2$ with this property, see Corollary \ref{coro_renerew}.}
 \label{fig:intersection}
 \end{figure}

Our goal is to extend the operator $\COLOP_{gain}[\ccdot,\ccdot] : C^\infty_c( \PP M)\times C^\infty_c( \PP M) \rightarrow C^\infty( \overline{\PP} M) $ to conormal distributions over submanifolds 
of $\PP M$, such as those described above. 
The analysis of such an extension requires microlocal analysis on manifolds with boundary, since $\overline{\PP} M$ has a boundary given by the collection of lightlike vectors. 
The analysis of distributions over manifolds with boundary can be involved and technical. \tbl{We are able to avoid difficulties related to manifolds with boundary by introducing an auxiliary vector field and composing it with $\COLOP_{gain}$. This is explained next.} 

Fix an open set $U\subset M$ and a smooth vector field \tjb{$P: U\to \overline{\PP} U$}\f{changed F to P. This change was done throughout this section (and hence the document).}. Let $u,v\in C_c^\infty(\PP M)$, and let $x\in U$. We define the operator $\COLOP_{gain}^P: C_c^\infty(\PP M)\times C_c^\infty(\PP M)\to C^\infty(U)$ as 
\begin{equation}\label{def_of_QF}
\COLOP_{gain}^P[u,v](x)=\COLOP_{gain}[u,v](P(x)).
\end{equation}
\tbl{We will see in this section that we are able to analyze $\COLOP_{gain}^{P}$ operator by using standard techniques such as those in~\cite{duistermaat2010fourier}. }\tjb{The analysis of $\COLOP^{P}_{gain}$ presented in this section will be used to study the singular structure of solutions $\Phi_{L^{+}V}f$ to \eqref{boltzmann} for given sources $f$ which are constructed in Section \ref{main-proof}. We note now that in Section \ref{main-proof},  we we will choose a specific $P$, which we will denote by $P_{e}$.}
%

We first record a couple of auxiliary lemmas. \tbl{The first lemma considers the conormal bundle of a submanifold of $U\times \PP U\times \PP U$. The points of $U\times \PP U\times \PP U$ are denoted by 
$(x,y,z,p,q) = (x,(y,p), (z,q))$.} 
\begin{lemma}\label{Lambda_R-lemma}
Let $Y_1$ and $Y_2$ and $U$ be as in Definition~\ref{intersection_coords} and adopt also the associated notation.
Define
\[
\Lambda_R =   N^* \Big( \bigcup_{x\in U}  \{ x\} \times  \PP_xU \times \PP_xU  \Big).
\]
The submanifold $\Lambda_R$ of $T^*( U \times \PP U \times \PP U) \setminus \{0 \}$ equals the set
\begin{align*}
\Lambda_R = \{ 
\big(x, y, z,p, q \, ; \, \xi^x, \xi^y, \xi^z,\xi^p,\xi^{q} \big) \in \  &T^*(U \times \PP U \times  \PP U) \setminus \{0\}  :  \\
&  \xi^x + \xi^y +\xi^z = 0 , \ \xi^p = \xi^{q} = 0, \ x=y=z  \}.
\end{align*}

\tbl{We have
\begin{align}\label{lambda_R_prime}
\Lambda_R' &
  =\{ \big((x ; \xi^x), (y, z,p, q \, ; \, \xi^y, \xi^z,\xi^p,\xi^{q} )\big) \in \ T^*U \times T^* ( \PP U \times  \PP U) \setminus \{0\}   : \nonumber  \\
 &  \quad\quad\quad\quad\quad\quad\quad\quad\quad\quad\quad\quad\quad \xi^x =\xi^y+ \xi^z \neq 0, \
   \xi^p =  \xi^{q} =0 , \ 
    x=y=z \}.
\end{align}
}
%
The spaces $\Lambda'_R \times ( N^*[Y_1 \times  Y_2] ) $ and $T^*U \times \text{diag}\, T^*( \PP M \times \PP M) $ intersect transversally in $T^*U \times T^*(\PP M \times \PP M) \times  T^*(\PP M \times \PP M)$.

\end{lemma}
Please see Appendix \ref{appendixa1} for a proof. 
\medskip


%
%

For the next theorem let $Y_1,Y_2 \subset \PP M$ be smooth manifolds which satisfy the properties in Definition \ref{intersection_coords}. Fix $x_0 \in \pi(Y_1) \cap \pi(Y_2) \cap X$ (below we choose $X= \pi (\text{supp}A) $) and let $\Lambda_j$, $j=1,2$, be the Lagrangian manifolds specified in the definition. In the canonical coordinates $(x,\xi) = (x',x'', \xi', \xi'')$ on $T^*U$ we obtain the expressions
\begin{align}
\Lambda_0 = \{ (x,\xi) : x' = 0, \ x'' = 0 \}, \label{ossdad1} \\
\Lambda_1 = \{ (x,\xi) : x' = 0, \ \xi'' = 0 \}, \\
\Lambda_2 = \{ (x,\xi) : x''= 0, \ \xi' = 0 \}\label{ossdad3}.
\end{align}
Thus, the elements of the manifolds $\Lambda_0,\Lambda_1, \Lambda_2$ can be parametrized by the free coordinates $(\xi',\xi'')$, $(x'',\xi')$, and $(x',\xi'')$, respectively.

\antticomm{For the purposes of the theorem below, we reparametrize $TU$ with new coordinates $(\tilde{x}',\tilde{x}'',\tilde{p}',\tilde{p}'')$ so that locally $Y_1= \{ \tilde{x}' = 0, \ \tilde{p}'= 0, \ \tilde{p}''= 0 \}$ and $Y_2 = \{ \tilde{x}''= 0,  \ \tilde{p}'= 0, \ \tilde{p}''= 0\}$. That is, we consider the local  reparametrization 
\begin{equation}\label{repara}
(\tilde{x}',\tilde{x}'',\tilde{p}',\tilde{p}'') \mapsto (x'( \tilde{x}' ),x''(\tilde{x}'') ,p'(\tilde{x}',\tilde{p}'),p''(\tilde{x}'',\tilde{p}''))
\end{equation}
given by 
\[
\begin{split}
&x'( \tilde{x}' ) := \tilde{x}' \\
&x''( \tilde{x}'' ) := \tilde{x}'' \\
&p'(\tilde{x}',\tilde{p}') := \theta_2(\tilde{x}')+ \tilde{p}' \\
&p''(\tilde{x}'',\tilde{p}''):=  \theta_1(\tilde{x}'') + \tilde{p}''
\end{split}
\]
and proceed in these coordinates. 
Above, $\theta_1,\theta_2$ are the local fields in Definition \ref{intersection_coords}. With a slight abuse of notation, we redefine coordinates $(x',x'',p',p'')$ as the reparametrization  $(\tilde{x}',\tilde{x}'',\tilde{p}',\tilde{p}'')$. Again, we denote $x= (x',x'')$ and $p=(p',p'')$ etc.  Notice that the $p$-coordinates are not canonically induced by the $x$-coordinates.  
One checks that the expressions of $\Lambda_R$ and $\Lambda_R'$ in Lemma \ref{Lambda_R-lemma} hold also in the canonical coordinates 
of the new parametrization (substituting $(x,p;\xi^x,\xi^p) = (x',x'',p',p''\ ; (\xi^x)' , (\xi^x)'',(\xi^p)',(\xi^p)'')$ and similarly with the other parameters). Since the $x$-coordinates are not changed, the reparametrisation has no effect on the expressions 
(\ref{ossdad1}-\ref{ossdad3}). 
}

\antticomm{Throughout this section we consider the reparametrisation above. } 
In the canonical coordinates \\ $(x,p ; \xi^x,\xi^p) = (x',x'',p',p'' ; (\xi^x)',(\xi^x)'',(\xi^p)',(\xi^p)'' )$ in $T^*TU$ we have that
\begin{align}
N^*Y_1 = \{ (x,p ; \xi^x,\xi^p) : x'= 0, \ p'=0, \ \antticomm{p''=0}, \ (\xi^x)'' = 0\}, \label{YFEW1}\\
N^*Y_2 = \{ (x,p ; \xi^x,\xi^p) : x''= 0, \ p''=0, \ \antticomm{p'=0}, \ (\xi^x)' = 0\} \label{YFEW2},
\end{align}
so the manifolds $N^* Y_1$ and $N^* Y_2 $ are locally parametrized by the \tbl{coordinates $(x'', ( \xi^x)', \xi^p)$ and $(x', ( \xi^x)'', \xi^p)$ respectively. 
The identity \eqref{symboiu} in the next theorem is written in terms of the coordinates $(x'', ( \xi^x)', \xi^p)$, $(x', ( \xi^x)'', \xi^p)$ and $((\xi^x)',(\xi^x)'')$ for the symbols $\sigma(f_1)$, $\sigma(f_2)$ and $\sigma ( \COLOP_{gain}^P[f_1,f_2])$ respectively. }

We extend $\COLOP_{gain}^P$ to conormal distributions as follows. 


\begin{theorem}\label{sofisti}   
Let $(M,g)$ be a globally hyperbolic Lorentzian manifold.  
Let $Y_1,\, Y_2 \subset \PP M$, 
be smooth manifolds that have the admissible intersection \antti{property} (Definition \ref{intersection_coords}) at some $x_0 \in \pi Y_1 \cap \pi Y_2$ and let $\Lambda_l$, $l=0,1,2$, and $U\subset M$ be as in Definition~\ref{intersection_coords}. 
%
Let $P$  be a smooth section of the bundle $\pi: \overline{\mathcal{P}}U \rightarrow U$  
and  $\COLOP$ be the collision operator with an admissible collision kernel. 

Then  the operator $\COLOP_{gain}^P: C_c^\infty (\PP M ) \times C_c^\infty (\PP M ) \rightarrow C^\infty ( U )$ defined in~\eqref{def_of_QF}
extends into a \antticomm{sequentially} continuous  map
\[
 \antticomm{\mathcal{E}'_{N^*Y_1 } (\PP M )  \times \mathcal{E}'_{N^*Y_2 } (\PP M )  }
 \rightarrow \mathcal{D}' (U), 
\]
\antticomm{where $\mathcal{E}_{N^*Y_k}'(\PP M) := \{ f \in  \mathcal{E}' (\PP M) : WF(f) \subset N^*Y_k \}$. (See also Remark \ref{jalkirem} below)}

For $(f_1 , f_2) \in \antticomm{\mathcal{E}'_{N^*Y_1 } (\PP M )  \times \mathcal{E}'_{N^*Y_2 } (\PP M )  }$
we have that 
\[
WF(\COLOP_{gain}^P[f_1,f_2]) \subset \Lambda_0 \cup \Lambda_{1} \cup \Lambda_{2}. 
\]

Moreover, microlocally away from both $\Lambda_{1}$ and $\Lambda_{2} $, we have 
\begin{equation}\label{si3478}
\begin{split}
\COLOP_{gain}^P : I^{m_1}_{comp} ( \PP M ;\, N^*Y_1)\times I^{m_2}_{comp} (\PP M ;\, N^*Y_2)  \to  I^{m_1+m_2+3\nnn/4} (U;\,\Lambda_0 \setminus( \Lambda_{1} \cup \Lambda_{2}) ), \\ 
\end{split}
\end{equation}
together with the symbol 
\begin{equation}\label{symboiu}
\begin{split}
 \sigma ( \COLOP_{gain}^P [f_1,f_2]) \tbl{(\xi',\xi'')} = 
 C \sigma( f_1) (0 ;  \tbl{\xi' , 0  ) \sigma(f_2) ( 0 ;  \xi''} , 0) , 
 \end{split}
\end{equation}
where $C$ is given \antticomm{in terms of the unique vectors $\theta_1(x_0) \in Y_1 \cap T_{x_0} M$ and $\theta_2(x_0) \in Y_2 \cap T_{x_0} M$ by 
\[
C = c\s A\big( \ x_0 \  , \ \hat{p} (x_0) \ , \ \theta_1(x_0)+\theta_2(x_0)- \hat{p}(x_0) \ , \ \theta_1(x_0) \ , \ \theta_2(x_0) \ \big)
\] }
and $\hat{p}(x_0) \in \overline{\PP}_{x_0} U$ is such that $P(x_0) = ( x_0, \hat{p}(x_0)) $. 
Here $c$ is some non-zero constant
\footnote{The constant is coordinate invariant. It depends on geometric quantities, such as the choice of the smooth volume form on $(TM)^4$.} and the manifolds $\Lambda_0 \setminus( \Lambda_{1} \cup \Lambda_{2})$, $T^* TU \cap N^* Y_1$, and $T^* T U \cap N^* Y_2$ are parametrised by the  coordinates \tbl{$(\xi ',\xi'')$}, $(x'', (\xi^x)',\xi^p)$ and $(x' , (\xi^x)'', \xi^p)$ respectively \antticomm{ $($see \eqref{ossdad1} and (\ref{YFEW1}-\ref{YFEW2}) after the reparametrization \eqref{repara}$)$.}
%

\end{theorem}
\antticomm{
\begin{remark}\label{jalkirem}
The sequential continuity above remains true even if one replaces $\mathcal{E}_{N^*Y_k}'(\PP M)$, $k=1,2$ with $\mathcal{E}_{\Gamma_k}'(\PP M)$, where $\Gamma_k \supset N^*Y_k$ is a closed cone near $N^*Y_k$. The proof in that case is essentially the same. Indeed, the existence of the sequentially continuous extension follows from \eqref{intersection_of_wavefronts_empty} by \cite[Corollary 1.3.8]{duistermaat2010fourier}, as deduced for $\Gamma_k =  N^*Y_k $ in the proof below. 
\end{remark}
}


%
%
%
%
%
%
%
%

\begin{proof}[Proof of Theorem \ref{sofisti}]
We write $I^m (\PP M ;\, \Lambda) = I^m(\Lambda)$ for a Lagrangian manifold $\Lambda \subset T^* (\PP M)$ and $m \in \R$. 
To prove the claims of the proposition, 
we first represent $\COLOP_{gain}^P [ f_1, f_2 ]$ as a composition of a Fourier integral operator and a tensor product of $f_1$ and $f_2$. Then we appeal to results for Fourier integral operators in~\cite{duistermaat2010fourier} to conclude the proof. 
The majority of the proof consists of demonstrating that the Fourier integral operators we use to decompose $\COLOP_{gain}^P [ f_1,f_2 ]$ satisfy the conditions required by~\cite[Theorem 2.4.1]{duistermaat2010fourier} and~\cite[Corollary 1.3.8]{duistermaat2010fourier}. 


Let \antticomm{$(x',x'',p',p'') := (\tilde{x}',\tilde{x}'',\tilde{p}',\tilde{p}'')$ be the local coordinates \eqref{repara}} in $TU$, where $U$ is an open neighbourhood of $x_0$, as described in Definition~\ref{intersection_coords}. 
 Let us define an 
 integral operator 
\[
R : C^\infty_c  ( \PP M  \times \PP M  ) \rightarrow \mathcal{D}' ( U  )
\]
by the formula
\[
\begin{split}
\langle R [\phi] , \psi \rangle =&  \int_{  \PP U \oplus \PP U } \phi ( x, p , x , q) \psi(x) \, dv (x,p,q) \\
\end{split}
\]
where $dv$ is \tbl{the induced volume form on the direct sum bundle $\PP U \oplus \PP U := \{ \big((x,p), (y,q) \big) \in \PP U \times \PP U : x = y \}$ and where $\phi\in C^\infty_c  ( \PP M  \times \PP M  )$ and $\psi\in C_c^\infty(U)$. The induced volume form $dv$ is given by considering $\PP U \oplus \PP U$ as a submanifold of $\PP U\times \PP U$ equipped with the product volume form.} 

%
Let $P$ be a smooth local section of the bundle $\pi: \overline{\mathcal{P}}U \rightarrow U$. We represent $P$ as $P(x)=(x,\hat{p}(x))$.
Next, consider the operator	
\[
 \langle R_A [\phi] , \psi \rangle = \int_{  \PP U \oplus \PP U } A(x,\hat{p}(x),\, p+q - \hat{p}(x) ,\, p,\,q)\phi ( x, p , x , q ) \psi(x) \, dv (x,p,q) 
\]
%
%
Notice that
\begin{equation}\label{uusi7878782}
R_A[h_1 \otimes h_2]= 
\COLOP_{gain}^P [h_1, h_2] , \quad h_1,h_2 \in C_c^\infty( \PP M ).
\end{equation}
Here $h_1\otimes h_2$ is the tensor product of $h_1$ and $h_2$\tbl{, i.e. $(h_1\otimes h_2)((x,p),(y,q))= h_1(x,p)h_2(y,q)$.}
%
\tbl{The distribution kernel associated to the operator $R $ is
\[
k_R  \in \mathcal{D}' ( U \times  \PP M \times \PP M ). 
\]}
Since \tbl{$k_R$} is a delta distribution over the submanifold 
\[
 \bigcup_{x\in U}  \{ x\} \times  \PP_x M \times \PP_x M \  \subset \ U  \times  \PP M \times \PP M,
\]
it can be viewed as a conormal distribution
\begin{align*}
k_R &\in  I^{-\nnn/4 }(\tbl{U  \times  \PP M \times \PP M \s;  \Lambda_R }), \\
\Lambda_R &:=   N^* \Big( \bigcup_{x\in U}  \{ x\} \times  \PP_x M \times \PP_x M \Big).
\end{align*}
Hence $R$ is a Fourier integral operator 
 of class $I^{-\nnn/4}( U, \PP M \times \PP M , \Lambda'_R)$. Since the collision kernel $A$ is smooth, $R_A$ is also of class $I^{-\nnn/4}( U, \PP M \times \PP M , \Lambda'_R)$. 
By Lemma \ref{Lambda_R-lemma}, 
the Lagrangian manifold $\Lambda_R$ is given by
\begin{align*}
\Lambda_R = \{ 
\big(x, y, z,p, q \, ; \, \xi^x, \xi^y, \xi^z,\xi^p,\xi^{q} \big) \in \  & T^*(U \times \PP U \times  \PP U) \setminus \{0\}  : \\
& \xi^x + \xi^y +\xi^z = 0 , \ \xi^p = \xi^{q} = 0, \ x=y=z  \}.
\end{align*}
%
Further, we have that $WF'(R_A) \subset \Lambda_R'$, where $\Lambda_R'$ equals by its definition~\eqref{relation_defn} the set 
\begin{align}\label{hauska12221}
\Lambda_R' 
  =\{ \big((x ; \xi^x), (y, z,p, q \, ; \, \xi^y, \xi^z,\xi^p,\xi^{q} )\big) & \in \  T^*U \times T^* ( \PP U \times  \PP U) \setminus \{0\}   : \nonumber  \\
 & \xi^x =\xi^y+ \xi^z \neq 0, \
   \xi^p =  \xi^{q} =0 , \ 
    x=y=z \}.
\end{align}

%

Now we show 
that $\COLOP_{gain}^P [\ccdot, \ccdot]$ can be extended to give a \antticomm{sequentially continuous map on  $\mathcal{E}'_{N^*Y_1 } ( \PP M) \times \mathcal{E}'_{N^*Y_2 } ( \PP M)$. In particular, the extension is then defined on} the compactly supported conormal distributions in $I^{m_1}_{comp} (N^* Y_1 )\times I^{m_2}_{comp} (N^* Y_2 )$, $m_1,m_2\in \R$. By \cite[Corollary 1.3.8]{duistermaat2010fourier}, it is sufficient to demonstrate that 
\begin{equation}\label{intersection_of_wavefronts_empty}
WF_{\PP M \times \PP M}'(R_A)\cap WF(f_1 \otimes f_2)=\emptyset.
\end{equation}
%
%
for \antticomm{$(f_1,f_2)\in \mathcal{E}'_{N^*Y_1 } ( \PP M) \times \mathcal{E}'_{N^*Y_2 } ( \PP M)$. }
%
%
%
Since the wavefront set of $f_j \in  \antticomm{  \mathcal{E}_{N^*Y_j} ' ( \PP M)}$, 
$j=1,2$ is contained in $N^* Y_j$, we have by \cite[Proposition 1.3.5]{duistermaat2010fourier} that the wavefront set of the tensor product $f_1 \otimes f_2$ satisfies 
\[
WF(f_1 \otimes f_2) \subset (N^* Y_1 \times \{0\}_{\PP M }) \cup (\{0\}_{\PP M} \times N^*Y_2) \cup (N^*Y_1 \times N^*Y_2 ),
\]
where $\{0\}_{\PP M} \subset T^* (\PP M)$ is the zero bundle over $\PP M$. 
 By using the fact that $WF(k_R)\subset \Lambda_R$ and the equation~\eqref{hauska12221}, the set $WF_{\PP M \times \PP M}'(R) \subset  T^* ( \PP U \times  \PP U) \setminus \{0 \} $, defined in~\eqref{wf_set_fio_domain}, satisfies
\begin{align}\label{restricted_WF_calc}
 WF_{\PP M \times \PP M}'(R) &= \{ (y,z, p, q \; \xi^y, \xi^z,\xi^p,\xi^{q})\in T^*(\PP M\times \PP M)\setminus \{0\}:  \nonumber \\ 
&\qquad\quad \text{ there is }
x \in  U
\text{ such that } \big((x\;0), (y,z, p, q \; \xi^y, \xi^z,\xi^p,\xi^{q}) \big) \in WF(k_R) \} \nonumber
\\
& \subset \{ (x,x,p, q \; \xi^y, -\xi^y,0,0)\in T^*(\PP M\times \PP M)\setminus \{0\} :  p,q \in \PP_x U, \  \xi^y \in \R^\nnn \setminus \{0 \} \}.
\end{align}
Since $Y_1$ and $Y_2$ satisfy Definition \ref{intersection_coords}, pairs of  \antticomm{$\xi^x$}-elements  in the fibers of $N^*Y_1$ and $N^*Y_2$ are linearly independent. They are also non-zero by the definition of a normal bundle. Thus an element of $N^*Y_1\times N^*Y_2$ cannot be of the form $(x,x, p, q \; \xi^y, -\xi^y,0,0)$.
 We deduce that $WF_{\PP M \times \PP M }'(R) \cap (N^*Y_1 \times N^*Y_2 )= \emptyset$. By a similar consideration, we see that the set in~\eqref{restricted_WF_calc} does not intersect $(N^* Y_1 \times \{0\}_{\PP M}) \cup (\{0\}_{\PP M} \times N^*Y_2 )$. In particular, we have~\eqref{intersection_of_wavefronts_empty}. 
 The  set $WF_{U}'(R) \subset  T^* U \setminus \{0 \}$, defined in~\eqref{wf_set_fio_image}, satisfies
\begin{align*}
 WF_{U}'(R) &= \{ (x,\xi^x) \in T^*U: \,\text{ there is } \\ 
&\qquad\qquad
(y,p,z, q) \in  \PP U\times \PP U
\text{ such that } \big((x\; \xi^x), (y,z, p, q\, ;\, 0, 0,0,0) \big) \in WF(k_R) \}.
  \end{align*}
  By using the fact that $WF(k_R)\subset \Lambda_R$ and~\eqref{hauska12221}, we see that if $\big((x\; \xi^x), (y,z, p, q\, ;\, 0, 0,0,0) \big) \in WF(k_R)$, then $\xi^x=0$. 
Thus $WF_{U}'(R)= \emptyset$.
By \cite[Corollary 1.3.8]{duistermaat2010fourier} the composition $R_A \circ( f_1 \otimes f_2)$ is well defined and we obtain the desired \antticomm{sequential} continuous extension for $\COLOP_{gain}^P [ f_1,f_2 ] \in \mathcal{D}'(U)$.
The second claim of the proposition follows directly from \cite[Corollary 1.3.8]{duistermaat2010fourier} and the facts $WF_{U}'(R_A) \subset WF_U'(R)= \emptyset$ and $WF(R_A)\subset \Lambda_R'$: 
\begin{equation}\label{wave_front_calc_in_prop4}
WF ( \COLOP_{gain}^P [ f_1,f_2 ] ) \subset \Lambda_R' \circ  \Big( (N^* Y_1 \times \{0\}_{\PP M}) \cup (\{0\}_{\PP M} \times N^*Y_2) \cup (N^*Y_1\times N^*Y_2) \Big).
\end{equation}

By using the coordinate description \eqref{hauska12221}, we show next that
\begin{equation}\label{lambda_composed_with_Y1}
 \Lambda'_R \circ (  N^*Y_1 \times \{0\}_{\tbl{ \PP M}}) = \Lambda_{1}, \quad \Lambda'_R \circ ( \{0\}_{\tbl{\PP M}} \times N^*Y_2) = \Lambda_{2}.
\end{equation}
 By definition (see \eqref{relation_of_composition}), we have that  
\begin{align}\label{lambda_composed_with_Y2}
\Lambda'_R \circ (  N^*Y_1 \times \{0\}_{\tbl{ \PP M }})&=\{ (x,\xi^x)\in T^*U: \text{ such that } \big((x,\xi^x), (y,p \; \xi^y, \xi^p), (z,q \; \xi^z, \xi^q)\big)\in \Lambda_R' \nonumber \\
& \text{ where } (y,p \; \xi^y, \xi^p)\in N^*Y_1 \text{ and } \xi^z=\xi^q=0 \}.
\end{align}
By using~\eqref{hauska12221}, it follows that in the expression above, we must have that $(x,\xi^x)$ is any element of the form $(y,\xi^y)$, where $(y,p\; \xi^y, 0)\in N^*Y_1$ for some $p$.~\f{The condition $(\xi^x)''=0$ was indeed missing. Removed.}  Thus we have $\Lambda'_R \circ (  N^*Y_1 \times \{0\}_{\tbl{\PP M}}) = \Lambda_{1}$. We similarly have $\Lambda'_R \circ ( \{0\}_{\tbl{\PP M}} \times N^*Y_2) = \Lambda_{2}$. We have proven~\eqref{lambda_composed_with_Y1}.	

By using the coordinates $(x,p)=(x',x'',p',p'')$ we may also write any $(x,\xi)\in T^*U$ as $(x,\xi)=(x',x'',\xi',\xi'')$. 
We have that  
\begin{align*}
\Lambda_{1} &= \{ (x\;\xi)\in T^*U : x'=0, \ \xi''=0 \}, \\ 
\Lambda_{2} &= \{ (x\;\xi)\in T^*U : x''= 0, \ \xi'=0 \}.
\end{align*}
We calculate similarly as in~\eqref{lambda_composed_with_Y2} 
\begin{align}\label{lambda_composed_with_Y3}
\Lambda'_R \circ (  N^*Y_1 \times N^*Y_2)&=\{ (x,\xi^y+\xi^z)\in T^*U: \text{ such that } \big((x,\xi^x), (x,p \; \xi^y, 0), (x,q \; \xi^z, 0)\big)\in \Lambda_R' \nonumber \\
& \quad\quad\quad \text{ where } (x,p \; \xi^y, 0)\in N^*Y_1 \text{ and } (x,q \; \xi^z, 0)\in N^*Y_2 \} \nonumber \\
 &= \{ (x', x'' \; \xi', \xi'')\in T^*U \setminus \{0 \} :  x'= x'' = 0 \} = \Lambda_0.
\end{align}
Here we used again~\eqref{hauska12221}.
By combining~\eqref{wave_front_calc_in_prop4},~\eqref{lambda_composed_with_Y2} and~\eqref{lambda_composed_with_Y3}, we have shown that 
\[
WF ( \COLOP_{gain}^P [ f_1,f_2 ]) \subset  \Lambda_0 \cup \Lambda_{1} \cup \Lambda_{2}. 
\]

We are left to show the last claims of the theorem. Fix arbitrary conic neighbourhoods $\Gamma_1,\Gamma_2$ of $\Lambda_1$ and $\Lambda_2 $, respectively. 
Let $\epsilon>0$ be small enough to satisfy
\begin{equation}\label{zskla}
   \Big\{ (x,\xi) \in T U :  \| x' \| < \epsilon , \    \|\xi''\| < \epsilon \|\xi' \| \Big\} \subset \Gamma_1 
\end{equation}
and
\[
\Big\{ (x,\xi) \in T U :   \| x'' \| < \epsilon , \|\xi' \|  < \epsilon \| \xi'' \|  \Big\} \subset \Gamma_2 .
\]
By multiplying the amplitude in the oscillatory integral representation of $f_1 \otimes f_2$ by
\[
1 = \phi + (1-\phi) , 
\]
where $\phi \in C^\infty( T^*( \PP M \times \PP M) \setminus \{0\} )$ is positively homogeneous of degree $0$ and equals $1$ near $N^*[Y_1 \times \PP M] = N^* Y_1 \times \{ 0 \}_{\PP M}$ and $0$ near $N^* [\PP M \times Y_2] = \{ 0 \}_{\PP M} \times N^* Y_2$,~\f{To Matti's question, I don't know.}     
we write 
\[
f_1 \otimes f_2  \equiv v_1+ v_2 
\]
where
\begin{align}
&v_1    \in I^{m_1+m_2 + \nnn, - m_2 - \nnn}_{comp} \big( N^* Y_1  \times \{0\}_{\PP M},   N^* [Y_1  \times  Y_2]  \big), \\
 &v_2 \in I^{m_1 + m_2+ \nnn,-m_1-\nnn}_{comp} \big( \{0\}_{\PP M} \times N^*Y_2  , N^* [Y_1  \times  Y_2]  \big), 
\end{align}
To prove \eqref{si3478} it is sufficient to show that for both $j=1,2$ there is a decomposition $R_A \circ v_j = u_j + r_j$, where $u_j$ is a Lagrangian distribution over $ \Lambda_0$ and $WF(r_j) \subset \Gamma_j$. 
We only consider the $v_1$ component of $f_1\otimes f_2$ and write $v=v_1$. The argument for the other component is similar. 
Fix small $\delta_1,\delta_2 \in (0,1)$ with $\frac{2\delta_1}{1-\delta_2}< \epsilon$, and 
let $\psi \in C^\infty (T^* ( \PP M \times \PP M) \setminus \{0\})$ be positively homogeneous of degree 0 such that it equals $1$ on the conic neighbourhood
\[ 
\begin{split}
X_1:= \Big\{ (x,p,y,q; \xi^x,\xi^p,\xi^y,\xi^q) \in T^* ( \PP U \times \PP U) \setminus \{0\}   :  \| (\xi^x)'' \| < \frac{\delta_1}{2} \| (\xi^x)' \| , \\
 \| x' \| < \frac{\epsilon}{2} , \  \| (\xi^y)'' \| < \frac{\delta_1}{2} \| (\xi^x)' \|,  \ 
\| (\xi^y)' \| < \frac{\delta_2}{2} \| (\xi^x)' \| \Big\} 
\end{split}
\] of $N^*Y_1 \times \{0 \}_{\PP M}$ 
and vanishes in the exterior of the larger neighbourhood
\[
\begin{split}
X_2:= \Big\{ (x,p,y,q; \xi^x,\xi^p,\xi^y,\xi^q) \in T^* ( \PP U \times \PP U) \setminus \{0\}   :  \| (\xi^x)'' \| < \delta_1 \| (\xi^x)' \| , \\
 \| x' \| < \epsilon , \  \| (\xi^y)'' \| < \delta_1 \| (\xi^x)' \|, 
\| (\xi^y)' \| < \delta_2 \| (\xi^x)' \| \Big\} .
\end{split}
\]
By dividing the amplitude in the oscillatory integral according to $1 = \psi + (1-\psi)$ we obtain the decomposition $v = v' + v''$, where $v'  \in I_{comp}^{m_1+m_2 + n} ( N^*[Y_1 \times Y_2])$ and $WF(v'') \subset X_2$. Moreover, 
\[
\sigma(v') = (1-\psi) \sigma(v) = \sigma(v)= c \phi (\sigma(f_1) \otimes \sigma(f_2) )\quad \text{on} \quad  (N^*[Y_1 \times Y_2]) \setminus X_2. 
\]
Applying \cite[Corollary 1.3.8]{duistermaat2010fourier} and $WF(v'') \subset X_2$ we deduce
\[
\begin{split}
WF( R_A \circ v'') \subset    WF'(R_A) \circ WF(v'') \cup WF_U'(R_A) \subset  ( \pi(\text{supp} A) ) \cap   ( \Lambda_R' \circ WF(v'')) \\ 
 \subset ( \pi(\text{supp} A) )  \cap   ( \Lambda_R' \circ X_2) 
\end{split}
\]
By definition, an element $(x,\eta)$ in $( \Lambda_R' \circ X_2)$ satisfies  $\| x' \| < \epsilon$ and 
\[
\eta = \xi^x+ \xi^y, \quad  \| (\xi^x)'' \| < \delta_1 \| (\xi^x)' \| , 
 \  \| (\xi^y)'' \| < \delta_1 \| (\xi^x)' \|,  \ 
\| (\xi^y)' \| < \delta_2 \| (\xi^x)' \|. 
\]
By using the triangle inequality and the inequalities above one computes
\[
\frac{ \| \eta'' \| }{\| \eta' \| } \leq \frac{ \| (\xi^x)'' \| + \| (\xi^y)'' \|  }{| \| (\xi^x)' \| - \| (\xi^y)'\| |  } < \frac{2 \delta_1 }{ 1-\delta_2 } < \epsilon
\]
which implies $ \| \eta'' \| < \epsilon \| \eta' \|$. 
Thus, by \eqref{zskla}, we conclude
\[
WF( R_A \circ v'') \subset  \Gamma_1
\]
To finish the proof we are left to show that the conditions of \cite[Theorem 2.4.1]{duistermaat2010fourier} (cf. the global formulation  \cite[Theorem 4.2.2]{duistermaat2010fourier}) for the composition of $v'$ and $R_A$ are satisfied. 
The symbol identity \eqref{symboiu} follows from \cite[eq. (4.2.10)]{duistermaat2010fourier}. 
The condition \cite[eq. (2.4.8)]{duistermaat2010fourier}  is satisfied since $v'$ is compactly supported. The  conditions \cite[eq. (2.4.9), (2.4.10)]{duistermaat2010fourier} follow from the definitions (\ref{turhautuminen1}) and (\ref{hauska12221}). The last condition  \cite[eq. (2.4.11)]{duistermaat2010fourier} follows from Lemma \ref{lambda_R_prime}.
Since the conditions of \cite[Theorem 2.4.1]{duistermaat2010fourier} are met, we have that $R_A \circ v'$ is a well-defined oscillatory integral of order $m_1+m_2 + 3\nnn/4$ with the canonical relation $\Lambda'_R \circ (  N^*(Y_1 \times Y_2)) = \Lambda_0$. 
In conclusion, for arbitrary conic neighbourhoods $\Gamma_1,\Gamma_2$ of $\Lambda_1$ and $\Lambda_2$ there is the decomposition
\[
 \COLOP_{gain}^P [ f_1,f_2 ] = u + r, \quad u := R_A \circ (v_1'+v_2') \in I_{comp}^{m_1+ m_2 +3\nnn/4}( \Lambda_0) , \quad  r:=  R_A \circ (v_1''+v_2'') ,
\]
such that $WF(r) \subset \Gamma_1 \cup \Gamma_2$ and the symbol identity  \eqref{symboiu} holds on $\Lambda_0 \setminus (\Gamma_1 \cup \Gamma_2)$. 
\end{proof}

%
%
%
%
%
%
%
%

Recall that $K_S \subset \mathcal{S}^+ M$ is the geodesic flowout of $S \subset \mathcal{P}^+ M$ in $\mathcal{P}^+ M$, that is, the union of the inextendible geodesic velocity curves $(\gamma_{(x,p)},\dot\gamma_{(x,p)})$ over $(x,p) \in S$. 
%
%
%
%
We next use Theorem \ref{sofisti} together with the microlocal properties of the geodesic vector field to show that $\mathcal{Q}^P_{gain}[\cdot,\cdot]$ can be extended to a sequentially continuous operator over $I^{l_1}_{comp} ( \PP M ;\, N^*S_1) \times I^{l_2}_{comp} (\PP M ;\, N^*S_2)$ whenever $l_1,l_2$ are integers, $\mathcal{C}$ is a Cauchy surface in $M$, and $S_1,S_2\subset \mathcal{P}^+ \mathcal{C}$ are two submanifolds whose geodesic flowouts satisfy the admissible intersection property. We remark that in the inverse problem we consider, we only require the existence of such manifolds $S_1, S_2 \subset\mathcal{P}^+(\mathcal{C})$. The existence is proven later in Corollary 4.8. We do not require nor give an algorithmic characterization of how to build the manifolds $S_1$ and $S_2$.


\begin{corollary}[Extension to distributions solving Vlasov's equation]\label{sofisti_further}
%

\tbl{Let $(M, g)$ be a globally hyperbolic Lorentzian manifold and let $\mathcal{C}$ be a Cauchy surface of $(M,g)$. 
Let $S_1, S_2 \subset \PP^+ \mathcal{C} $} be smooth manifolds such that 
the geodesic flowouts $Y_1 := K_{S_1}$ and $Y_2:= K_{S_2}$ have an admissible intersection \antti{property} (see Definition \ref{intersection_coords}) in $\pi(\text{supp} A)$ for admissible $A$. 

\antticomm{Assume that  there is some $x_0 \in \pi (Y_1) \cap \pi (Y_2) \cap \pi(\text{supp} A)$ (cf. Remark \ref{remarko} below)} and let $U$ be a small neighbourhood as in the definition of admissible intersection, let $P : U \rightarrow \OVS U$ be a smooth section of $\OVS U$ and assume additionally that $S_1,S_2 \subset \PP M \setminus \PP \overline{U} $ (e.g. $ \overline{\pi S_1} \cap   \overline{ \pi S_2} = \emptyset$ and small $U$). 
Additionally, for $f_j \in   I_{comp}^{l_j}(\PP M ;\,N^*S_j)$, $j=1,2$,
 let $ u( f_j) $ solve the Vlasov's equation with source $f_j$ \tbl{and which vanish in $\mathcal{C}^-$.}

 Then, the operator \tbl{$\COLOP_{gain}^P [ u(\ccdot) , u(\ccdot) ]$, defined in~\eqref{def_of_QF},}
 defines a  \antticomm{sequentially continuous}  map
 \[
 \begin{split}
I^{l_1}_{comp} ( \PP M ;\, N^*S_1) \times I^{l_2}_{comp} (\PP M ;\, N^*S_2)\rightarrow \mathcal{D}' (U) 
\end{split}
\]

Moreover, microlocally away from both $\Lambda_{1}$ and $\Lambda_{2} $, we have that 
\begin{equation}
\COLOP_{gain}^P [u(f_1),u(f_2)]\in I^{l_1+l_2+3\nnn/4-1/2} (U;\,\Lambda_0 \setminus( \Lambda_{1} \cup \Lambda_{2}))
\end{equation}
together with the symbol 
\begin{equation}\label{symbol_in_the_corollary}
\begin{split}
 \sigma ( \COLOP_{gain}^P [u(f_1),u(f_2)]) (\xi',\xi'') = 
 C \sigma( u(f_1) ) (0;\xi',0) \sigma(u(f_2)) ( 0 ;  \xi'' , 0) , 
 \end{split}
\end{equation}
where the constant $C$ is given 
\antticomm{in terms of the unique vectors $\theta_1(x_0) \in Y_1 \cap T_{x_0} M$ and $\theta_2(x_0) \in Y_2 \cap T_{x_0} M$ by 
\[
C = c\s A\big( \ x_0 \  , \ \hat{p} (x_0) \ , \ \theta_1(x_0)+\theta_2(x_0)- \hat{p}(x_0) \ , \ \theta_1(x_0) \ , \ \theta_2(x_0) \ \big)
\] }
Here $c$ is some non-zero constant and the manifolds $\Lambda_0 \setminus( \Lambda_{1} \cup \Lambda_{2})$, $T^*TU \cap N^* Y_1$, and $T^* TU \cap N^* Y_2$ are parametrised by the canonical coordinates $(\xi',\xi'')$, $(x''; (\xi^x)',\xi^p  )$, and $(x' ; (\xi^x)'', \xi^p)$, respectively \antticomm{ $($see \eqref{ossdad1} and (\ref{YFEW1}-\ref{YFEW2}) and the reparametrization \eqref{repara}$)$.}
\end{corollary}




\begin{proof}


For $j=1,2$, let $f_j \in I^{l_j}_{comp} ( N^* S_j ; \PP M )$. Choose a cut-off function $\chi \in C^\infty_c (M) $ so that $\chi =1 $ on $U \cap \text{supp} (x \mapsto A(x,\ccdot))$ and \tbl{$\chi=0$} on a neighbourhood of $\pi(S_1 \cup S_2) \subset M \setminus \overline{U}$ 

Now, for each $j=1,2$, by Lemma \ref{vlasov_ext_conormal}, there exists a solution $u(f_j)$ to the Vlasov equation with source $f_j$ and initial data $0$. Moreover, as $Y_j = K_{S_j}$, \antti{the sources $f_j$ are compactly supported and time-like geodesics can not be trapped} we get $\chi u(f_j) \in I_{comp}^{l_j -1/4} (\PP M ;\,N^*Y_j )$.


Substituting $u = \chi u  + (1-\chi) u$ into $\COLOP_{gain}$, we obtain that 
 \[
 \begin{split}
 \COLOP_{gain}^P [ u(f_1), u(f_2)]    \equiv & \COLOP_{gain}^P [\chi  u(f_1), \chi  u(f_2)].  
 \end{split}
 \]
Thus we have reduced to the setting of Proposition \ref{sofisti} and obtain the desired results.

\end{proof}

\antticomm{
%

\begin{remark}\label{remarko}
Corollary \ref{sofisti_further} above does not include the case where the base manifold flowouts $G_{S_j} = \pi (Y_j)$, $j=1,2$ do not meet in $\pi ( \text{supp}A )$. This corresponds to a setting where no interactions of singularities take place. Such situations can be reduced to the trivial case \[\COLOP_{gain} [ u(f_1), u(f_2)]= 0 \] via  localization of the sources. Indeed, the techniques used in this article allow us to localize the support of each source $f_j \in I_{comp}^m (N^*S_j)$, $j=1,2$ arbitrarily close to a given vector. This implies that the projected support 
\[
\pi \ \text{supp}( u(f_j)) = \pi \text{supp} \Big( (x,p) \mapsto \int_{-\infty}^0 f_j( \gamma_{x,p}(s) , \dot\gamma_{x,p}(s) ) ds \Big) 
\]
(the integral in the sense of distributions) 
intersects the compact set $\pi ( \text{supp}A )$ only arbitrarily near  the single geodesic $ \gamma_j $ (note that $(M,g)$ is causally disprisoning) through the point at which $f_j$ was localized to. Thus, if these geodesics $\gamma_1$ and $\gamma_2$ do not intersect in $\pi ( \text{supp}A )$, the localization of the sources implies that $\pi ( \text{supp}A ) \cap \pi (\text{supp} \  u(f_1)) \cap \pi ( \text{supp} \ u(f_2)) = \emptyset $ and hence $\COLOP_{gain} [ u(f_1), u(f_2)]  $ vanishes at every point. 

\end{remark}
}

\subsection{Existence of transversal collisions}\label{travserse-collisions-section}


In this section, we prove the existence of submanifolds $S_1,S_2\subset TM$ whose geodesic flowouts $K_{S_1}$ and $K_{S_2}$ satisfy the admissible intersection property (see Definition~\ref{intersection_coords}). 

In the proof of our main theorem, Theorem~\ref{themain}, we will construct particle sources in a common open set $V$ of two manifolds $M=M_l$,  $l=1,2$, such that they send information into the unknown region $W_l \subset M_l$ to create point singularities produced by using the nonlinearity. We then use the source-to-solution map to study the propagation of that singularity.

To construct the singularity, the idea is that for two time-like future pointing vectors $(y,q)$ and $(x,p)$ with distinct base-points in $V$ we build a manifold $S_1\subset \OVS V$ around $(y,q)$ such that the geodesic flowouts $Y_1= Y_{1,M_l,g_l} := K_{S;M_l,g_l} $ and $Y_2=Y_{2,M_l,g_l} :=K_{\{(x,p)\} ; M_l,g_l}$ will satisfy the admissible intersection property (see Definition~\ref{intersection_coords} and Figure~\ref{fig:intersection}) in $ W_l \subset \pi (\text{supp}A_l)$ in both manifolds $M_l$ simultaneously. 
Here $A_l$ is an admissible collision kernel associated to the Boltzmann equation in $(M_l,g_l)$. 
The Corollary \ref{lemma_renerew} below states that such a manifold $S_1$ exists.  
Consequently, 
Corollary \ref{sofisti_further} will be applicable in both manifolds for sources conormal to $S_1$ and $S_2 = \{(x,p)\}$.

Working independently on a single manifold is unfortunately not sufficient. Indeed, fixing $S_1,S_2$ in the common set $V$ such that the intersection property holds for the flowouts $Y_j =  K_{S_j}$, $j=1,2$ in one spacetime, say in $(M_1,g_1)$, does not in general imply that the property holds for the analogous flowouts in $(M_2,g_2)$.




\begin{lemma}\label{lemma_renerew}
Let $(M_l,g_l)$, $l=1,2$ be two globally hyperbolic manifolds containing open $V_l \subset M_l$ and assume that there is a diffeomorphism $\Psi : V_1 \rightarrow V_2$ such that $g_1|_{V_1} = \Psi^* g_2$. 
For $l=1,2$ consider $(x_l,p_l) \in \mathcal{P}^+ V_l$ with $(x_2,p_2) = D\Psi (x_1,p_1)$. Let $\{ (z_{l,j},v_{l,j}) : j \in J_l \}$ $($possibly $J_l  = \emptyset$ $)$ be a countable  family of vectors in the set
\[
\antticomm{\mathcal{P}^+_{ \gamma_{(x,p)}} M_l \ \setminus \ \R \dot\gamma_{(x_l,p_l)}  := }\big\{ (z,v) \in  \mathcal{P}^+M_l  \ \big| \  \exists t \ : \ z= \gamma_{(x_l,p_l)} (t) , \   v \in   \mathcal{P}^+_z M_l  \setminus \text{span}\{ \dot\gamma_{(x_l,p_l)}(t) \} \big\} , 
\]
\antticomm{i.e. time-like future-pointing vectors on $\gamma_{(x_l,p_l)}$ that are not tangent to the curve}.  
Let $\mathcal{C}$ be a space-like Cauchy surface through $x_l$ in one of the manifolds $M_l$, $l=1,2$ and copy its restriction $\mathcal{C}_l := \mathcal{C} \cap V_l$ to the other by setting $\mathcal{C}_2 = \Psi \mathcal{C}_1$. 
Then there exist $\antticomm{(n-2)}$-dimensional submanifolds $S_l \subset \mathcal{P}^+  \mathcal{C}_l$, $l=1,2$  
such that $(x_l,p_l) \in S_l$, $S_2 = D\Psi S_1$ and the following two conditions hold for every $j\in J_l$:
\begin{enumerate}[(i)]
\item The point $z_{l,j}$ has an open neighbourhood $U_{l,j} \subset M_l$ such that the intersection $U_{l,j} \cap G_{S_l}$, where $G_{S_l} := \pi K_{S_l}$, is a $(n-1)$-dimensional submanifold of $M_l$. 
\item
$ v_{l,j} \notin  T_{z_{l,j}}G_{S_l}  $.
\end{enumerate}
\end{lemma}


\begin{corollary}\label{coro_renerew}
Let $(M_l,g_l)$, $l=1,2$ be globally hyperbolic manifolds with a mutual open set $ V  \subset M_l$,
$l=1,2$ and assume that $g_1|_V = g_2|_V$. Consider $(x,p),(y,q)  \in \mathcal{P}^+ V$ with distinct base points $x\neq y$.  Assume that the geodesics $\gamma_l$ and $\tilde\gamma_l$ in $M_l$, defined by 
\begin{align}
&  (\gamma_2(0) ,\dot\gamma_2(0) ) =   (\gamma_1(0) ,\dot\gamma_1(0) )=(x,p) , \\
&(\tilde\gamma_2(0) ,\dot{\tilde\gamma}_2(0) ) = (\tilde\gamma_1(0) ,\dot{\tilde\gamma}_1(0) ) =(y,q) , 
\end{align} 
are distinguishable as paths (on their maximal domains), that is, $(x,p)$ and $(y,q)$ are not tangent to the same geodesic. Let $\mathcal{C}$ be a space-like Cauchy surface through $x$ $($resp. $y)$ in one of the manifolds $M_l$ and let $X_l \subset M_l$ be compact. $($e.g. $X_l = \pi(\text{supp} A_l)$ for admissible collision kernels $A_l$,   $l=1,2)$
Then, there is a $\antticomm{(n-2)}$-dimensional
submanifold $S \subset \mathcal{P}^+ ( \mathcal{C} \cap V)$ containing $(x,p) $ $(\text{resp. }(y,q) )$ such that 
the pair consisting of the flowouts $K_S = K_{S;M_l,g_l}$ and $K_{(y,q)}= K_{\{ (y,q)\};M_l,g_l}$ $($resp.  $K_{(x,p)} = K_{\{(x,p)\};M_l,g_l})$ have admissible intersection property in $X_l$ for both $l=1,2$. 
 \end{corollary}

 \begin{proof}[Proof of Lemma \ref{lemma_renerew}]
 Since $M_l$ is globally hyperbolic, the space $\mathcal{P}^+M_l$  can locally near the curve $(\gamma_{x_l,p_l}, \dot\gamma_{x_l,p_l})$ be written  as the product $\mathcal{P}^+\mathcal{C}_l   \times \R$ by identifying $(x,p,t)  $ with $ ( \gamma_{(x,p)}(t), \dot\gamma_{(x,p) }(t) )$. 
 Denote by $\phi_l $ be the projection from the neighbourhood of $(\gamma_{x_l,p_l}, \dot\gamma_{x_l,p_l})$ to $ \mathcal{P}^+\mathcal{C}_l$ that in terms of the identification above equals the cartesian projection $(x,p,t) \mapsto (x,p) $. 
 That is; $\phi_l$ takes each $(x,p)$ in the neighbourhood into the unique intersection of $(\gamma_{(x,p)}, \dot\gamma_{(x,p)})$ and $\mathcal{P}^+\mathcal{C}_l $. Let 
 $\phi_l(x,p)=(z_{l,j},p_{l,j})\in \mathcal{P}^+\mathcal{C}_l$, 
 that is,
 $p_{l,j}$ stand for the velocity $\dot\gamma_{x_l,p_l}$ at $z_{l,j}$. We define 
  $L_{j,l} \subset T_{z_{l,j} } M_l $ to be the 2-plane spanned by $ p_{l,j} $ and $ v_{l,j}$.  Let $\pi :  \mathcal{P}^+M \to M$ be the canonical projection. 
We see that the linear space
 \[
 E_{l,j} :=  D_{(z_{l,j},p_{l,j})}  \phi_l  (  D_{(z_{l,j},p_{l,j})} \pi )^{-1}  L_{l,j}   \subset T_{(x_l,p_l)} \mathcal{P}^+\mathcal{C}_l , 
  \]
   is $(n+1)$-dimensional. There are only countable many of such manifolds so 
 there is a small submanifold $S_1 \subset  \mathcal{P}^+\mathcal{C}_1 $ through $(x_1,p_1)$ of dimension $\dim T_{(x_1,p_1)} \mathcal{P}^+\mathcal{C}_1 - (n+1) = n-2$ such that each of the spaces $ E_{1,j}$, $j\in J_1$ and $D\Psi^{-1}   E_{2,j}$, $j\in J_2$ intersect $T_{x_1,p_1} S_1$ transversally. By considering the dimension of these linear spaces, we observe that the intersection occurs only at the origin.  
  This implies the analogous condition also for $S_2 := D\Psi S_1$.
   It is straightforward to check that
   \[
   \text{ker} ( D_{(z_{l,j}, p_{l,j})} \pi )  \cap T_{(z_{l,j}, p_{l,j})} K_{S_l} = \{ 0 \} 
   \]
which ensures that $D_{(z_{l,j}, p_{l,j})} \pi $ defines an isomorphism from $T_{(z_{l,j}, p_{l,j})} K_{S_l}$ to its image. This implies that $G_{S_l} = \pi  K_{S_l}$ near $(z_{l,j}, p_{l,j})$ is a manifold of dimension $\dim K_{S_l} =\dim S_l +1 = n-1$. 
Let us deduce the condition $ v_{l,j} \notin  T_{z_{l,j}}G_{S_l}  = \emptyset$.  
By the construction above, 
\[
T_{(x_l,p_l)} S_l \cap   E_{l,j}= \{0\}.
\]
Hence, 
\[
  ( D_{(z_{l,j},p_{l,j})} \phi_l )^{-1} (  T_{(x_l,p_l)} S_l \  \cap \     E_{l,j}  )\subset     \text{ker} ( D_{(z_{l,j},p_{l,j})} \phi_l )  .
\]
We check that
\[
 ( D_{(z_{l,j},p_{l,j})} \phi_l )^{-1} (  T_{(x_l,p_l)} S_l )= T_{(z_{l,j},p_{l,j})} K_{S_l}
 \]
  and 
  \[
   ( D_{(z_{l,j},p_{l,j})} \phi_l )^{-1}  E_{l,j}  =  (D_{(z_{l,j},p_{l,j})} \pi)^{-1}  L_{l,j}.
   \] 
By substitution we conclude 
\begin{equation}\label{a numbered formula}
T_{(z_{l,j},p_{l,j})} K_{S_l}  \ \cap \   (D_{(z_{l,j},p_{l,j})} \pi)^{-1}  L_{l,j}  \subset   \text{ker} ( D_{(z_{l,j},p_{l,j})} \phi_l ). 
\end{equation}
Applying $D_{(z_{l,j},p_{l,j})} \pi$ gives 
\[
T_{z_{l,j}} G_{S_l}  \ \cap \    L_{l,j}  \subset  D_{(z_{l,j},p_{l,j})} \pi  (  \text{ker} ( D_{(z_{l,j},p_{l,j})} \phi_l ) )  =  \R p_{l,j}
\] 
which implies $v_{l,j} \notin  T_{z_{l,j}} G_{S_l} $ by the definition of $ L_{l,j}$.

 \end{proof}

\color{black}

\section{Proof of Theorem \ref{themain}}\label{main-proof}

In this section we prove our main result Theorem~\ref{themain}. As shown in Section \ref{Bman-model}, the second Frech\'et derivative $\Phi''$ of the source-to-solution map $\Phi$  satisfies
\begin{equation}\label{adso00}
\begin{split}
\Phi '' (0; f,h) 
&= \IX \COLOP [ \IX(f),  \IX( h) ]+  \IX \COLOP [ \IX (h),  \IX(f)] \quad \text{on } \SP M,
\end{split}
\end{equation}
where $f,\, h$ are any compactly supported smooth functions and $\IX$ is the solution operator of the linearized problem~\eqref{vlasov1}. 
We use the microlocal properties of the collision operator $\COLOP$ we proved in the previous section to determine the wavefront set of $\Phi''(h_1,h_2)$ for sources $h_{1}$ and $h_{2}$ with singularities. 

\subsection{Delta distribution of a submanifold}\label{delta_dist_of_subm}


First, we construct the specific particle sources which we will use in our proofs.
Let $(M,g)$ be a $C^\infty$ smooth globally hyperbolic manifold of dimension $n$. Let $\mathcal{C}$ be a Cauchy surface of $(M,g)$. \antticomm{Despite the similar notation, this Cauchy surface should not be confused with the one in Theorem \ref{boltz-exist} and in the source-to-solution map. 
The surface here is fixed for the construction of the controllable sources below and it may intersect $\pi (\text{supp}A) $. 
} We introduce a parametrization $\R\times \SP \mathcal{C} \to \SP M$ for $\SP M$ as:
\begin{equation}\label{ll5789}
   \Psi(s, (x,p))= (\gamma_{x,p}(s),\dot{\gamma}_{x,p}(s)), \quad s\in \R, \ (x,p) \in  \SP \mathcal{C}.  
\end{equation}
Here $\SP \mathcal{C}\subset TM$ is the set of future-directed time-like vectors with base points in $\mathcal{C}$. We call the parametrization~\eqref{ll5789} the \emph{flowout parametrization} of $\SP M$. 
We refer to~\cite[Theorem 9.20]{Lee_2013} for properties of flowouts in general.

Let $S$ be a  submanifold \antti{(not necessarily closed) }
of $\SP\mathcal{C}\subset \SP M$.
The delta distribution $\delta_S\in \mathcal{D}'(\SP M)$ of the submanifold $S$ on $\SP M$ is defined as usual by
\[
 \delta_S(f)=\int_S \tbl{f(x,p)}\, dS, \text{ for all } f\in C_c^\infty( \PP^+ M),
\]
where $dS$ is the volume form \tbl{ of the submanifold $S$ of $\PP^+M$.} 
For our purposes it will be convenient to consider the delta distribution on $S$, in the case where $S$ is considered as a submanifold $\SP \mathcal{C}$ (instead of $\SP M$). \tbl{We distinguish this case and denote by $\check{\delta}_S\in \mathcal{D}'(\SP \mathcal{C})$ the distribution $\check{\delta}_S(\check{f})=\int_S \check{f}(x,p)\, dS$  for all  $\check{f}\in C_c^\infty(\SP \mathcal{C})$. 
The representation of $\delta_S$ in the flowout parametrization~\eqref{ll5789} of $\SP M$ is then 
\[
 \delta_S=\delta_0\otimes\check{\delta}_S,
\]
where $\delta_0\in \mathcal{D}'(\R)$ is the delta distribution on the real line with its support at the $0\in \R$, and $\check{\delta}_S\in \mathcal{D}'(\SP \mathcal{C})$ is as above. We will write simply $\delta_S(s,(x,p))=\delta_0(s)\s \check{\delta}_S(x,p)$, $s\in \R$ and $(x,p)\in \SP \mathcal{C}$. We write similarly for other products of pairs of distributions 
supported in mutually separate variables.}

 \tbl{We would like to view $\delta_S$ as a conormal distribution of the class $I^{m}(\SP M; N^*S)$. However, this is not strictly speaking possible due to the possible bonudary points of $\overline{S}$. 
 %
%
To deal with the possible boundary points of the submanifold $\overline{S}$, we are going to consider the product of a cutoff function and $\delta_S$ as follows. Let $(y,q)\in S$ and let $(s',s'')\in \R^{\dim{(\SP \mathcal{C})}-\dim{(S)}}\times  \R^{\dim{(S)}}$ be coordinates on a neighborhood $B\subset \SP \mathcal{C}$ of $(y,q)$ such that $S$ corresponds to the set $\{s'=0\}$ and $S$ is parametrized by the $s''$ variable. Let $\chi_R=\chi_R(s'')$ be a \antticomm{non-negative} cutoff function, which is supported in a small ball of radius $R$ and outside a neighborhood of the boundary of submanifold $\overline{S}$. We have that 
$\chi_R\delta_S =\chi_R(\delta_0\otimes\check{\delta}_S)\in \mathcal{E}'(\SP M)$ in the flowout parametrization is the oscillatory integral
\begin{equation}\label{delta_as_oscillory2}
 (\chi_R\delta_S)(s,s',s'')=\chi_R(s'')a(s'')\iint e^{is\s \xi} e^{i s' \cdot \xi'} d\xi d\xi'.
\end{equation}
Here $a(s'')$ corresponds the volume form of $S$ and the integration is over $\xi\in \R$ and $\xi'\in \R^{\dim{(\SP \mathcal{C})}-\dim{(S)}}$. 
}

By the above definition, $\chi_R(x,p)\delta_S(x,p)$ 
is a conormal distribution in the class $I^{m}_{comp}(\SP M; N^*S)$, where the order $m$ is 
\begin{equation}\label{order_m}
 m=\codim(S)/2-\dim(\SP M)/4,
\end{equation}
see Section~\ref{lagrangian-dist} for the definition of the order $m$. Here $\codim(S)=\dim(\SP M) -\dim(S)$. 
In the proof of Theorem~\ref{themain}, the submanifold $S$ will be either of dimension $0$ or $n-2$, and thus $m$ will be either $n/2$ or $1$ respectively. We remark that when $S$ is of dimension $0$ the cutoff function in~\eqref{delta_as_oscillory2} can be omitted.

\subsubsection{Approximate delta distributions}\label{approx_delta_dists}

We will use $C^\infty$ smooth sources that approximate delta distributions \tbl{$\delta_{S_1}$ and $\delta_{S_2}$ (multiplied by cutoff functions), where $S_1$ and $S_2$ are  submanifolds of $\SP \mathcal{C}$.} The dimension of $S_1$ will be $n-2$ and $S_2$ will be a point. These approximations are described next.

Let $S$ be a submanifold of $\SP \mathcal{C}\subset \SP M$ and let $\delta_S$ and $\chi_R$ be a \antticomm{non-negative} cut-off as in Section~\ref{delta_dist_of_subm} above. By using standard (Friedrichs) mollification, see e.g.~\cite{friedlander1998}, we have that there is a sequence $(h^\eps)$, $\eps>0$, of \antticomm{non-negative} functions $C_c^\infty(\SP M)$ such that
\begin{equation}\label{regu}
 h^\eps \to \chi_R\s \delta_{S} \text{ in } \mathcal{D}'(\SP M),
\end{equation}
as $\eps\to 0$. 
Let $u^\eps=\IX h^\eps\in C^\infty(\SP M)$ be the solution to
\begin{equation}\label{approx_sol}
 \begin{split}
\XX u^\eps  &= h^\epsilon \text{ on } \SP M \\
u^\eps &= 0 \text{ on } \SP \mathcal{C}^-.
\end{split}
\end{equation}
By the representation formula of solutions to~\eqref{approx_sol} given in Theorem~\ref{vlasov-exist} we have that 
\begin{equation}\label{approx_sol_prop}
u^\eps \geq 0, \quad \text{supp}(u^\eps) \subset K_{\text{supp}(h^\epsilon)}.
\end{equation}
By considering only small enough $\eps>0$, the support of the solutions $u^\epsilon$ can be taken to be in any neighborhood of $K_S$ chosen beforehand. 

Let $\widetilde{\mathcal{C}}$ be another Cauchy surface of $(M,g)$, which is in the past $\mathcal{C}^-$ of $\mathcal{C}$. Note that $u^\eps$ is then also a solution to $\XX u^\eps=h^\eps$ with \tbl{ $u^\eps =0$ on $\SP \widetilde{\mathcal{C}}^{-}$}. By Lemma~\ref{vlasov_ext_conormal}, $u^\eps$ is  unique. We also have that $WF(\chi_R\s\delta_S)\cap N^*(\SP \widetilde{\mathcal{C}})\subset N^*S\cap N^*(\SP \widetilde{\mathcal{C}})=\emptyset$.
By applying Lemma~\ref{vlasov_ext_conormal} again, we obtain 
\[
\lim_{\epsilon \rightarrow 0} u^\eps
= u
\]
in the space of distributions $\mathcal{D}'(\SP M)$. Here $u=\IX (\chi_R\s\delta_{S})$ solves
\begin{equation}\label{solution_with_delta_source}
\begin{split}
\XX u = \chi_R\s\delta_S \text{ on } \SP M \\
u = 0 \text{ on } \SP \widetilde{\mathcal{C}}.
\end{split}
\end{equation}
Finally, if $\chi\in C_c^\infty(\SP M)$ is a cutoff such that $\supp{\chi}\subset\subset \SP M\setminus S$, we have that $\chi u\in I^{m-1/4}(\SP M; N^*K_S)$, with $m$ given in~\eqref{order_m}.

\subsection{\tbl{Nonlinear interaction in the inverse problem}}\label{sca123}
Throughout this Section~\ref{sca123} we assume that $(M,g)$ is a globally hyperbolic manifold of dimension $n$ and that $\mathcal{C}$ is a Cauchy surface of $(M,g)$. \antticomm{Again, this Cauchy surface should not be confused with the one in Theorem \ref{boltz-exist} used for the source-to-solution map.} We also assume that the submanifolds $S_1\subset \PP^+ \mathcal{C}$ and $S_2=\{(x_0,p_0)\}\in \PP^+ \mathcal{C}$, $x_0 \notin \overline{S}_1$,  are such that the flowouts $K_{S_1}$ and $K_{S_2}$ satisfy the admissible intersection property (Definition~\ref{intersection_coords}) in $\pi(\text{supp} A)$, where $A$ is an admissible collision kernel.  
\antticomm{\tblu{As shown in Corollary \ref{coro_renerew}, such submanifolds $S_1$ and $S_2$ can be constructed on neighborhoods of  any pair of distinct base points of $\mathcal{C}$. The choice of $S_1$ and $S_2$ can be done so that the corresponding flowouts have admissible intersection property on both manifolds $M_1$ and $M_2$ (in $\pi(\text{supp} A_1)$ and $\pi(\text{supp}A_2)$ respectively).} In the proof for the main theorem of this article (Theorem \ref{themain}) we shall vary the sources and the underlying Cauchy surface in order to generate collisions at variable points in the space-time. }
\antticomm{Let $\hat\mu : [-1,1] \rightarrow M$ be $C^\infty$-smooth timelike curve and $V\subset M $ be an open neighbourhood of $\hat{\mu}$. As we are allowed to control the sources in $V$, the distinct points which the sources are constructed around shall lie in $\mathcal{P}^+V$. By localizing the sources we may always assume that $S_1,S_2$ also lie in $\mathcal{P}^+V$. }

The length of a piecewise smooth causal path $\alpha : [a,b] \rightarrow M$ is defined  as 
\begin{equation}\label{lengtyht}
l(\alpha) := \sum_{j=0}^{m-1} \int_{a_j}^{a_{j+1}} \sqrt{ -g( \dot\alpha(s), \dot\alpha(s)) } ds,
\end{equation}
where $a_0<a_1< \cdots<a_{m-1} < a_m$ are chosen such that $\alpha$ is smooth on each interval $(a_j,a_{j+1})$ for $j=0,\dots,m-1$. \tbl{The time separation function, see e.g.~\cite{oneill1983semiriemannian}, is denoted by $\tau : M \times M \rightarrow [0, \infty)$ and defined as 
\[
\tau(x,y) :=\begin{cases}  \sup l(\alpha) , & x<y \\
0, & \text{otherwise,}
\end{cases}
\] 
where the supremum 
is taken over all piecewise smooth lightlike and timelike curves $\alpha : [0,1] \rightarrow M$, which are smooth on each interval $(b_j,b_{j+1})$, and that satisfy $\alpha(0)  = x$ and $\alpha(1) = y$. 
%
%
If $\tau(x,y)=0$ and there is a lightlike geodesic $\gamma$ connecting points $x,y\in M$, we call $\gamma$ \emph{optimal}.
By~\cite[Proposition 14.19]{oneill1983semiriemannian}, we have that if $(M,g)$ is globally hyperbolic and if $x,y\in M$ satisfy $\tau(x,y)=0$, then an optimal lightlike geodesic $\gamma$ always exists. 
}

The main result of this Section~\ref{sca123} is the following:
\begin{proposition}\label{kooasu} 
%
Let $(M,g)$ be a globally hyperbolic manifold, $A:(TM)^4\to \R$ an admissible collision kernel with respect to the relatively compact subset $W \antticomm{=  I^-(x^{+}) \cap I^+(x^{-}) }\subset M$ \antticomm{$($see \eqref{W-set}$)$}. 
Let $\mathcal{C}\subset M$ be a Cauchy surface \antticomm{such that $\pi(\text{supp}A) \subset \mathcal{C}^+$}, $S_1$ be a submanifold of $\PP^+ \mathcal{C}$, and $S_2=\{(x_0,p_0)\}\in \PP^+ \mathcal{C}$, $x_0\notin \overline{S}_1$. Assume that
$K_{S_1}$ and $K_{S_2}$ have the admissible intersection property in $\pi(\text{supp} A)$ according to Definition \ref{intersection_coords}. Assume also that $\pi(K_{S_1})$ and $\pi(K_{S_2})$  intersect  in $ \pi( \text{supp} A)$ first time at some $z_1\in W$. Let $\gamma$ be \tbl{an optimal} future-directed light-like geodesic in $M$ such that $\gamma(0) = z_1$ and $e := \gamma(T)$. 

Additionally, let $h_{1}^{\epsilon},\s h_{2}^{\epsilon}\in C_c^\infty(\PP^+M)$ be the  approximations of the distributions $\chi_{R}\s\delta_{S_1}$ and $\delta_{S_2}$, \tbl{where $S_2$ is a point}, as described 
in Section~\ref{approx_delta_dists}. 
 
Then, there is a section $P_{e}: V_e\to L^+V_e$ on a neighborhood $V_e$ of $e$ 
such that the limit
\[
 \sslimit:=\lim_{\epsilon \rightarrow 0} (  \Phi'' (0;h_1^\epsilon, h_2^\epsilon)  \circ P_e ) 
 \]
  exists and 
  \antti{ 
\[
  \text{singsupp} ( \sslimit)   = \gamma  \cap V_e
\]
}
\end{proposition}
In the context of the main theorem, the proposition above can be used to detect rays of light propagating from the unknown set $W$ to the measurement neighbourhood $V$. In that setting 
$e \in V$ and $V_e \subset V$.
The situation is pictured in \tbl{Figure~\ref{fig:light_obs2}} below.
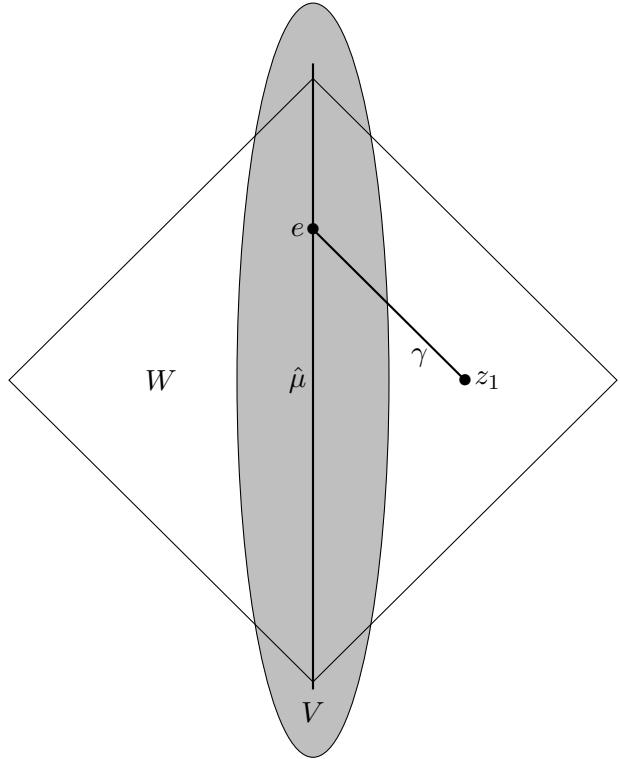
\begin{figure}[ht!]\label{fig:kooasu}
    \centering
%
\begin{tikzpicture}
    \draw [fill=lightgray] (0,0) ellipse (1cm and 5cm);
    \draw[thick] (0,4.2) -- (0,-4.1);
    \node at (-0.2,0) {$\hat\mu$} ;
     \node(dot) at (0, 2) {\textbullet};
    \node at (-0.2,2) {$e$} ;
        \node at (1.4,0.3) {$\gamma$} ;
       \node at (2,0) {\textbullet} ;
         \node at (2.3,0) {$z_1$} ;
         \draw[thick] (2,0) -- (0,2);
        \node at (-2,0) {$W$} ;
                \node at (0,-4.4) {$V$} ;
    \draw (0,4) -- (4,0) -- (0,-4) -- (-4,0) -- (0,4);
\end{tikzpicture}
   \caption{\antti{The setup of Proposition \ref{kooasu} can be used to detect singularities propagating along rays of light in $V$.} }
    \label{fig:light_obs2}
\end{figure}

Before proceeding with the proof of Proposition \ref{kooasu}, \antti{we state the following supporting result which follows from a simple dimension argument similar to the one used in Lemma \ref{lemma_renerew}. In other words, there is so much freedom for variation that caustic effects can be avoided at finite number of fixed points. 
The vector field $P_e$ in Proposition \ref{kooasu} is constructed as a restriction of $P$ in the lemma below for convenient choice of the vectors $(x_j,p_j)$. }

\begin{lemma}\label{ooiiuu}
Let $(x_j,p_j)  \in L^+(M)$, $j=1,\dots,m$, be a finite set of vectors. 
There is an open (possibly disconnected) neighborhood $ Q  $ of $\{x_1,\dots,x_m\}$  in $M$ and a smooth local section 
 $P : Q \rightarrow L^+Q$,  $P(x) = (x,p(x))$,
of the bundle $\pi : L^+M  \rightarrow M$ such that for \tbl{$x\in Q$ and $s\in \R$ such that $\gamma_{(x,p(x))} (s) \in  Q$
\begin{equation}\label{conditions_for_P}
 \begin{split}
P( x_j ) &= (x_j,p_j)  \\  
(\gamma_{x,p(x)}(s),\dot{\gamma}_{x,p(x)}(s)) &= P( \gamma_{(x,p(x))} (s) ).
\end{split}
\end{equation}
}
\end{lemma}

\tbl{Following~\cite{oneill1983semiriemannian},} we say a path $\alpha([t_1,t_2])$ is a pre-geodesic if  $\alpha(t)$ is a $C^1$-smooth
curve such that $\dot\alpha(t)\not =0$ 
on $t\in [t_1,t_2]$, and there exists a \tbl{reparametrization} of $\alpha([t_1,t_2])$ so that it becomes a geodesic. 
Proposition 10.46 of \cite{oneill1983semiriemannian} implies the existence of a shortcut path between points which are not connected by lightlike pre-geodesics: 

\begin{lemma}[Shortcut Argument]\label{shortcut_lemma}
\tbl{Let $(M,g)$ be globally hyperbolic and }$x,y,z\in M$. Suppose that $x$ can be connected to $y$ by a future-directed lightlike geodesic $\gamma_{x\to y}$, and $y$ can be connected to $z$ by a  future-directed lightlike geodesic $\gamma_{y\to z}$. Additionally, assume that $\gamma_{x\to y}\cup\gamma_{y\to z}$ is not a lightlike pre-geodesic.
Then, there exists a timelike geodesic connecting $x$ to $z$.~\f{Let's say what the shortcut argument implies, and where it is used.}
\end{lemma}


%
%
%
We now prove Proposition \ref{kooasu}:

\begin{proof}[Proof of Proposition \ref{kooasu} ]\label{89347jjaa}
\antticomm{Recall the notation $G_{S} = \pi( K_{S})$ for the base manifold flowout (see \eqref{Gees}) from a manifold $S$.} 
At this point, let $U$ be an open neighborhood of $z_1$, small enough so that $G_{S_1}\cap U$ and $G_{S_2}\cap U$ are manifolds and that $U\cap (\pi(S_1)\cup \pi(S_2))=\emptyset$. 
Let $P: U\to L^+(U)$ be an arbitrary smooth section on $U$. We remark that we will choose a specific $P$ later in the proof. 


Let $h_1^\eps$ and $h_2^\eps$ be the respective approximations of the distributions $\chi_{R}\s\delta_{S_1}$ and $\delta_{S_2}$ as described in~\eqref{regu}-\eqref{solution_with_delta_source} compactly supported in $\PP^+M$. (We do not multiply $\delta_{S_2}$ with a cutoff function, since $S_2$ is just a point.) 
For $j=1,2$, let $u_j^\eps$ be the solutions corresponding to $h_j^\eps$: 
\begin{equation}\label{regularized_equation}
\begin{split}
\XX u_j^\eps  &= h_j^\epsilon \text{ on } \SP M \\
u_j^\eps &= 0 \text{ on } \SP \mathcal{C}^-.
\end{split}
\end{equation}
By Lemma \ref{deriv} we have that the second linearization $\Phi''$ of the source to solution map $\Phi$ satisfies 
\begin{equation}\label{second_lin_solves}
\begin{split}
\Phi'' (0\;h_1^\epsilon, h_2^\epsilon)  \circ P  =  \big(\IX\mathcal{Q} [ u_1^\eps , u_2^\eps ]\big)\circ P +  \big(\IX\mathcal{Q} [u_2^\eps , u_1^\eps ]\big)\circ P, \\
\end{split}
\end{equation}
where $\IX$ is the solution operator to the Vlasov equation~\eqref{vlasov1}. 
We first study the limit $\epsilon \to 0$ of 
\[
 \mathcal{Q} [ u_1^\eps , u_2^\eps ]\circ P +  \mathcal{Q} [u_2^\eps , u_1^\eps ]\circ P.
\]

%
Since $h^{\epsilon}_{1}$ and $h^{\epsilon}_{2}$ are supported on $\SP M$ and since $P(x)$ is light-like, we find that 
\begin{equation}\label{tyhja123}
\mathcal{Q}_{loss} [ u_1^\eps , u_2^\eps ] \circ P(x) = \mathcal{Q}_{loss} [u_2^\eps , u_1^\eps ] \circ P(x)=0 .
\end{equation}
Therefore, the light which scatters from the collisions arises only from the terms 
\[
\mathcal{Q}_{gain}^P [ u_1^\eps , u_2^\eps ]= \mathcal{Q}_{gain} [ u_1^\eps , u_2^\eps ] \circ P \ \ \text{and} \ \ \mathcal{Q}_{gain}^P [ u_2^\eps , u_1^\eps ]= \mathcal{Q}_{gain} [ u_2^\eps , u_1^\eps ] \circ P.
\]


\antti{By the discussion in Section~\ref{approx_delta_dists} we have that, away from $S_1$ and $S_2$, the element $  \IX(\chi_R\s \delta_{S_1})$ lies in $ I^{n/2-1/4}(\SP M; N^*K_{S_1})$ and $  \IX(\delta_{S_2})$ lies in $ I^{1-1/4}(\SP M; N^*K_{S_2})$. 
Therefore, if $\Gamma$ is any conical neighborhood of $\Lambda_1\cup\Lambda_2$, where $\Lambda_1$ and $\Lambda_2$ are as in Property~\ref{intersection_coords}, we have by \antti{Corollary} \ref{sofisti_further} that 
\begin{equation}\label{thiscvo}
\begin{split}
   \mathcal{Q}_{gain}^P& [ \IX \delta_{S_2} , \IX (\chi_R\s\delta_{S_1}) ] +  \mathcal{Q}_{gain}^P [ \IX (\chi_R\s \delta_{S_1}) ,  \IX \delta_{S_2}  ]  \\
 &\in I^l(\Lambda_0 \setminus (\Lambda_1 \cup \Lambda_2)\;U) + \mathcal{D}_\Gamma'(U) \subset \mathcal{D}'(U),
 \end{split}
 \end{equation}
 where $l=(n/2-1/4)+(1-1/4)+3n/4-1/2=5n/4$ \antti{and $U$ is a small neighbourhood around $z_1$. }
 Here we use the standard notation to denote 
 \[
 \begin{split}
%
 \mathcal{D}_\Gamma'(U) = \{ u \in \mathcal{D}'(U)\s:\s WF(u) \subset \Gamma \}.
 \end{split}
\]
By \antticomm{the sequential} continuity of $\IX$ (see Section~\ref{approx_delta_dists}) and $\mathcal{Q}_{gain}^P$ (Corollary \ref{sofisti_further}), the distribution~\eqref{thiscvo} equals the limit
\begin{equation}\label{implieslimit}
\lim_{\epsilon\rightarrow 0} \big( \mathcal{Q}_{gain}^P [ u_1^\epsilon , u_2^\epsilon ] +  \mathcal{Q}_{gain}^P [ u_2^\epsilon,  u_1^\epsilon  ]\big),
\end{equation}
\tbl{ in $\mathcal{D}'(\SP M)$.} It also follows by definition of the collision operator and $\text{supp} (u_j^\epsilon ) \subset K_{\text{supp}(h_j^\epsilon)}$ that the support of \eqref{implieslimit} focuses close to the intersection point $z_1$ as $\epsilon \to 0$. In particular, we may consider it as a compactly supported distribution in $U$.
Now let $P : x \mapsto (x,p(x))$ be a restriction of a section in Lemma \ref{ooiiuu} for $k=2$, $(x_1,p_1)=(z_1, \dot\gamma(0))$ and $(x_2,p_2) = (e,\dot\gamma(T))$ where $\gamma$ is the optimal geodesic connecting $z_1$ to $e$. Then the terms in \eqref{second_lin_solves} satisfy
\begin{align}
 (  \IX  \mathcal{Q}_{gain} [ u_1^\eps,u_2^\eps ]  ) \circ P   &= P^{-1}  \mathcal{Q}_{gain}^P [ u_1^\eps,u_2^\eps ] \label{2134}  ,
 \\
 ( \IX  \mathcal{Q}_{gain} [ u_2^\eps,u_1^\eps ]  ) \circ P &=   P^{-1} \mathcal{Q}_{gain}^P [ u_2^\eps,u_1^\eps ] \label{2135} ,
\end{align}
where $P^{-1} : C_c^\infty (U) \to C^\infty(V_e)$, $P^{-1} \phi (x) := \int^0_{-\infty}  \phi ( \gamma_{x,p(x)} (t) ) dt$, i.e. the integration along integral curves of the light-like field $P$. Here $V_e$ is a small neighbourhood of $e$ in the domain of the field. We may assume that $U$ and $V_e$ are distinct.  
The operator $P^{-1}$ corresponds to the canonical relation\footnote{We do not  have to treat $P^{-1}$ as a FIO with a pair $(\Delta, C)$ of canonical relations  (cf. \cite{melrose-uhlmann1979}) since $U$ and $V_e$ are distinct sets. That is; the diagonal part $\Delta$ does not contribute in these domains.} 
\[
C=  \{ (x,\xi \ ; \ y,\eta ) \in T^*V_e \times T^* U :  y \in \gamma_{x,p(x)} , \  p^j(x) \xi_j = 0, \ p^j(y) \eta_j = 0  \}.
\]
\antticomm{Moreover, $C$ is the set of pairs $(x,\xi \ ; \Xi (r, x, \xi) ) $ for convenient parameters $r<0$ where 
\[
r \mapsto \Xi (r, x, \xi) 
\in ( \gamma_{(x,p(x))}(r)  , \dot\gamma^\perp_{(x,p(x))}(r)  ) , \footnote{Here $\dot\gamma^\perp_{(x,p(x))}(r): = \{ \xi \in T^*_{\gamma_{(x,p(x))}(r)} M : \langle \xi , \dot\gamma_{(x,p(x))}(r) \rangle = 0 \}  $.} 
\]
is the bicharacteristic of $-iP$ through $(x,\xi)$ and the homogeneous principal symbol of $P^{-1}$ is non-vanishing and independent of the flow parameter $r$.
Hence, we may write the principal symbol of $P^{-1}$ on $C$ as a function of $(x,\xi)$ by applying the parametrisation of $C$ above. 

Recall from Corollary \ref{sofisti_further} that $\mathcal{Q}_{gain}^P [\IX (\chi_R\s\delta_{S_1}),  \IX \delta_{S_2}  ]$ is well defined and corresponds to the Lagrangian manifold $\Lambda_0= T^*_{z_1}M$ microlocally away from $\Lambda_j $, $j=1,2$. 
The distributional limit 
\begin{align}
 L := 
 \lim_{\epsilon \to 0} P^{-1}  \mathcal{Q}_{gain}^P [ u_1^\eps,u_2^\eps ]  = P^{-1} \mathcal{Q}_{gain}^P [\IX (\chi_R\s\delta_{S_1}),  \IX \delta_{S_2}  ], 
\end{align}
together with 
\[
WF (L) 
\subset C \circ \Lambda_0  = N^* \gamma_{(z_1,p(z_1))} \cap T^*V_e
\]
is a direct consequence of the Corollary \ref{sofisti_further}, \cite[Corollary 1.3.8]{duistermaat2010fourier} and the fact (deduced above) that the support of  $\mathcal{Q}_{gain}^P [ u_1^\eps,u_2^\eps ]$ focuses into the single point $z_1$. 
Analogous identities hold for the other term $\eqref{2135}$.  
For more detailed analysis of the wave front set, we need to compute the principal symbol: 

Provided that the conditions of transversal intersection calculus are satisfied (shown below), the principal symbol on $C \circ \Lambda_0 = N^* \gamma_{(z_1,p(z_1))}\cap T^*V_e$ away from $C \circ \Lambda_j$, $j=1,2$ (e.g. near $( \gamma_{z_1,p(z_1)}, -\dot{\gamma}^\flat_{z_1,p(z_1)}   )$ ) can be computed using the standard formula (see e.g. \cite[Theorem 4.2.2.]{duistermaat2010fourier}) which in our setting reads
\[
\sigma ( L ) (x,\xi) =
   \sigma ( P^{-1} ) (x,\xi)  \   \sigma (\mathcal{Q}_{gain}^P [\IX (\chi_R\s\delta_{S_1}),  \IX \delta_{S_2}  ] )(\Xi_0(x,\xi)) ,
\]
where $\Xi_0(x, \xi ) \in (z_1, p^\perp(z_1)) \subset  \Lambda_0$ stands for the homogeneous translation of $(x,\xi) \in N^* \gamma_{(z_1,p(z_1))} \cap T^*V_e$ along the bicharacteristic $r \mapsto \Xi(r,x,\xi)$. 
As the positively homogeneous principal symbol of $P^{-1}$ is non-vanishing, it suffices to focus on the latter term $\sigma (\mathcal{Q}_{gain}^P [\IX (\chi_R\s\delta_{S_1}),  \IX \delta_{S_2}  ] )(\Xi_0 (x,\xi)) $. The term was computed in Corollary \ref{sofisti_further} and we obtain a non-vanishing principal symbol by the fact that $A$ is admissible in $W\ni z_1$. This implies the presence of singularities in the claim. 

Let us finish the proof by showing that the conditions \cite[Theorem 4.2.2]{duistermaat2010fourier} of the transversal intersection calculus are satisfied for the composition of $P^{-1}$ and $\chi \mathcal{Q}_{gain}^P [\IX (\chi_R\s\delta_{S_1}),  \IX \delta_{S_2}  ]$. Here $\chi= \chi(x,D)$ is a microlocal cut-off that vanishes on $\Lambda_1$ and $\Lambda_2$ so that the resulting object is Lagrangian distribution over $\Lambda_0$. 
The condition \cite[(4.2.4), Theorem 4.2.2]{duistermaat2010fourier} is clear by locality of the construction.  Moreover, \cite[(4.2.5), (4.2.6), Theorem 4.2.2]{duistermaat2010fourier} follow directly from the definition of $C$ and $\Lambda_0$. Let us check the transversality condition \cite[(4.2.7), Theorem 4.2.2]{duistermaat2010fourier}. It suffices to show that $(C \times \Lambda_0) \cap (T^*V_e \times \Delta_{T^*U})$ is a manifold of dimension $n$. This intersection is the set 
\[
\{  ( x ,\xi \ ; \  \Xi_0(x  ,\xi) \ ; \ \Xi_0(x ,\xi) ) : (x,\xi)  \in N^* \gamma_{z_1,p(z_1)}  \cap T^*V_e \} 
\]
which is a manifold of dimension $\dim ( N^* \gamma_{z_1,p(z_1)} \cap T^*V_e ) = n$. In conclusion, the conditions are satisfied. 

}

}
\end{proof}
 
Let $S_1$ and $S_2$ be as earlier. Then, $G_{S_1} =  \pi K_{S_1} $ and $G_{S_2} = \pi K_{S_2}$ intersect in $ \pi(\text{supp}(A))$ only finitely many times and at discrete points. Given that the intersections exist, we can write  $\{z_1, \dots, z_k\} = \pi(K_{S_1})\cap \pi(K_{S_2})\cap \pi(\text{supp}(A))$.
\begin{lemma}\label{estupido}
\antti{As above, denote the points in $\pi(K_{S_1})\cap \pi(K_{S_2})\cap \pi(\text{supp}(A))$ by $z_1,\dots, z_k$ (if exist) and  \tbl{arrange them so that $z_1\ll z_2\ll\cdots\ll z_k$. 
For every section} $P$ of the bundle $L^+V$ } we have
\[
\text{supp}( \sslimit_P) \subset
\bigcup_{l=1}^k \mathcal{L}^+(z_l) \subset  J^+(z_1),
\]
or $\text{supp}( \sslimit_P) = \emptyset$ if $\pi(K_{S_1})\cap \pi(K_{S_2})\cap \pi(\text{supp}(A)) = \emptyset$. 
Here we denote
\[
\sslimit_P := \lim_{\epsilon \rightarrow 0} (\Phi''( 0 ; h_1^\epsilon, h_2^\epsilon) \circ P ).
\]
\tbl{In particular, if for some section $P$ we have that $\text{supp}( \sslimit_P) \neq \emptyset$,} then the first intersection point $z_1$ exists.  
\begin{proof}
\antti{Let us adopt the notation of the proof of Proposition \ref{kooasu}.
One checks that}
\[
 \text{supp}\big(\mathcal{Q}_{gain}^P [u_1^\eps , u_2^\eps ] +  \mathcal{Q}_{gain}^P [ u_2^\eps,  u_1^\eps  ]\big) \subset E_\epsilon 
\]
\tbl{and
\[
\Phi''( 0 ; h_1^\epsilon, h_2^\epsilon) \circ P 
=\int_{-\infty}^0  \mathcal{Q}_{gain}^{\tbl{P}} [ u_1^\eps , u_2^\eps ]  ( \gamma_{(x,p(x))} (s) ) + \mathcal{Q}_{gain}^{\tbl{P}} [ u_2^\eps , u_1^\eps  ]  ( \gamma_{(x,p(x))} (s) ) ds, 
\]
}
where for $\epsilon>0$ 
\[
E_\epsilon := G_{\text{supp} ( h_1^\epsilon)}^+ \cap G_{\text{supp} ( h_2^\epsilon)}^+ \cap \pi(\text{supp}(A)) .
\]
\antti{As $E_\epsilon$ is a monotone sequence }it has the limit which is $\bigcap_{\epsilon>0} E_\epsilon = \{z_1,\dots,z_k\}$.
Thus, we obtain 
\[
\text{supp}( \Phi''( 0 ; h_1^\epsilon, h_2^\epsilon) \circ P ) \subset  \mathcal{L}^+ E_\epsilon
\]
and hence
\[
\text{supp}( \sslimit_P ):= \text{supp}(  \lim_{\epsilon \rightarrow 0} (\Phi''( 0 ; h_1^\epsilon, h_2^\epsilon) \circ P ) ) \subset \bigcap_{\tbl{\epsilon>0}} \mathcal{L}^+ E_\epsilon = \mathcal{L}^+ \{ z_l, \dots, z_k \} = \bigcup_{l=1}^k \mathcal{L}^+ (z_l)
\]
\tbl{It also follows from the definition of $\mathcal{L}^+$ and $J^+$ that $\bigcup_{l=1}^k \mathcal{L}^+ (z_l) \subset J^+(z_1)$.}
\end{proof}
\end{lemma}
In our inverse problem we will observe the $\text{supp}(\sslimit_P)$ and use the last claim of Lemma~\ref{estupido} to detect the collision of particles in a subset $W\subset M$.

%
%

 \subsection{Separation time functions }\label{time_sep_fnct}
Let $(M,g)$ be a globally hyperbolic, $C^\infty$-Lorentzian manifold. Also, let $\hat{\mu}:[-1,1]\to M$ be a given smooth, future-directed, timelike geodesic, and $V$ be an open neighbourhood of $\hat{\mu}$.  In this subsection we will use $g|_V$, $\hat{\mu}$ and $V$ to introduce a useful representation for certain points in $M$. 
Our representation scheme will associate a point $w\in M$ to a subset $\mathcal E_{\mathcal U}(w)\subset {\mathcal U}$, where ${\mathcal U}\subset V$ is an open set. Loosely speaking, the subset $\mathcal{E}_{\mathcal U}(w)$ will be comprised of points in ${\mathcal U}$ which lie on a \tbl{optimal} lightlike geodesic emanating from $w$. This representation was first introduced by Kurylev, Lassas, and Uhlmann in~\cite{Kurylev2018}; we reproduce a summary of it here for the reader's convenience. We later will use this representation to construct the desired isometry $F:W_1\to W_2$ described in Theorem \ref{themain}.


To begin, as shown in \cite[Section II]{manasse-misner}, there exists a bounded, connected, open set $\mathcal{A}\subset \R^{n-1}$ and a  neighbourhood $\mathcal{U}\subset V$ of $\hat{\mu}$ on which we may define \tbl{coordinates}
\[
x\in \mathcal{U}\mapsto (s,a^1,a^2,\dots,a^{n-1})\in [-1,1]\times\mathcal{A} 
\]
These coordinates have the property that $\hat\mu(s) =(s,0,\dots,0)$ and for fixed $a= (a^1,a^2,\dots,a^{n-1}) \in \mathcal{A}$ the map $\mu_a(s) = (s,a^1,a^2,\dots,a^{n-1})$ is a $C^\infty$-smooth timelike curve. Further, writing $\mu_{\hat{a}} = \hat{\mu}$ where $\hat{a} = (0,\dots,0)\in \mathcal{A}$, we have
\[
\mathcal{U} = \bigcup_{a\in \mathcal{A}}\mu_a[-1,1].
\]
Let $\overline{\mathcal{A}}$ be the closure of $\mathcal{A}$ in $\R^{n-1}$.
Below, we will assume that for all $a\in \overline{\mathcal{A}}$ we have $\hat{\mu}(s_+)\ll \mu_a(1)$ and 
 $\mu_a(-1)\ll \hat{\mu}(s_-)$.
%
%

 Given $\mathcal{U}$ and the family of curves $\mu_a$, ${a\in \overline{\mathcal{A}}}$ (which may be defined by replacing $\mathcal{A}$ above by a smaller open subset if necessary), we will next define the notions of time separation functions and observation time functions. 

Consider $-1< s^{-} < s^{+} < 1$ and set $x^{\pm}:= \hat\mu(s^{\pm})\in V$. As in \cite[Definition 2.1]{Kurylev2018}, for each $a\in \overline{\mathcal{A}}$ and corresponding path $\mu_a$, we define the \emph{observation time functions} $f_a^\pm:J^- (x^+) \setminus I^- (x^- )\to \R$ by the formulas
\[
f_a^+ (x) := \inf(  \{ s \in (-1,1) : \tau(x, \mu_a(s) ) > 0 \}  \cup \{1\} ) 
\]
and
\[
f_a^- (x) := \sup(  \{ s \in (-1,1) : \tau(\mu_a(s),x ) > 0 \}  \cup \{-1\} ).
\]
\tbl{Here $\tau$ is the time separation function defined in Section~\ref{sca123}. We note that if $x\in M$ and  $a\in \mathcal{A}$ are such that at least one point in $\mu_a(-1,1)$ can be reached from $x$ by a future-directed timelike curve
we obtain $\tau(x, \mu_a( f_a^+ (x) ) ) = 0$, see \cite{Kurylev2018}.~\f{Check that this is found from the paper, add location. -Tony} In this case, there also
exists a future-directed optimal light-like geodesic that connects $x$ to $\mu_a(f_a^+(x))$ as discussed in Section \ref{sca123}.} 

The  earliest time observation functions  $f_a^+:J^- (x^+) \setminus I^- (x^- )\to \R$
determine the set
\begin{equation}\label{eioreiie-eq}
\mathcal{E}_{\mathcal{U}} (w) = \{ \mu_a (f_a^+ (w) ) : a\in \mathcal{A} \}\subset U,
\end{equation}
that is the \emph{earliest light observation set} of $w\in J^- (x^+) \setminus I^- (x^- )$.

Finally, as shown in \cite[Proposition 2.2.]{Kurylev2018}, we may construct the conformal type of the open, relatively compact set $W\subset J^- (x^+) \setminus I^- (x^- )$ when we are given the collection
of all earliest light observation sets associated to points $w\in W$, that is, 
\[
\mathcal{E}_{\mathcal{U}}(W) =\{ \mathcal{E}_{\mathcal{U}} (w)\,:\, w\in W\} \subset 2^{\mathcal{U}}.
\]

%
%

 \subsection{Source-to-Solution map determines earliest light observation sets}\label{ss_elos}
 
In this section, we prove that the source-to-solution map for light observations (see \eqref{light-source-soln})
of the Boltzmann equation on a subset $V$ of a manifold determines the earliest light observation sets on a  subset of the manifold which properly contains $V$. We will define such a set below. 
After proving this, the main result of this paper, Theorem \ref{themain}, will follow by applying~\cite[Theorem 1.2]{Kurylev2018}, which states that the earliest light observation sets determine the Lorentzian metric structure of the manifold up to conformal class. 
 
  
 
From this point onwards, we assume that $(M_1,g_1)$ and $(M_2, g_2)$ are two geodesically complete, globally hyperbolic, $C^\infty$-Lorentzian manifolds, which contain a common open subset $V$ and 
 \[
g_1|_V=g_2|_V.  
 \]
 We assume that $\hat\mu:[-1,1]\to V$ is a given future-directed timelike geodesic. Let $\mathcal{A}\subset \R^{n-1}$, the family of paths $(\mu_a)_{a\in \mathcal{A}}$, and the subset $\mathcal{U}\subset V$ be as in the Section \ref{time_sep_fnct}.

 For $-1< s^{-} < s^{+} < 1$, we set $x^{\pm}:= \hat\mu(s^{\pm})\in V$ and define
 \begin{align}
 W_1 &:=  I^-(x^{+}) \cap I^+(x^{-})\subset M_1 \ \ \text{defined with respect to $(M_1,g_1)$},\label{ppa1}
 \\
 W_2 &:=  I^-(x^{+}) \cap I^+(x^{-})\subset M_2 \ \ \text{defined with respect to $(M_2,g_2)$}\label{ppa2}.
 \end{align}
 Additionally, for $\lambda=1,2$, let $A_\lambda$ be an admissible collision kernel (see Definition \ref{good-kernels}) with respect to the space $(M_\lambda,g_\lambda)$ and write $\Phi_{\lambda,L^+V}$ for the source-to-solution map for light observations (see Equation \eqref{light-source-soln}) associated to the relativistic Boltzmann equation \eqref{boltzmann} with respect to $g_\lambda$.  The notation $\Phi_\lambda$ denotes the full source-to-solution map for \eqref{boltzmann}.
 
 
In the above setting we prove:
\begin{repproposition}{ss_map_first_obs}
Let $\Phi_{1,L^+V}$ and $\Phi_{2, L^+V}$ be the above source-to-solution maps for light observations.  
Then  $\Phi_{1,L^+V}=\Phi_{2, L^+V}$ implies
\[
\mathcal{E}^1_{\s\mathcal{U}}({W_1}) = \mathcal{E}^2_{\s\mathcal{U}}({W_2}). 
\]
\end{repproposition}

We prove Proposition~\ref{ss_map_first_obs} by showing that $\Phi_{1,L^+V}=\Phi_{2,L^+V}$ implies the existence of a diffeomorphism
\[
F:W_1 \rightarrow W_2 \]
 which satisfies 
 \[
 \mathcal{E}_{\s\mathcal{U}}^1(w_1) = \mathcal{E}_{\s\mathcal{U}}^2(F(w_1)),\quad w_1\in W_1.\]
 
To construct the map $F:W_1 \rightarrow W_2$, consider the observation time functions on $(M_j,g_j)$, $j=1,2$ which we denote by $f_{a,j}^{\pm}$. For each 
\[
w_1 \in W_1\subset M_1,
\]
we define $\eta_{w_1}$ to be an optimal future-directed light-like geodesic in $M_1$ such that
\begin{equation}\label{9090sdj}
\eta_{\s {w_1}} (0) = \hat\mu (f_{\hat{a},1}^- (w_1) ) \text{ and }\eta_{\s {w_1}}(T) = w_1 \text{ for some }T>0.
\end{equation}
By the definition of $W_1$ such $\eta_{\s {w_1}}$ exists. In the following, $w_1$ will be fixed and we abbreviate
\[
 \eta_1:=\eta_{\s w_1}.
\]


Since $\mathcal{U}$ is open and $\text{dim}(M_1)=\nnn>2$, we may choose
another future-directed optimal light-like geodesic $\widetilde{\eta}_{\s 1}$ that is not tangential to $\eta_1$ and satisfies
\begin{equation}\label{9090sdjlk}
\widetilde{\eta}_{\s {1}}(0) = \mu_a (f_{a,1}^-(w_1)) \text{ and }\widetilde{\eta}_{\s {1}}(\widetilde{T}) = w_1\text{ for some }\widetilde{T}>0,  \quad  a\in \mathcal{A} \setminus \hat{a}.
\end{equation}
Since both segments $\eta_1|_{[0,T]},\tilde\eta_1|_{[0, \widetilde{T}]}$ are optimal, the shortcut argument implies that $\eta_{\s {1}}(s)$, $s>0$ and $\widetilde{\eta}_{\s {1}}(s')$, $s'>0$ intersect the first time at $w_1$. 


In the lemma below we approximate the light-like geodesics $\eta_{1}$ and $\tilde{\eta}_{1}$ by  time-like geodesics
\[
\gamma_{(\hat{x},\hat{p})}(s) \text{ and } \gamma_{(\hat{y},\hat{q})}(s) ,  \quad s>0, \quad (\hat{x}, \s \hat{p})\in \mathcal{P}^+V, \quad  (\hat{y}, \s \hat{q})  \in \mathcal{P}^+V
\]
that \emph{intersect for the first time at $w_1$ as geodesics in $M_1$}.
One may always fix the geodesics such that $\hat{x}$ and $\hat{y}$ belong to a  same  Cauchy surface and 
$(\hat{x}, \s \hat{p}) $ is arbitrarily near the curve   $(\gin,\dgin)$, and $(\hat{y}, \s \hat{q})  $ is arbitrarily near the curve $  (\ginp, \dginp)$.
Notice the abuse of notation: $\gamma_{(\hat{x},\hat{p})}$, $\gamma_{(\hat{y},\hat{q})}$ may refer to a pair of geodesics either in $(M_1,g_1)$ or $(M_2,g_2)$. 
 Due to global hyperbolicity we may always redefine the initial vectors by sliding them along the geodesic flow so that $\hat{x}$ and $\hat{y}$ lie in a Cauchy surface. 
 Let $\eta_2$ and $\tilde{\eta}_2$ be the geodesics in $(M_2,g_2)$, which 
 have the same initial data  with $\eta_{1}$ and $\tilde{\eta}_{1}$ respectively:
 \begin{equation}\label{last_geo}
  \dot{\eta}_2(0)=\dot{\eta}_{1}(0)\in T_{\eta_{1}(0)}\mathcal{U} \text{ and } \dot{\tilde{\eta}}_2(0)=\dot{\tilde{\eta}}_{1}(0)\in T_{\tilde{\eta}_{1}(0)}\mathcal{U}.
 \end{equation}
 (Recall that $ \mathcal{U} \subset V$ is a mutual set of $M_1$ and $M_2$ so that this makes sense.)

 
We thus define  
 \[
F :W_1 \rightarrow W_2  
 \]
as the map which assigns a given point $w_1\in W_1$ to the first intersection of $\eta_2(s)$, $s>0$ and $\tilde{\eta}_2(s')$, $s'>0$ denoted by $w_2\in M_2$.  For the assignment $w_1\mapsto w_2$ to be well-defined, we of course need to show that the first intersection $w_2$ exists and lies in $W_2$. There are also many choices for the geodesics $\eta_1$ and $\tilde{\eta}_1$ on $(M_1,g_1)$, which are used to define $\eta_2$ and $\tilde{\eta}_2$ on $(M_2,g_2)$. Therefore we need also to show that $w_2$ is independent of our choices of $\eta_1$ and $\tilde{\eta}_1$. These necessities are proven in Lemma~\ref{ss_map_determines_intersection} below. 
 \begin{figure}[h]
\begin{tikzpicture}[scale=2]
    \draw[fill=gray!20] (0,0) ellipse (0.5cm and 2.7cm); 
    \draw (0,-2.7) -- (0,2.7);
    \draw (0.25,-1.5) -- (0.25,1.5);
    \draw[thick] (0,-0.5) -- (1,0.5);
      \draw[blue,thick] (-0.03,-0.7) -- (1,0.5) -- (1.515,1.1);
    \draw[thick] (0.25,-0.5) -- (1,0.5);
      \draw[blue,thick] (0.27,-0.7) -- (1.365,1.1);
    \draw (0,1.8) -- (2,0) -- (0,-1.8) -- (-2,0) -- (0,1.8);
    \filldraw[red] (1,0.5)  circle (1pt);
  \node[red] at (1.3,0.3) {$w_{1}$};
\end{tikzpicture}
\caption{Given $w_{1} \in W_1$ (in red) the light-like geodesics $\gin,\ginp$ (in black) are chosen such that they maximize distance between their initial points in $\mathcal{U}$ (in gray) and $w_{1}$. The timelike geodesics $\gamma_{(\hat{x},\hat{p} )}$ and $\gamma_{(\hat{x},\hat{p} )}$ (in blue) intersect first time at $w_{1}$ and approximate the light-like segments.
}
\end{figure}
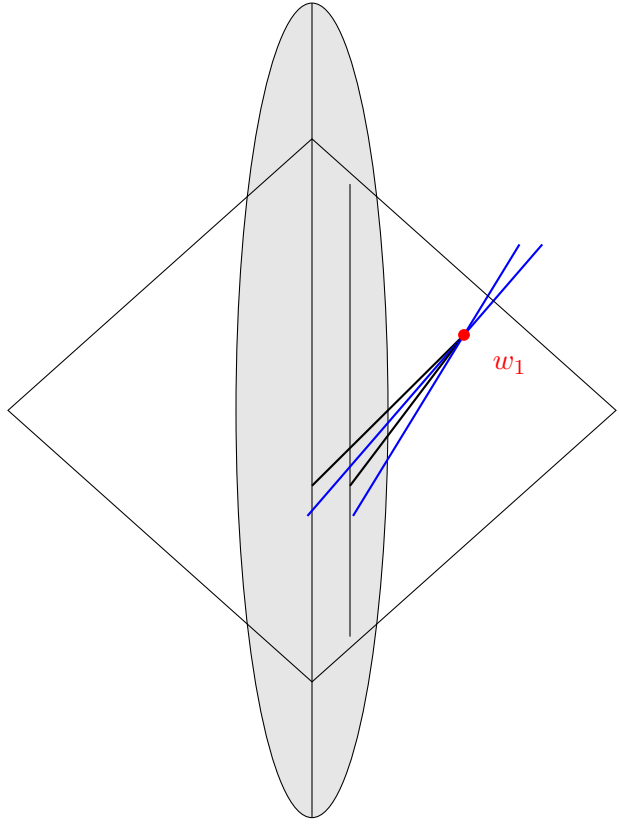
 
 \begin{lemma}\label{ss_map_determines_intersection}
 
 Let $(M_j,g_j)$, $\mathcal{U}\subset V \subset M_j$, $W_j \subset M_j$, $\Phi_{j}$ for $j=1,2$ be as described above. 
 Let $w_{1} \in W_1$  and consider light-like future-directed geodesics $\eta_{1}(s)$, $s>0$, $ \tilde{\eta}_{1}(s')$, $s'>0$ in $(M_1,g_1)$ with \eqref{9090sdj} and \eqref{9090sdjlk} intersecting the first time at $w_1$. Let $\eta_2$ and $\widetilde\eta_2$ be the associated light-like geodesics in $(M_2,g_2)$ with the initial conditions $( \eta_2(0), \dot\eta_2(0)) = ( \eta_1(0), \dot\eta_1(0))$ and $( \tilde\eta_2(0), \dot{\tilde\eta}_2(0)) = ( \tilde\eta_1(0), \dot{\tilde\eta}_1(0))$. 
Then the condition $\Phi_{1,L^+V}=\Phi_{2,L^+V}$ implies the following:
\begin{enumerate}
\item 
There exists the first intersection  $w_2$  of $\tgin(s)$, $s>0$ and $\tginp(s')$, $s'>0$ in $(M_2,g_2)$. Moreover, $w_2 \in W_2$.  
\item The first intersection point $w_{2}$ is independent from the choice of the geodesics $\gin $, $\ginp $ satisfying the required conditions above. 
\item For every  pair $(\hat{x},\hat{p}) \in \mathcal{P}^+ V$ and $(\hat{y},\hat{q}) \in   \mathcal{P}^+ V $ 
with the geodesics $\gamma_{(\hat{x},\hat{p})}(s)$, $s>0$, and $\gamma_{(\hat{y},\hat{q})}(s')$, $s'>0$ in $(M_1,g_1)$ intersecting the first time at $w_{1}$, the associated geodesics in $(M_2,g_2)$ intersect for the first time at $w_2$. 
\end{enumerate}

\end{lemma}




%
%

\begin{proof}[Proof of Lemma \ref{ss_map_determines_intersection}]
 Let $w_1\in W_1$ and let $e \in \mathcal{L}^+ ( w_{1} )$ be a first observation of $w_{1}$ in $\mathcal{U}\subset V \subset M_1$. That is, $e\in \mathcal{E}_{\mathcal{U}}^1(w_1)$. By the fact \eqref{eioreiie-eq} there is a point $a\in \mathcal{A}$ and the corresponding path $\mu_a$ such that
 \begin{equation}\label{diieee4}
 e= \mu_{a}(f_{a,1}^+(w_{1} ) ) \in (M_1,g_1).
 \end{equation}
Let $\gamma_{1}$ be the \tbl{optimal} geodesic in $(M_1,g_1)$ as in Proposition~\ref{kooasu} such that $ \gamma_1(0)=w_{1}$ and $ \gamma_1(1)=e$. 

We approximate the light-like geodesics $\eta_1$ and $\tilde{\eta}_1$ by time-like geodesics from $(\hat{x},\hat{p})\in  \mathcal{P}^+  V $ and $(\hat{y},\hat{q})\in \mathcal{P}^+ V$ as described earlier by requiring that for the geodesics in $(M_1,g_1)$ their first intersection for positive parameter values is at $w_1$. 
\tbl{Let $S_2=\{\hat{x},\hat{p}\}\subset P^{m_1}V$. By Corollary \ref{coro_renerew} 
 there exists a submanifold $S_1\subset \mathcal{P}^+ V$ with $(\hat{y},\hat{q})\in S_1$ such that the geodesic flowouts $Y_{j,1} := K_{S_1;M_j} \subset \mathcal{P}^+M_j$ of $S_1$ and $Y_{j,2}:= K_{S_2;M_j} \subset  \mathcal{P}^+M_j$ of $S_2$ in $(M_j ,g_j)$ have an admissible intersection property in the sense of Definition \ref{intersection_coords}} for both $j=1,2$. Therefore, we are in the setting for which the earlier results of this this section and Section~\ref{microlocal-collision} are valid. 
With this in mind, let us write  
\begin{equation}\label{3249reeeredfds}
 \pi(Y_{1,1})\cap\pi(Y_{1,2})=\{z_{1,1},z_{1,2},\ldots,z_{1,k_1}\}\subset M_1,
\end{equation}
where the intersection points $z_{1,l}$, $l=1,\ldots, k_1$, of $\pi(Y_{1,1})$ and $\pi(Y_{1,2})$ are ordered causally as $z_{1,1}\ll z_{1,2}\ll\cdots \ll z_{1,k_1}$. (The index $1$ in $z_{1,l}$ 
refers to the manifold $(M_1,g_1)$)
Notice that $z_{1,1} = w_1$. 
For $\eps>0$, let 
\[
 h_1^\eps\in C_c^\infty(\SP V)\text{ and } h_2^\eps\in C_c^\infty(\SP V)                                                                                                                                                                                                                                                                                            \]
be the sequences of approximative delta functions of $S_1$ and $S_2$ described in~\eqref{regu}.


By Proposition \ref{kooasu} there is a neighbourhood $V_e\subset V$ of $e$ and a section $P_e:V_e\to  L^+V_e$ 
 such that the distribution 
 \[
  \sslimit_1:=\lim_{\epsilon \rightarrow 0} (  \Phi_1'' (0;\s h_1^\epsilon,\s h_2^\epsilon)  \circ P_e )\in \mathcal{D}'(V_e),
 \]
 satisfies
\[
\text{singsupp}( \sslimit_1) = \gout  \cap V_e.
\]

The index $2$ in $Y_{2,j}$, and in $z_{2,l}$, $\mathcal{L}_2^+$ and $J_2^+$ below, refers to corresponding quantities on the manifold $(M_2,g_2)$. The submanifolds of $Y_{2,1}$ and $Y_{2,2}$ have admissible intersection property by their definition above. 

The condition \tbl{$\Phi_{1,L^+V}=\Phi_{2,L^+V}$} 
implies 
\[
 \sslimit_1:=\lim_{\epsilon \rightarrow 0} ( \Phi_1'' (0;\s h_1^\epsilon,\s h_2^\epsilon)  \circ P_e )=\lim_{\epsilon \rightarrow 0} ( \Phi_2'' (0;\s h_1^\epsilon,\s h_2^\epsilon)  \circ P_e ) =:\sslimit_2.
\]
Therefore, we have $\text{singsupp}( \sslimit_2) = \gamma_1 \cap V_e\subset V$ and
\[
 e \in \text{supp} (\sslimit_1)\cap \text{supp} (\sslimit_2).
\]
Consequently,  the sets $\pi(Y_{2,1})$ and $\pi(Y_{2,2})$ intersect by Lemma \ref{estupido}. Let us denote 
\[
 \pi(Y_{2,1})\cap\pi(Y_{2,2})=\{z_{2,1},z_{2,2},\ldots,z_{2,k_2}\}\subset M_2,
\]
where the intersection points of $\pi(Y_{2,1})$ and $\pi(Y_{2,2})$ are ordered as $z_{2,1}\ll z_{2,2}\ll\cdots \ll z_{2,k_2}$.
At this point we do not know whether $z_{2,1}$ is $w_2$ or not. 
By applying Lemma~\ref{estupido} again, we conclude for both $j=1,2$ that
\begin{equation}\label{tytyt}
\gamma_1  \cap V_e \subset \bigcup_{l=1}^{k_j} \mathcal{L}_2^+ (z_{j,l}) \subset J_j^+ (z_{j,1})  \subset (M_j,g_j).
\end{equation}
In particular, 
\begin{equation}\label{ralksd347}
  e \in  J_j^+(z_{j,1}) \subset (M_j,g_j).
\end{equation}

Recall that $V\subset M_j$ for both $j=1,2$ and $g_1|_V=g_2|_V$. Let $\gamma_2$ be the geodesic in $(M_2,g_2)$, which has the same initial condition as $\gamma_1$ at $e$, that is, $(\gamma_2(1) , \dot\gamma_2(1)) = (\gamma_1(1) , \dot\gamma_1(1))$. 
It follows that the geodesics 
$\gamma_1$ and $\gamma_2$ coincide in $ V_e \subset V$. 
Thus,
\begin{equation}\label{uuipaa}
\gamma_l  \cap V_e \subset  \bigcup_{h=1}^{k_j} \mathcal{L}_2^+ (z_{j,h}) \subset J_j^+(z_{j,1}) \subset  (M_j,g_j)
\end{equation}
for every combination of $l=1,2$ and $j=1,2$. 
%
%
%
%

\noindent\textbf{Proof of (1):} We prove 
that $\tgin$ and  $\tginp$ intersect in $W_2\subset M_2$ the first time at geodesic parameter times $s>0$ and $s'>0$. 


Fix $e:= \hat\mu( f_{\hat{a}}^+ (w_{1} ) )$, that is, $a= \hat{a}$ in \eqref{diieee4}.  
We approximate the light-like geodesics $\eta_1$ and  $\tilde{\eta}_1$ with  sequences $\gamma_{(\hat{x}_l,\hat{p}_l)}$ and $\gamma_{(\hat{y}_l,\hat{q}_l)}$, $l \in \mathbb{N}$ of time-like geodesics $\gamma_{(\hat{x},\hat{p})}$ and $\gamma_{(\hat{y},\hat{q})}$. 
In other words, we choose the geodesics such that for every $l\in \N $ the first intersection $z_{1,1}$ of $\gamma_{(\hat{x}_l,\hat{p}_l)}$ and $\gamma_{(\hat{y}_l,\hat{q}_l)}$ as geodesics in $M_1$ is $w_1$ and the initial values $(\hat{x}_l,\hat{p}_l)$ and $  (\hat{x}_l,\hat{p}_l)$ converge to some points in $ ( \eta_1, \dot\eta_1 ) \cap \OVS V$ and $( \tilde\eta_1, \dot{\tilde\eta}_1 )  \cap  \OVS V$, respectively.  
We may take $\hat{x}_l$ and $\hat{y}_l$ to lie in a fixed Cauchy surface $\mathcal{C}$ in $M_2$ 
and  $\hat{x}_l \in  W_2$ by removing the first terms in the sequence, if necessary.  
One applies the shortcut argument (Lemma \ref{shortcut_lemma}) and convergence of $\hat{x}_l$ to show that $\hat{x}_l \in J^+_2( \hat\mu(s_0) )$ for all indices $l$ and some $s_0\in (s^-,s^+)$. 
Moreover,  $\gamma_{(\hat{x}_l,\hat{p}_l)}(s) \in J^+_2( \hat\mu(s_0) )$ in $M_2$, for all $s>0$, by a similar argument. 
Consequently, we have for the first intersections $z_{2,1} = z_{2,1}(l)$ of $\gamma_{(\hat{x}_l,\hat{p}_l)}$ and $\gamma_{(\hat{y}_l,\hat{q}_l)}$ in $M_2$ the condition
\[
z_{2,1} \in \pi(Y_{2,1})\cap \pi(Y_{2,2})=\{\gamma_{(\hat{x},\hat{p})}(s)\in (M_2,g_2): s\geq 0 \}\cap \pi(Y_{2,2}) \subset J^+_2( \hat\mu(s_0) ) \quad \text{in} \quad (M_2,g_2),
 \]
 where we omitted the index $l$. 
From (\ref{ralksd347}) we obtain 
\[
 z_{2,1}(l)\in J^+_2( \hat\mu(s_0) ) \cap J^-_2(e)  \text{ in } (M_2,g_2),
\]
where $z_{2,1}(l)$ is the first intersection of the geodesics $\gamma_{(\hat{x}_l,\hat{p}_l)}$ and $\gamma_{(\hat{y}_l,\hat{q}_l)}$ in $(M_2,g_2)$. 
Thus, there exist sequences $(s_l)$ and $(s_l')$ of positive numbers such that 
\begin{equation}\label{324234325346}
 z_{2,1}(l)= \gamma_{(\hat{x}_l, \s \hat{p}_l)} (s_l) =   \gamma_{(\hat{y}_l, \s \hat{q}_l)}(s_l') \in   J^+_2( \hat\mu(s_0) ) \cap J^-_2(e). 
\end{equation}
Here $s_l>0$ and $s_l'>0$ are so that the geodesics $\gamma_{(\hat{x}_l,\hat{p}_l)}$,  $ \gamma_{(\hat{y}_l,\hat{q}_l)}$ intersect the first time at the geodesic parameter times $s_l$ and $s_l'$. 
To finish the prove of $(1)$, we show that a subsequence of $z_{2,1}(l)\in M_2$ converges to the first intersection of $\eta_2$ and $\tilde{\eta}_2$. 

Since $(M_2,g_2)$ is globally hyperbolic, the set $J^+_2( \hat\mu(s_0) ) \cap J^-_2(e)$ is compact. Thus, we may pass to a subsequence so that 
\[
z_{2,1}(l) \text{ converges in } J^+_2( \hat\mu(s_0) ) \cap J^-_2(e)\subset W_2 \text{ as } l\to \infty.
\]
%
Applying the parametrisation $\overline{\mathcal{P}}^+ \mathcal{C} \times \R \to \overline{\mathcal{P}}^+ M_2$, $(x,p,t) \mapsto \gamma_{x,p}(t)$ near the curve $\eta_2$ (resp. $\tilde{\eta}_2$) to \eqref{324234325346} implies that $s_l$ and $s_l'$ must converge. 
Thus, as the time-like curves in $M_2$ with initial values $(\hat{x}_l,\hat{p}_l)$ and $(\hat{y}_l,\hat{q}_l)$ approximate $\eta_2$ and $\tilde{\eta}_2$, respectively, it follows that there must exist the first intersection point of $\eta_2$ and $\tilde{\eta}_2$ in $W_2$ at the limit of $z_{2,1}(l)$. This is the point $w_2$. 

 \noindent\textbf{Proof of (2) and (3):} 
%
%
 Recall that the first intersection $w_1\in W_1$ of  the time-like approximations $\gamma_{ (\hat{x},\hat{p}) }$ and $\gamma_{(\hat{y},\hat{q})}$
is also the first intersection $w_1$ of $\gin$ and $\tilde{\eta}_1$ in $(M_1,g_1)$ by definition.  We also know that the first intersection exist for the associated time-like geodesics in $(M_2,g_2)$. 
  To prove (2), let $\gino$ and $\ginpo$ be another pair of geodesics that satisfies the  conditions  of $\gin$ and $\ginp$. 
 The first part of the proof above applies also for $\gino$ and $\ginpo$ and the first intersection is obtained as a limit of first intersections for some pair of time-like geodesics $\gamma_{ (\hat{x}_l',\hat{p}_l') }$ and $\gamma_{(\hat{y}_l',\hat{q}_l')}$ 
 that as geodesics in $M_1$ intersect first time in $w_1$ for every $l$ and approximate the light-like geodesics $\gino$ and $\ginpo$. 
 Thus it suffices to show that for two pairs 
 $(\hat{x},\hat{p}),(\hat{y},\hat{q})\in \mathcal{P}^+V$ and $(\hat{x}',\hat{p}'), (\hat{y}',\hat{q}')  \in \mathcal{P}^+V$ with the associated pairs of geodesics  in $(M_1,g_1)$  intersecting the first time at $w_1$ have the property that both pairs $\gamma_{ (\hat{x},\hat{p}) }$, $\gamma_{(\hat{y},\hat{q})}$ and $\gamma_{ (\hat{x}',\hat{p}') }$, $\gamma_{(\hat{y}',\hat{q}')}$ of geodesics in $(M_2,g_2)$ intersect the first time at a mutual point. Note that this point must be the limit $z_{2,1}(l) \to w_2$ of the first intersection points 
 constructed in the proof of $(1)$. In fact, the sequence is a constant sequence.

 As earlier, let $\gamma_1$ be a light-like optimal geodesic in $M_1$ from $w_1= \gamma_1 (0 )$ to $e = \gamma_1(1)$ and define the geodesic $\gamma_2$ in $M_2$ by the condition $(\gamma_2(1) , \dot\gamma_2(1)) = (\gamma_1(1) , \dot\gamma_1(1))$.
Let us then consider three pairs of initial vectors 
\begin{align}\label{vector-pairs}
\{ (\hat{x},\hat{p}), (\hat{y},\hat{q}) \},\ \ \{(\hat{x},\hat{p}), (\hat{y}',\hat{q}')\},  \text{ and }  \{(\hat{x}',\hat{p}'), (\hat{y}',\hat{q}') \}. 
\end{align}
constructed from the pairs  $(\hat{x},\hat{p}),(\hat{y},\hat{q})\in \mathcal{P}^+V$ and $(\hat{x}',\hat{p}'), (\hat{y}',\hat{q}')  \in \mathcal{P}^+V$ above. 
To each of these pairs of vectors, we may associate a pair of geodesics in $(M_1,g_1)$ which have the vectors as initial data. Each pair of geodesics has the property that they intersect for first time (for positive geodesic parameter times) at $w_{1}\in W_1$. On the other hand, to each pair of vectors in \eqref{vector-pairs}, we may achieve a pair of geodesics in $(M_2,g_2)$ that have the vectors as initial conditions at  $s=0$. 
As shown earlier, each pair of geodesics has the property that they intersect in $M_2$ and for the first time (for positive geodesic parameter times) it happens in $W_2$. 
We label these first intersection points in $W_2$ by $z^{(1)}$, $z^{(2)}$ and $z^{(3)}$, respectively. These intersections points lie in $\{ \gamma_2(s) : s<1 \}$ according to Lemma \ref{intersection_lightlike_geo} below. 
Before proving the lemma let us assume that it holds and show that the intersection points $z^{(1)}$, $z^{(2)}$, $z^{(3)}\in M_2$ ($z^{(1)} = z_{2,1}$ in the proof of (1)) are actually identical. The claim (2) then follows from this  since the first intersections  of the approximative time-like geodesics accumulate arbitrarily near $w_2$, as shown in the proof of (1). 

We argue by contradiction and suppose that $z^{(1)}\neq z^{(2)}$.   
That is, we suppose the first intersection of the pairs of geodesics 
\[
\{ \gamma_{(\hat{x},\hat{p})}, \gamma_{(\hat{y},\hat{q})}\} \text{ and } \{  \gamma_{(\hat{x},\hat{p})}, \gamma_{(\hat{y}',\hat{q}')}\}.
\]
in $(M_2,g_2)$ are distinct. 
The proof for the case $z^{(2)}\neq z^{(3)}$ is analogous. 
Applying Lemma \ref{intersection_lightlike_geo} we deduce that  $\gamma_{(\hat{x},\hat{p})}(s)$, $s>0$ hits $\{ \gamma_2(s) : s<1 \}$ twice, first at one of the points $z^{(1)}$, $z^{(2)}$ and then after at the other. 
%
We may assume that $\tau(z^{(1)} ,z^{(2)})>0$ in $M_2$. For $\tau(z^{(2)} ,z^{(1)})>0$ one simply swaps the roles of the points in the proof. 
The shortcut lemma (Lemma \ref{shortcut_lemma}) implies that there is a time-like future-directed segment connecting $z^{(1)}$ to $e$ in $M_2$.
Thus, there is some point $e'$ in the curve $\hat\mu$ that satisfies $\tau(e',e)>0$ and which can be reached from $z^{(1)}$ along a future-directed lightlike geodesic segment. 
We can now apply Corollary \ref{coro_renerew}  and 
 Proposition \ref{kooasu} to get $e' \in\text{singsupp} ( \lim_{\epsilon \rightarrow 0} ( \Phi''(0;  h_1^\epsilon ,  h_2^\epsilon) \circ P_{e'} ) ) $ in both spaces from which by Lemma \ref{estupido} one obtains $e' \in  J^+ (w_{1})$ in the space $(M_1,g_1)$. Since $\tau(e',e)>0$, we have $e \neq \hat\mu ( f_{\hat{a}}^+(w_{1}) )$ which is a contradiction. 
 %
 Thus, we have that 
 \[
  z^{(1)}=z^{(2)}.
 \]
 In a similar way one shows that $ z^{(2)}=z^{(3)}$.

\end{proof}


We now prove the following auxiliary lemma:

\begin{lemma}\label{intersection_lightlike_geo}
   Let $\gamma_1$ be \tbl{an optimal future-directed lightlike geodesic in $(M_1,g_1)$ between the points $w_{1}= \gamma_1(0)\in W_1$  and $e=\hat\mu( f_{\hat{a}}^+ (w_{1} ) ) = \gamma_1(1)$.} 
   Let $\gamma_2$ be the geodesic in $(M_2,g_2)$ with $(\gamma_2 (1) , \dot\gamma_2(1))= (\gamma_1 (1) , \dot\gamma_1(1))$. Let $(x,p),(y,q) $ be elements in $ \mathcal{P}^+   V $ for both $j=1,2$ such that in $(M_1,g_1)$ the geodesics $\gamma_{(x,p)}(s)=\gamma_{(x,p);M_1}(s)$, $s>0$ and $\gamma_{(y,q)}(s')= \gamma_{(y,q);M_1}(s')$, $s'>0$ intersect the first time at $w_1 \in W_1$. 
  Then, $\Phi_1 = \Phi_2$ implies that the first intersection point $z=z_{2,1}$ of $\gamma_{ (x,p) }(s) = \gamma_{ (x,p);M_2 }(s)$, $s>0$ and $\gamma_{(y,q)}(s')= \gamma_{(y,q);M_2}(s')$, $s'>0$  in $(M_2,g_2)$ belongs to $\{ \gamma_2(s) : s<1 \}$.  
 \end{lemma}
 \begin{proof}
Following the construction in page \pageref{3249reeeredfds} one puts together Corollary \ref{coro_renerew}, Proposition \ref{kooasu} to show that 
 for certain sources $h_1^\epsilon ,  h_2^\epsilon$ and a vector field $P_e$ on neighbourhood $U_e$ around $e$ we have 
 \[
 \text{singsupp} ( \lim_{\epsilon \rightarrow 0} ( \Phi''(0;  h_1^\epsilon ,  h_2^\epsilon) \circ P_{e} ) ) =  \gamma_2 \cap U_e.
 \]
Moreover, Lemma \ref{estupido} implies that $ \gamma_2 (1-\delta)$ belongs to the causal future of $z$ for sufficiently small $\delta >0$. Hence there is a causal geodesic $\nu$ connecting $z$ to $ \gamma_2 (1-\delta)$. 
If $z_1 \notin \{ \gamma_2(s) : s<1 \}$ we may combine $\nu$ with $\gamma|_{[1-\delta,1]}$ and apply shortcut argument to show that $z$ can be connected to $e$ with a timelike curve in $(M_2,g_2)$. 
Thus, $\hat{\mu}$ contains a point $e' \ll e$ that can be reached from $z$ with an optimal lightlike geodesic $\tilde\gamma_2$. We can then repeat the construction above with some vector field $P_{e'}$ tangent to $\tilde\gamma_2$ to derive that $e' \in J^+_1(w_1)$. Thus $\gamma_1$ does not define optimal segment from $w_1$ to $e$ which is in conflict with the assumptions. 

\end{proof}

%
%
The determination of the earliest light observation sets from the knowledge of the source-to-solution map data follows from 
Lemma \ref{ss_map_determines_intersection}: 

Below, let $f_{a,1}^+$ 
and $f_{a,2}^+$ be the earliest observation time functions on $M_1$ and $M_2$, respectively.
 Lemma~\ref{ss_map_determines_intersection} shows that there is a mapping 
%
\begin{equation}\label{F_map}
F : W_1 \rightarrow W_2 , \quad F(w_{1}) := w_{2},
\end{equation}
which maps the first intersection $w_1$ of $\eta_1$ and $\tilde{\eta}_1$ to the first intersection $w_2$ of $\eta_2$ and $\tilde{\eta}_2$. Here $\eta_j$ and $\tilde{\eta}_j$, $j=1,2$, are defined as in~\eqref{9090sdj}--\eqref{last_geo}. 

\begin{proposition}\label{prof: Map F preserves E-sets}
Let $\Phi_{1,L^+V}$ and $\Phi_{2,L^+V}$ be as earlier in Section~\ref{main-proof} and assume that $\Phi_{1,L^+V}=\Phi_{2,L^+V}$. Assume also that the conditions of Theorem \ref{themain} are satisfied for the Lorentzian manifolds $(M_1,g_1)$ and $(M_2,g_2)$.

Then, the map~\eqref{F_map}
\[
F : W_1 \rightarrow W_2 
\]
 is a bijection and
\begin{eqnarray}\label{eq: F and E sets}
f_{a,1}^+( w_1 ) = f_{a,2}^+ (F(w_1)),
\end{eqnarray}
for every $w_1\in W_1$ and $a\in \mathcal A$. 
\end{proposition}
\begin{proof}
Let $w_1\in W_1$ and $w_2=F(w_1)$. 
We show that the earliest observation time functions $f_{a,1}^+$ on $M_1$
and $f_{a,2}^+$ on $M_2$  satisfy $f_{a,1}^+(w_1)=f_{a,2}^+(w_2)$.

Let  $a\in \mathcal A$ and  $f_{a,1}^+:W_1\to \R$ 
be {an earliest observation time function} on $M_1$.
Moreover, let 
 $e= \mu_a (f_{a,1}^+ (w_{1} ) )$. (This means that $e$ can be reached by a lightlike future-directed geodesic from $w_{1}$.) 
Let $(\hat{x},\hat{p})$ and $(\hat{y},\hat{q})$ be timelike initial vectors as in the last condition of Lemma \ref{ss_map_determines_intersection}. In the space $(M_1,g_1)$ the first intersection of the associated geodesics is $w_1$ whereas the first intersection $z_{2,1}$ in $(M_2,g_2)$ is $w_2$. 
As in the proof of Lemma \ref{ss_map_determines_intersection} 
we deduce from  $\Phi_{1,L^+V}=\Phi_{2,L^+V}$  that 
$e \in J_2^+(z_{2,1})$ in $(M_2,g_2)$ and hence $w_2 \in J^-_2(e)$. 
Thus, the earliest light observation from $w_{2}$ on the path $\mu_a$ occurs either at $e$ or at a point that is in the past of $e$. In particular, 
\begin{equation}\label{ff}
f_{a,2}^+ (w_{2} ) \leq  f_{a,1}^+ (w_{1}). 
\end{equation}

To see that $F$ has an inverse, note that we may change the roles of $(M_1,g_1)$ and $(M_2,g_2)$ in Lemma~\ref{ss_map_determines_intersection} to have a mapping 
\[
 W_2\to W_1,
\]
which maps the first intersection of $\eta_2$ and $\tilde{\eta}_2$ to the first intersection of $\eta_1$ and $\tilde{\eta}_1$. Here $\eta_j$ and $\tilde{\eta}_j$, $j=1,2$, are defined as in~\eqref{9090sdj}--\eqref{last_geo}. This mapping is the inverse of $F$ and 
$F:W_1\to W_2$ is a bijection.

By interchanging the roles of $(M_1,g_1)$ and $(M_2,g_2)$ and repeating the above construction we see as in
\eqref{ff} that
\[
 f_{a,1}^+ (w_{1}) \leq  f_{a,2}^+ (w_{2}).
\]
Thus, we have
$f_{a,1}^+ (w_{1} ) =  f_{a,2}^+ (w_{2} )=  f_{a,2}^+ (F(w_{1}) )$ for all $w_1\in W_1$. 
Hence,
$F:W_1\to W_2$ is a bijection
 satisfying $f_{a,1}^+ (w_{1} ) = f_{a,2}^+ (F(w_{1}) )$ for all $w_1\in W_1$. 
 
\end{proof}

This result implies  Proposition \ref{ss_map_first_obs}:

\begin{proof}[Proof of Proposition \ref{ss_map_first_obs}]
By definition \eqref{eioreiie-eq} of the sets  
$\mathcal{E}_{\mathcal{U}}^j (w )$,
Proposition \ref{prof: Map F preserves E-sets}  proves that 
$\mathcal{E}_{\mathcal{U}}^1 (w_{1} )  = \mathcal{E}_{\mathcal{U}}^2  (w_{2} ) $ for
all $w_1\in W_1$ and $w_2=F(w_1)$. As $F:W_1\to W_2$ is a bijection,
this proves $\mathcal{E}_{\mathcal{U}}^1  (W_1)= \mathcal{E}_{\mathcal{U}}^2 (W_2) $.
\end{proof}

%
%

\subsection{Determination of the metric from the source-to-solution map}

 In this section we prove that $F$ is in fact an isometry,
\[
 F^*g_2=g_1.
\]
This will prove our main theorem, Theorem~\ref{themain}.

To this end, we apply Proposition \ref{prof: Map F preserves E-sets} to have the following implication of \cite[Theorem 1.2]{Kurylev2018}:

\begin{theorem}\label{oo1}
Let $\Phi_{1}$ and $\Phi_{2}$ be as earlier in Section~\ref{main-proof} and assume that $\Phi_{1,L^+V}=\Phi_{2,L^+V}$. Assume also that the conditions of Theorem \ref{themain} are satisfied for the Lorentzian manifolds $(M_1,g_1)$ and $(M_2,g_2)$.
Then, the map 
\[
{F} : W_1 \rightarrow W_2
\]
is a diffeomorphism  and the metric ${F}^* g_2$ is conformal to $g_1$ in $W_1$. 
\end{theorem}

\begin{proof} The claim follows from the proof of~\cite[Theorem 1.2]{Kurylev2018}. However,
let us briefly explain the main steps of the proof as we need the construction in the proof to determine the conformal factor between $F^*g_2$ and $g_1$ in the next section.

On the manifold $(M_j,g_j)$, let $\mathcal F_j:W_j\to C(\mathcal A)$ be the map from the point $q$
to the corresponding earliest observation time function, that is,  $\mathcal F_j(w)=f_{a,j}^+(w)$, $j=1,2$.
Then formula \eqref{eq: F and E sets} and \cite[Lemma 2.4 (ii)]{Kurylev2018} imply that
\begin{eqnarray}\label{eq: F and f maps }\mathcal F_1(w)=\mathcal F_2(F(w)),\quad
\hbox{for every $w\in W_1$}.
\end{eqnarray}
 In particular, the sets $\mathcal F_j(W_j)\subset C(\mathcal A)$, $j=1,2$, satisfy $\mathcal F_1(W_1)=\mathcal F_2(W_2)$, since $F:W_1\to W_2$ is a bijection by Proposition~\ref{prof: Map F preserves E-sets}. 
Thus we can write
$F=(\mathcal F_2)^{-1}\circ \mathcal F_1:W_1\to W_2$. By~\cite[Proposition  2.2]{Kurylev2018},
this implies that $F:W_1\to W_2$  is a homeomorphism. Hence the sets
$\mathcal F_j(W_j)$ can be considered as topological submanifolds of
the infinite dimensional vector space $C(\mathcal A)$.  It is shown in~\cite[Section 5.1.2]{Kurylev2018} 
that using the knowledge of the set $\mathcal F_1(W_1)$, we can construct 
for any point $Q\in \mathcal F_1(W_1)$ a neighborhood $U\subset  \mathcal F_1(W_1)$ of $Q$ and the values $a_1,a_2,\dots,a_n\in \mathcal A$ such that the following holds:
The point $q=\mathcal F_1^{-1}(Q)\in W_1$ has
a neighborhood $\mathcal F_1^{-1}(U)\subset W_1$,   
where the map $f_{a_1,\dots,a_n}:x\mapsto (f_{a_k,1}^+(x))_{k=1}^n$ defines smooth local coordinates on $W_1$. Observe that for $P\in U$, we have 
$f_{a_k,1}^+\circ \mathcal F_1^{-1}(P)=P(a_k)$. Thus when $E_a:C(\mathcal A)\to \R$, $a\in \mathcal A$ are the evaluation functions $E_a(P)=P(a)$ and $E_{a_1,\dots,a_n}(P)=(E_{a_j}(P))_{j=1}^n$, the above constructed pairs $(U,E_{a_1,\dots,a_n})$ determine on $\mathcal F_1(W_1)$ an atlas 
 of differentiable coordinates
on $\mathcal F_1(W_1)$ which makes $\mathcal F_1:W_1\to \mathcal F_1(W_1)$ a diffeomorphism. As noted above, such atlas can be constructed using only the knowledge of the set  $\mathcal F_1(W_1).$
Thus, as $\mathcal F_1(W_1)=\mathcal F_2(W_2)$, the atlases constructed on these sets
coincide. 
This implies that the homeomorphism $F=(\mathcal F_2)^{-1}\circ\mathcal F_1:W_1\to W_2$  is actually a diffeomorhpsim.
  
  Let us next consider the conformal classes of the metric tensors $g_1$ and $g_2$.  Let
 $(x,\eta)\in L^+\mathcal{U}$ be a future directed lightlike vector.
 Let us consider the lightlike geodesics $ \gamma^{\s j}_{(x,\eta)}(\R_-)$, $j=1,2$, on $(M_j,g_j)$ that lie in the causal past of the point $x$,
and let $s_1>0$ be such that $ \gamma^{\s j}_{(x,\eta)}([-s_1,0])\subset \mathcal{U}$.
 By~\cite[Proposition 2.6 ]{Kurylev2018} we have that a point $w_j\in W_j$ satisfies
 $w_j\in  \gamma^{\s j}_{(x,\eta)}((-\infty,-s_1])$  if and only if $ \gamma^{\s j}_{(x,\eta)}((-s_1,0))\subset \mathcal{E}_{\mathcal{U}}^j ( w_j )$. In other words, for a point $w_j\in W_j $ 
 the set $\mathcal{E}_{\mathcal{U}}^j ( w_j )$ determines whether $w_j\in  \gamma^{\s j}_{(x,\eta)}((-\infty,-s_1])$.
 By formula \eqref{eq: F and E sets}, this implies that $F$ maps the lightlike geodesic
  $\gamma^{\s 1} ((-\infty,-s_1])\cap W_1$ on to $\gamma^{\s 2} ((-\infty,-s_1])\cap W_2$,
  \begin{equation}\label{F_preserve_light_geos}
   F\big(\gamma^{\s 1} ((-\infty,-s_1])\cap W_1\big) \subset \gamma^{\s 2} ((-\infty,-s_1])\cap W_2
  \end{equation}

  Let 
 $w_1\in W_1$ and $\xi_1\in L^+_{w_1}M_1$  be such that $\gamma^{\s 1}_{(w_1,\xi_1)}(\R_+)\cap \mathcal{U}\not=\emptyset$,
 so that there is $t_1>0$ so that $\gamma^{\s 1}_{(w_1,\xi_1)}(t_1)\in \mathcal{U}$. Let $x=\gamma^{\s 1}_{(w_1,\xi_1)}(t_1)$  and 
 $\eta=\dot\gamma^{\s 1}_{(w_1,\xi_1)}(t_1)$.
   Then~\eqref{F_preserve_light_geos} implies that there is $t_0\in (0,t_1)$  such that 
 \begin{eqnarray}\label{eq: light-like geodesics}
 F\big(\gamma^{\s 1}_{(w_1,\xi_1)}((-t_0,t_0))\big)\subset F\big(\gamma^{\s 1}_{(x,\eta)}((-\infty,-t_1+t_0))\cap W_1\big)
\subset \gamma^{(2)}_{x,\eta}(\R)\cap W_2.
 \end{eqnarray}
Thus for any $w_1\in W_1$ and $\xi_1\in L^+_{w_1}M_1$  
 such that $\gamma_{w_1,\xi_1}(\R_+)\cap \mathcal{U}\not=\emptyset$, formula \eqref{eq: light-like geodesics}
 implies that 
  $F$ restricted to a neighborhood of $w_1$ maps the lightlike geodesic $\gamma^{\s 1}_{(w_1,\xi_1)}$ 
 to a segment of a light-like geodesic of $(M_2,g_2)$.
%
Hence, $F_*\xi_1\in L_{w_2}M_2$, where $w_2=w_1$.
 Thus for any $w_1\in W_1$ there are infinitely many vectors $\xi_1\in L^+_{w_1}M_1$  
  such that $\gamma_{(w_1,\xi_1)}^{\s 1}(\R_+)\cap \mathcal{U}\not=\emptyset$, and for which $F_*w_1\in L_{w_2}M_2$,
  we see that $F$  maps $L^+_{w_1}M_1$ to $L_{w_2}M_2$. Thus the metric tensor $g_1$ is conformal to $F^*g_2$ at $w_1$. As $w_1\in W_1$ is arbitrary, this implies that $g_1$ and $F^*g_2$  are conformal  
 \end{proof}

%
To complete the proof of Theorem \ref{themain} it ~ remains to show that the conformal factor is $1$.

%
%

\subsection{Determination of the conformal factor}\label{det-of-conform}
Here we prove that also the conformal factor of the metric is the same. This is the final step in the proof of Theorem \ref{themain}.

\begin{proposition}\label{conformii}
Let $(M,g), (\widetilde{M}, \widetilde{g})$ be globally hyperbolic $C^\infty$ manifolds with metrics $g$ and $\tilde{g}$.  
Let the subsets $W \subset M$ and $\widetilde{W} \subset \widetilde{M}$ be open. Let $c \in C^\infty (W)$ be a strictly positive function and assume there is a diffeomorphism 
$
F : W \rightarrow \widetilde{W} 
$ 
such that $(F^* \widetilde{g})(x) = c(x) g(x)$, $\forall x \in W$. 
Let $\gamma_{(x,p)}$ be the geodesic in $(M,g)$ with initial condition $\dot\gamma_{(x,p)} (0) = (x,p)$, where $(x,p) \in TW$. 
Also let $\sigma_{(x,p)}$ be the geodesic in the manifold $(W, F^*\widetilde{g})$ with $\dot\sigma_{(x,p)}(0) = (x,p)$.
Assume that for every $x \in W$ there exist two linearly independent timelike vectors $p_1,p_2\in T_x W$, real numbers $\epsilon_1,\epsilon_2>0$ and two smooth functions $\alpha_j: (-\epsilon_j,\epsilon_j) \rightarrow \R$, $j=1,2$, such that 
\begin{equation}\label{d77jlp}
\gamma_{(x,p_j)} (t) = \sigma_{(x,p_j)} ( \alpha_j (t) ) , \quad \text{for every} \quad t \in (-\epsilon_j , \epsilon_j), \quad j=1,2. 
\end{equation}

Then $c$ is a constant function on $W$. In particular, if the metrics $g$ and $F^*\widetilde{g}$ equal at some point $x \in W$, then they equal everywhere in $W$.
\end{proposition}

\begin{proof}
Let us write $c(x) = e^{\varphi(x)}$ for some smooth function $\varphi$, and $\Gamma_{ij}^k$ and $\widetilde\Gamma_{ij}^k$ for the Christoffel symbols of the metrics $g$ and $F^*\tilde{g}$ respectively. Denote the associated pair of covariant derivatives by $\nabla$ and $\widetilde\nabla$. 
The Christoffel symbols of the conformal metrics $g$ and $F^*\tilde{g}=c\s g$ are connected to each other by 
\[
\begin{split}
\widetilde\Gamma_{ij}^k
= \Gamma_{ij}^k +  \frac{1}{2}  \big(  \delta_j^k \partial_i  \varphi +  \delta_i^k \partial_j \varphi-  g_{ij} g^{kl} \partial_l  \varphi  \big).  
\end{split}
\]
(See  \cite[Ch. 3, Proposition 13.]{oneill1983semiriemannian}). 
Thus, for any smooth curve $\gamma$ in $W$, 
\begin{equation}\label{123hhheiq}
\begin{split}
\widetilde\nabla_{\dot\gamma} \dot\gamma =&  \ddot\gamma^k\partial_k  + \widetilde\Gamma_{ij}^k (\gamma)  \dot\gamma^i \dot\gamma^j  \partial_k\\
=&   \ddot\gamma^k \partial_k +\Big( \Gamma_{ij}^k(\gamma) +  \frac{1}{2}  \big(  \delta_j^k \partial_i  \varphi |_\gamma +  \delta_i^k \partial_j \varphi|_\gamma-  g_{ij}(\gamma) g^{kl}(\gamma) \partial_l  \varphi|_\gamma  \big)  \Big) \dot\gamma^i \dot\gamma^j  \partial_k\\
=&   \ddot\gamma^k \partial_k + \Gamma_{ij}^k(\gamma)  \dot\gamma^i \dot\gamma^j   \partial_k +     \partial_j \varphi |_\gamma \dot\gamma^j \dot\gamma^k  \partial_k -    \frac{1}{2}g_{ij}(\gamma) \dot\gamma^i \dot\gamma^j  g^{kl}(\gamma) \partial_l  \varphi |_\gamma \partial_k   \\
=&   \nabla_{\dot\gamma} \dot\gamma +      \langle d\varphi , \dot\gamma \rangle  \dot\gamma -    \frac{1}{2} g (\dot\gamma,\dot\gamma) \nabla \varphi  .  \\
\end{split} 
\end{equation}
Let $x \in W$ and,  for $j=1,2$, let 
$p_j\in T_xW$, $\eps_j>0$ and $\alpha_j: (-\epsilon_j,\epsilon_j) \rightarrow \R$ be as in the claim of this proposition. Let us relax the notation by omitting the indices $j=1,2$ from subscripts. 
By differentiating~\eqref{d77jlp} in the variable $t$ we obtain
\begin{equation}\label{ghus}
\dot\gamma_{(x,p)} (t) = \alpha'(t) \dot\sigma_{(x,p)} ( \alpha (t) )
\end{equation}
for $t \in (-\epsilon , \epsilon)$. 
Since 
\[
\sigma_{(x,p)}( \alpha(0)) = \gamma_{(x,p)} (0) = x = \sigma_{(x,p)} (0)
\]
and since the causal geodesics $\sigma_{(x,p)}$ does not have self-intersection by global hyperbolicity (recall that $F$ identifies $ \sigma_{(x,p)}$ with a causal geodesic in the globally hyperbolic manifold $\widetilde{M}$), we obtain $\alpha(0) = 0$. 
Substituting $\alpha(0) = 0$ to (\ref{ghus}) yields
\[
p = \dot\gamma_{(x,p)} (0) = \alpha'(0) \dot\sigma_{(x,p)} ( \alpha (0) ) = \alpha'(0) \dot\sigma_{(x,p)} ( 0 )  = \alpha'(0)\s p,
\]
that is, $\alpha'(0) = 1$. Thus, $\alpha(t) = t + O(t^2)$ near $t=0$. In particular,
\[
\tilde\nabla_{\dot\gamma_{(x,p)}} \big[ \dot\sigma_{(x,p)}( \alpha(t)) \big]|_{t=0 }  = \tilde\nabla_{\dot\sigma_{(x,p)}} \dot\sigma_{(x,p)}( t) |_{t=0 } = 0.
\]
Substituting this to (\ref{123hhheiq}) and using that $\gamma_{(x,p)}$ is a geodesic imply 
\[
\begin{split}
0 =& \nabla_{\dot\gamma_{(x,p)}} \dot\gamma_{(x,p)} (t) |_{t=0} =    \widetilde{\nabla}_{\dot\gamma_{(x,p)} } [  \alpha'(t) \dot\sigma_{(x,p)} ( \alpha (t) ) ] |_{t=0} -    \langle d\varphi , \dot\gamma_{(x,p)} (0) \rangle  \dot\gamma_{(x,p)} (0) \\
&+    \frac{1}{2} g (\dot\gamma_{(x,p)} (0), \dot\gamma_{(x,p)} (0) ) \nabla \varphi ( \gamma_{(x,p)} (0))  \\
=&  \alpha''(0)  \dot\sigma_{(x,p)} ( 0 ) +  \widetilde\nabla_{\dot\gamma_{(x,p)}}[ \dot\sigma_{(x,p)}( \alpha(t))] |_{t=0 }    -    \langle d\varphi (x) , p \rangle p+    \frac{1}{2} g (p,p ) \nabla \varphi(x)  \\
=& \big( \alpha''(0)      -    \langle d\varphi (x) , p \rangle \big)p+    \frac{1}{2} \underbrace{g (p,p )}_{\neq 0} \nabla \varphi(x).    \\
\end{split}
\]
In particular, $ \nabla \varphi(x) $ lies in the line spanned by $p$ in the space $T_xW$. Since this holds for every $x\in W$ and for the corresponding two linearly independent $p_1,p_2  \in T_xW$ (recall that we omitted the subscripts in the calculation above), we conclude that $\nabla \varphi(x) = 0$ for every $x \in W$. That is, the function $ \varphi$, and hence $c= e^\varphi$ is a constant function on $W$. 
\end{proof}

We finish the proof Theorem~\ref{themain} by showing that the conformal factor is $1$; that is, $g_1 = F^* g_2$ on $W_1$.

\begin{theorem}
Assume that the conditions of Theorem \ref{themain} are satisfied and $\Phi_{1,L^+V} = \Phi_{2,L^+V}$ for Lorentzian manifolds $(M_j,g_j)$, $j=1,2$. 
Equip the associated spaces $W_j$, $j=1,2$ with the canonical metrics $g_j|_{W_j}$ induced by the trivial inclusions $W_j \hookrightarrow M_j$.
Then, the map 
\[
{F} : W_1 \rightarrow W_2
\]
is an isometry. 
\end{theorem}

\begin{proof}
Let $w_{1} \in W_1$ be arbitrary and choose two time-like vectors $(\hat{x},\hat{p} ), (\hat{y},\hat{q} )$ based in $\mathcal{U}$ such that the geodesics $\gamma_{(\hat{x},\hat{p} )}$ and $\gamma_{(\hat{y},\hat{q} )}$ intersect first time at $w_{1}$ in $(M_1,g_1)$ as earlier.
By Lemma \ref{ss_map_determines_intersection} we may fix them such that the corresponding  geodesics in $(M_2,g_2)$ intersect first time at $F(w_{1}) = w_{2}$. For $(x,p) \in TW_1$ let us denote by $\sigma_{(x,p)}$ the geodesic in $(W_1,F^*g_2)$ that satisfies $\dot\sigma_{(x,p)}(0) = (x,p) $.
While keeping $(x,p)$ fixed, we can vary the initial direction $(y,q)$ so that the intersection point of the geodesics $\gamma_{(y,q)}$ and $\gamma_{(x,p)}$ varies along the geodesic $\gamma_{(x,p)}$ through an open geodesic segment.  
Since the geodesics $\gamma_{(\hat{x},\hat{p} )}$ and $\gamma_{(\hat{y},\hat{q} )}$ intersect in both geometries, this variation of $w_{2}=F(w_{1})$ corresponds to a path through $w_{2}$ that is some reparametrisation of the geodesic  $\gamma_{(\hat{x},\hat{p} )}$ in $(W_2,g_2)$. 
Mapping this reparametrised geodesic segment to $M_1$ using $F^{-1}$ implies that there is a vector $p_1 \in T_{w_{1}} W_1$ and a smooth reparametrisation $\alpha_1(s)$ such that the geodesic $\gamma_{(w_{1},p_1)}$ in $(M_1,g_1)$ satisfies $\gamma_{(w_{1},p_1)}(s) = \sigma_{(w_{1},p_1)}( \alpha_1(s) )$  on some interval $s\in (-\epsilon_1,\epsilon_1)$. Further, by exchanging the roles of $(\hat{x},\hat{p} )$ and $(\hat{y},\hat{q} )$ and then repeating the construction above one obtains another linearly independent vector $p_2 \in T_{w_{1}} W_1$ and a smooth  reparametrisation $\alpha_2(s)$, $s\in (-\epsilon_2, \epsilon_2)$ such that $\gamma_{(w_{1},p_2)}(s) = \sigma_{(w_{1},p_2)}( \alpha_2(s) )$, for $s\in (-\epsilon_2, \epsilon_2)$. By Theorem \ref{oo1}, we have $g_1|_{W_1}=c\s F^* g_2$. Consequently by Proposition \ref{conformii} we have that $c>0$ is constant. 
Since $F^* g_2 = g_1$ on $V$ we conclude that $c=1$. Thus $F$ is an isometry.
\end{proof}

\newpage

%
%
\appendix

\section{Auxiliary lemmas}\label{appendix-al}

\subsection{Lemmas used in the proof of Proposition \ref{sofisti}}\label{appendixa1}


\begin{replemma}{Lambda_R-lemma}
Let $Y_1$ and $Y_2$ and $U$ be as in Definition~\ref{intersection_coords} and adopt also the associated notation. 
Define
\[
\Lambda_R =   N^* ( \bigcup_{x\in U}  \{ x\} \times  \PP_xU \times \PP_xU  ).
\]
The submanifold $\Lambda_R$ of $T^*( U \times \PP U \times \PP U) \setminus \{0 \}$ equals the set
\begin{align*}
\Lambda_R &= \{ 
\big(x, y, z,p, q \, ; \, \xi^x, \xi^y, \xi^z,\xi^p,\xi^{q} \big) \in \ T^*(U \times \PP U \times  \PP U) \setminus \{0\}  : \\
&  \quad\quad\quad\quad\quad\quad\quad\quad\quad\quad\quad\quad\quad  \xi^x + \xi^y +\xi^z = 0 , \ \xi^p = \xi^{q} = 0, \ x=y=z  \}.
\end{align*}

\tbl{Then we have}
\begin{align}\label{lambda_R_prime_a}
\Lambda_R' &
  =\{ \big((x ; \xi^x), (y, z,p, q \, ; \, \xi^y, \xi^z,\xi^p,\xi^{q} )\big) \in \ T^*U \times T^* ( \PP U \times  \PP U) \setminus \{0\}   : \nonumber  \\
 &  \quad\quad\quad\quad\quad\quad\quad\quad\quad\quad\quad\quad\quad \xi^x =\xi^y+ \xi^z \neq 0, \
   \xi^p =  \xi^{q} =0 , \ 
    x=y=z \}.
\end{align}
The spaces $\Lambda'_R \times ( N^*[Y_1 \times  Y_2] ) $ and $T^*U \times \text{diag}\, T^*( \PP M \times \PP M) $ intersect transversally in $T^*U \times T^*(\PP M \times \PP M) \times  T^*(\PP M \times \PP M)$.
\end{replemma}
\begin{proof}

First notice that the manifold $T\Big( \bigcup_{x\in U}  \{ x\} \times  \PP_xU \times \PP_x U\Big)$ can be written as
\begin{equation}\label{KKKKKLLL123}
\mathcal{V}_1:=\{   (x,y,z, p,q \, ; \,  \dot{x},\dot{y}, \dot{z} , \dot{p},\dot{q} ) 
:  \dot{x} = \dot{y}=\dot{z}, \ x = y=z \} .
\end{equation}
Note the $\mathcal{V}_1$ is a space of dimension $6n$.
Let us consider the subspace of $T^*(U\times \PP U \times \PP U)$
\[
 \mathcal{V}_2:=\{ (x,y,z, p,q\ ; \ \xi^x, \xi^y , \xi^z, \xi^p,\xi^{q} )  : x=y=z, \ \xi^x + \xi^y +\xi^z = 0 , \ \xi^p = \xi^{q} = 0  \}.
\]
%

Note that a vector in $\mathcal{V}_1$ paired with a covector in $\mathcal{V}_2$ yields zero since
\[
 (\xi^x,\xi^y,\xi^z,\xi^p,\xi^q)\cdot (\dot{x},\dot{y},\dot{z},\dot{p},\dot{q})=\dot{x}\,\xi^x+\dot{y}\,\xi^y+\dot{z}\,\xi^z=\dot{x}(\xi^x + \xi^y +\xi^z)=0.
\]
Thus we have that $\mathcal{V}_2\subset \Lambda_R$. Note that the fibers of $\mathcal{V}_1$ are of dimension $5n-2n=3n$. Note also that the fibers of $\mathcal{V}_2$ have dimension $5n-3n=2n$. We thus have $\mathcal{V}_2=\Lambda_R$, since dimensions of a fiber of the conormal bundle of $\cup_{x\in U}  \{ x\} \times  \PP_xU \times \PP_x U$ is the same as the codimension of a fiber of $\mathcal{V}_1$. 
In conclusion, the coordinate expressions for $\Lambda_R$ and $\Lambda_R'$ hold.


Next we prove the transversality claim. First, fix the notation
\begin{align*}
X &= T^*U \times T^*( \PP M \times \PP M)\times T^*( \PP M \times \PP M)\\
L_1&= \Lambda'_R \times (  N^*[Y_1 \times  Y_2])\\
L_2&= T^*U \times \textrm{diag}\,T^*( \PP M \times \PP M).
\end{align*}
We want to show that the linear spaces $L_1$ and $L_2$ intersect transversally in $X$; that is for all $\lambda \in L_1\cap L_2$, 
 \begin{align}\label{jksadss}
 \dim (T_\lambda L_1 + T_\lambda L_2) &=
  \dim T_\lambda X.
%
 \end{align}

Let $\lambda\in L_1 \cap L_2$. Since $Y_1$ and $Y_2$ satisfy the intersection property, there exists \antticomm{local parametrisation $(\tilde{x}',\tilde{x}'',\tilde{p}',\tilde{p}'') \mapsto (\tilde{x}',\tilde{x}'',\theta_2(x') + \tilde{p}', \theta_1(x'') + \tilde{p}'') $ (see \eqref{repara}) on $TU$ such that  
\begin{equation}\label{turhautuminen11}
\begin{split}
Y_1 \cap TU &= \{ (x,p)   \in TU  : \tilde{x}' = 0  , \ \tilde{p}' = 0, \ \tilde{p}'' =0 \},  \\
Y_2 \cap TU &= \{  (x,p) \in TU : \tilde{x}'' = 0, \ \tilde{p}''= 0, \ \tilde{p}' =0 \}.
\end{split}
\end{equation}

We redefine $(x',x'',p',p'')$ as these coordinates. 
Again, we denote $(x,\xi)=(x',x'')$, $p= (p',p'') $ and similarly $\xi = (\xi',\xi'') $ for the associated covectors etc. 
In these coordinates, the elements of $\Lambda_R'= \mathcal{V}_2'$ are of the similar form as above in the sense that each element is of the form
\[
\begin{split}
 \big((x;\xi^y+\xi^z) ,(y,z,p,q,\xi^y,\xi^z,\xi^p,\xi^q) \big), \quad x=y=z, \quad \xi^x = \xi^y + \xi^z, \quad \xi^p = 0 = \xi^q
 \end{split}
\]
in the canonical coordinates. 
}
We next compute the induced local expressions of $L_1$, $L_2$ and $L_1\cap L_2$ with respect to the coordinate system in \eqref{turhautuminen11}.

If $\lambda \in L_1$, $\lambda$ has the local coordinate form
\begin{equation}
\lambda = ((x,\xi^x, y, p, \xi^y,0, z,q,\xi^z, 0),\alpha,\beta),\label{L1-coords}
\end{equation}
where $(\alpha,\beta) \in N^*[Y_1 \times Y_2]$, $(x,\xi^x)\in T^*U$, and $(y,p,\xi^y,0),(z,q,\xi^z,0)\in T^*(\PP U)$ are such that $y=z=x$ and $\xi^y+\xi^z=\xi^x$.

If $(\alpha,\beta)\in N^*[Y_1\times Y_2]$, then $(\alpha,\beta)$ is nonzero and must satisfy \antticomm{in coordinates 
\begin{align}\label{conormaly1y2}
\alpha &=(0,x'',0,0 \ ; (\xi^x)',0, (\xi^p)',(\xi^p)'') \\
 \beta &= (y',0,0,0 \ ; 0,(\xi^y)'', (\xi^q)',(\xi^q)''),
\end{align}
}
where one (but not both) of the components is allowed to be zero.

The
expression for $\lambda \in L_2$ is
\begin{equation}
\lambda = (x,\xi^x, y, p, z, q,\xi^y, \xi^p, \xi^z,\xi^q, y, p,z, q, \xi^y, \xi^p,\xi^z,\xi^q).\label{L2-coords}
\end{equation}
where $x,y,z \in U$.

Using \eqref{L1-coords}, \eqref{conormaly1y2}, and \eqref{L2-coords}, we obtain that points $\lambda \in L_1\cap L_2$ are described by
\begin{align}
\lambda &= (0,\xi^x, \gamma,\gamma), \label{L3-coords}
\end{align}
where 
\antticomm{
\begin{align*}
\xi^x &= ((\xi^x)',(\xi^x)''),\\
\gamma &= (\underbrace{0, 0}_{x},\underbrace{0,0}_{p} , \underbrace{0, 0}_{y},\underbrace{0,0}_{q};\underbrace{(\xi^x)',0}_{\xi^x}, \underbrace{0,(\xi^x)''}_{\xi^y},\underbrace{0,0}_{\xi^p}, \underbrace{0,0}_{\xi^q}),\\
\end{align*}
}
Indeed, if $\lambda  \in L_1\cap L_2$, we must have $(x',x'') = (y',0)=(0,z'')$ which is only satisfied if $x=y=z=0$. From \eqref{L1-coords}, we have $\xi^p=0=\xi^q$ in \eqref{L2-coords}. Since $\xi^x=\xi^y+\xi^z$ in \eqref{L1-coords}, from \eqref{conormaly1y2} and \eqref{L2-coords} we find $((\xi^x)',(\xi^x)'') = ((\xi^z)',(\xi^y)'')$. From \eqref{turhautuminen11}, $(p',p'') =\antticomm{ (0,0) = (q',q'')}
$.

From the above coordinate expressions, we now compute the dimensions of $L_1$, $L_2$, and $L_1\cap L_2$. First note that
\begin{align*}
\dim(X)= \dim(T^*U) + \dim( T^*( \PP M \times \PP M) )+ \dim( T^*( \PP M \times \PP M)) = 18n 
\end{align*}
Similarly one computes $\dim(L_2) = 10n$.
%
The expression \eqref{L3-coords} shows that $\dim(L_1\cap L_2)=n$.

Therefore, for $\lambda\in L_1\cap L_2$,
\[
\dim (T_\lambda L_1) + \dim ( T_\lambda L_2) - \dim T_\lambda (L_1 \cap L_2) = 9n +10n -n = 18n = \dim(T_\lambda X). 
\]
This shows that $L_1$ is transverse to $L_2$ in $X$.
 
\end{proof}

\section{Existence theorems for Vlasov and Boltzmann Cauchy problems}\label{appendix-ss}
In this section, the space $(M,g)$ is assumed to be globally hyperbolic
$C^\infty$-Lorentzian manifold. 
By $\gamma_{(x,p)}:(-T_1,T_2)\to M$ we denote the inextendible geodesic which satisfies
\begin{equation}
\gamma_{(x,p)}(0)=x \text{ and } \dot{\gamma}_{(x,p)}(0)=p.                                                                                                                            
\end{equation}
We do not assume that $(M,g)$ is necessarily geodesically complete. Therefore, we might have that $T_1<\infty$ or $T_2<\infty$.
We will repeatedly use the fact that if 
$f(x,p)$ is a smooth function on $\OVS M$ whose support on the base variable $x\in M$ is compact, 
then the map 
\begin{equation}\label{integrate_over_flow2}
(x,p) \mapsto \int_{-\infty}^0 f( \gamma_{(x,p)}(t) , \dot\gamma_{(x,p)}(t) )  dt \quad \text{on} \quad (x,p) \in \OVS M
\end{equation}
 is well defined. This is because on a globally hyperbolic Lorentzian manifold any causal geodesic $\gamma$ exits a given compact set $K_\pi$ permanently after finite parameter times. That is, there are parameter times $t_1,t_2$ such that $\gamma(\{t<t_1\}),\gamma(\{t>t_2\})\subset M\setminus K_\pi$. Thus the integral above is actually an integral of a smooth function over a finite interval. Further, since $f$ and the geodesic flow on $(M,g)$ are smooth the map in~\eqref{integrate_over_flow2} is smooth.  If $(M,g)$ is not geodesically complete and if $\gamma_{(x,p)}:(-T_1,T_2)\to M$, we interpret the integral above to be over $(-T_1,0]$. 
 We interpret similarly for all similar integrals in this section without further notice.

We record the following lemma.
\begin{lemma}\label{param_length_lemma}
 Let $(M,g)$ be a globally hyperbolic Lorentzian manifold, let $X$ be a compact subset of $\OVS M$ and let $K_{\pi}$ be a compact subset of $M$. 
 Then the function $\ell:\OVS M \to \R$
 \[
 \ell(x,p)= \antticomm{\max} \{s\geq 0: \gamma_{(x,p)}(-s)\in K_{\pi} \}.
\]
\antticomm{is well defined},   
upper semi-continuous and \antticomm{there is the maximum}
\begin{equation}\label{maxl}
l_0 = \max \big\{ \ell(y,q): (y,q
)\in X \}<\infty.
\end{equation}
In addition, if $\lambda >0$ then $\ell(x,\lambda p ) = \lambda^{-1} \ell (x,p)$.
\end{lemma}
\begin{proof}
 Since the globally hyperbolic manifold $(M,g)$ is \antticomm{causally} disprisoning and causally pseudoconvex, any of its causal geodesics exits the compact set $K_{\pi}$ permanently after finite parameter times \antticomm{in the corresponding inextendible domain}, see e.g.~\cite[Proposition 1]{MR1216526} and~\cite[Lemma 11.19]{beem_book}. 
%
Hence, $\ell(x,p)<\infty$ \antticomm{is well defined}  for all $(x,p)\in \OVS M$.

The upper semi-continuity of $\ell$ follows from the global hyperbolicity of $(M,g)$ and the compactness of $K_\pi$.


  The maximum in~\eqref{maxl} exists since $\ell$ is upper semi-continuous and $X$ is compact. Since $\gamma_{(x,\lambda p)} (s) = \gamma_{(x, p)} (\lambda s)$, for all $\lambda\in \R$, we have that $\ell(x,\lambda p ) = \lambda^{-1} \ell (x,p)$, for $\lambda>0$. 
\end{proof}



\begin{reptheorem}{vlasov-exist}
Assume that $(M, g)$ is a globally hyperbolic $C^\infty$-Lorentzian manifold. Let $\mathcal{C}$ be a Cauchy surface of $(M,g)$, $K\subset \OVS \mathcal{C}^+$ be compact and $k\ge 0$.
Let also $f\in C_K^k(\OVS M)$. Then, the problem
\begin{alignat}{2}\label{vlasov11}
\mathcal{X}u(x,p)&= f(x,p) \quad &&\text{on} \quad \OVS M \nonumber \\
u(x,p)&= 0 \quad &&\text{on} \quad  \OVS \mathcal{C}^- 
\end{alignat}
has a unique solution $u$ in $C^k ( \OVS M )$. 
 In particular, if $Z\subset \OVS M $ is compact, there is a constant $c_{k,K,Z}>0$ such that  
\begin{align} \norm{u|_Z}_{C^k ( Z)}\le c_{k,K,Z}||f||_{C^{k}
( \OVS M )}. \label{est-vlasov1_appdx}
\end{align}
If $k=0$, the estimate above is independent of $Z$: 
\[
 \norm{u}_{C ( \OVS M)}\le c_{K}||f||_{C( \OVS M )}.
\]

%

\end{reptheorem}

\begin{proof}
Let us denote by $K_{\pi}=\pi(K)$ the compact set containing $\pi(\supp{f})$. 
Let $(x,p)\in \OVS M$ and $f\in C_c^k(\OVS \mathcal{C}^+)$. 
Evaluating~\eqref{vlasov11} at $(\gamma_{(x,p)}(s),\dot\gamma_{(x,p)}(s))$ reads
\[
 (\XX u)(\gamma_{(x,p)}(s),\dot\gamma_{(x,p)}(s))=f(\gamma_{(x,p)}(s),\dot\gamma_{(x,p)}(s)).
\]
Since $\XX$ is the geodesic vector field we have for all $s$ that
\[
 (\XX u)(\gamma_{(x,p)}(s),\dot\gamma_{(x,p)}(s))=\frac{d}{ds}u(\gamma_{(x,p)}(s),\dot\gamma_{(x,p)}(s)). 
\]
By integrating in $s$, 
we obtain
\begin{equation}\label{vlasov_integrated_form_appendix}
 u(x,p)  = \int_{-\infty}^0 f(\gamma_{(x,p)} (s), \dot\gamma_{(x,p)} (s) )ds. 
\end{equation} 
Here we used that $f(\gamma_{(x,p)} (s), \dot\gamma_{(x,p)} (s) )$ vanishes for $s<-\ell(x,p)$ by Lemma~\ref{param_length_lemma}, where 
 \[
 \ell(x,p)=\max \{s\geq 0: \gamma_{(x,p)}(-s)\in K_{\pi} \}.
\]
\antticomm{Indeed, any inextendible causal geodesic in a globally hyperbolic $(M,g)$ leaves permanently the compact set $\pi \text{supp}(f)$ (see page \pageref{integrate_over_flow2}). This holds even without assumptions on completeness. }

We verify that $u$ is a solution to~\eqref{vlasov11}. 
Note that if $(y,q)=\big(\gamma_{(x,p)} (s), \dot{\gamma}_{(x,p)} (s)\big)$, then
\[
 \gamma_{(y,q)}(z)=\gamma_{(x,p)}(z+s)  \text{ and } \dot{\gamma}_{(y,q)} (z)=\dot{\gamma}_{(x,p)}(z+s).
\]
It follows that
\begin{align*}
 u(\gamma_{(x,p)}(s),\dot{\gamma}_{(x,p)}(s))&=\int_{-\infty}^0f\big(\gamma_{(x,p)}(z+s),\dot{\gamma}_{(x,p)}(z+s)\big)dz =\int_{-\infty}^sf\big(\gamma_{(x,p)}(z),\dot{\gamma}_{(x,p)}(z)\big)dz
\end{align*}
and consequently
\[
 \XX u(x,p)=\frac{d}{ds}\Big|_{s=0}u(\gamma_{(x,p)}(s),\dot{\gamma}_{(x,p)}(s))=f\big(\gamma_{(x,p)}(0),\dot{\gamma}_{(x,p)}(0)\big)=f(x,p).
\]

If $(x,p)\in \OVS \mathcal{C}^-$, then $u(x,p)=0$ by the integral formula~\eqref{vlasov_integrated_form_appendix} and the fact that $f\in C_c^k(\OVS \mathcal{C}^+)$. We have now shown that a solution $u$ to~\eqref{vlasov11} exists. The solution $u$ is unique since it was obtained by integrating the equation~\eqref{vlasov11}. 

Next we prove the estimate \eqref{est-vlasov1_appdx}. 
We have by the representation formula~\eqref{vlasov_integrated_form_appendix} for the solution $u$ that  
\begin{equation}\label{sup_restriction}
 \sup_{(x,p)\in \OVS M}\abs{u(x,p)}=\sup_{(x,p)\in \OVS K_\pi}\abs{u(x,p)}.
\end{equation}
The equation~\eqref{sup_restriction} holds since $\pi(\supp{f})$ is properly contained in $K_{\pi}$. 
Let $e$ be some auxiliary \tbl{smooth} Riemannian metric on $M$ and let $SK_{\pi}\subset TM$ be the unit sphere bundle with respect to $e$ over $K_{\pi}$. Let us also denote
\[
 X=SK_{\pi}\cap \OVS K_{\pi}
\]
the bundle of future directed causal (with respect to $g$) vectors that have unit length in the Riemannian metric $e$. Since $X$ is a closed subset of the compact set $SK_{\pi}$, we have that $X$ is compact. By Lemma~\ref{param_length_lemma} we have that
\[
l_0 = \max \big\{ \ell(y,q): (y,q
)\in X \}
\]
exists.

Let us continue to estimate $\abs{u(x,p)}$ for $(x,p)\in \OVS K_{\pi}$. If $(x,p)\in \OVS K_{\pi}$, then there is $\lambda>0$ such that $(x,\lambda^{-1} p)\in X$. Let us denote $q=\lambda^{-1}\, p\in SK_{\pi}$. We have that
\begin{align}\label{u_scaling_identity}
 u(x,p)  &= \int_{-\infty}^0 f(\gamma_{(x,\lambda q)} (s), \dot\gamma_{(x,\lambda q)} (s) )ds =\int_{-\infty}^0 f(\gamma_{(x,q)} (\lambda s), \lambda\dot\gamma_{(x,q)} (\lambda s) )ds \nonumber \\
 &=\frac{1}{\lambda}\int_{-\infty}^0 f(\gamma_{(x,q)} (s), \lambda\dot\gamma_{(x,q)} (s) )ds=\frac{1}{\lambda}\int_{-l_0}^0 f(\gamma_{(x,q)} (s), \lambda\dot\gamma_{(x,q)} (s) )ds.
\end{align}
Here we used
\[
 \gamma_{(x,\lambda p)}(z)=\gamma_{(x,p)}(\lambda z) \quad  \text{ and } \quad \frac{d}{dz} \gamma_{(x,\lambda p)}(z)=\lambda \dot\gamma_{(x,p)}(\lambda z).
\]

Let us define two positive real numbers
\begin{align*}
 C&=\max_{s\in [0,l_0]}\max_{(x,q)\in X}\abs{\dot\gamma_{(x,q)}(-s)}_e<\infty \\ 
 R&=\inf \{r>0: f|_{B_e(0,r)\subset T_xM}=0 \text { for all } x\in K_{\pi}\}>0.
\end{align*}
Here for $x\in K_{\pi}$, the set $B_e(0,r)$ is the unit ball of radius $r$ with respect to the Riemannian metric $e$ in the tangent space $T_xM$.
 The constant $R$ is positive since $f$ has compact support in $\OVS M$ by assumption. 
Let us define
\[
 \lambda_{\text{min}}:=\frac{R}{C}>0.
\]
Then, if $\lambda < \lambda_{\text{min}}=\frac{R}{C}$, we have for $(x,q)\in X$ that
\[
 \int_{-l_0}^0 f(\gamma_{(x,q)} (s), \lambda\dot\gamma_{(x,q)} (s) )ds=0,
\]
since in this case $\abs{\lambda\dot\gamma_{(x,q)}(s)}< R$ for all $s\in [-l_0,0]$. It follows that for all $\lambda>0$ we have that
\[
 \frac{1}{\lambda}\left|\int_{-l_0}^0 f(\gamma_{(x,q)} (s), \lambda\dot\gamma_{(x,q)} (s) )ds\right|\leq \frac{\textcolor{blue}{l_0}}{\lambda_{\text{min}}}\norm{f}_{C(\OVS M)}.
\]
Finally, combining the above with~\eqref{sup_restriction} and~\eqref{u_scaling_identity} shows that
\begin{align*}
 \norm{u}_{C(\OVS M)}&=\sup_{(x,p)\in \OVS M}\abs{u(x,p)}
 \leq \sup_{\lambda>0}\sup_{(x,q)\in X}\frac{1}{\lambda}\left|\int_{-l_0}^0 f(\gamma_{(x,q)} (s), \lambda\dot\gamma_{(x,q)} (s) )ds\right| \\
 &\leq  \frac{l_0}{\lambda_{\text{min}}}\norm{f}_{C(\OVS M)}.
\end{align*}

{Let $Z\subset \OVS M$ be a compact set.}
We next show that
\[
\norm{u}_{C^k (Z)}\le c_{k,K,Z}\norm{f}_{C^{k}( \OVS M )} 
\]
for $k\geq 1$. 

We have proven that this estimate holds for $k=0$. We prove the claim for $k>0$.  Let $\p$ denote any of the partial differentials $\p_{x^a}$ or $\p_{p^a}$ in canonical coordinates of the bundle $TM$. We apply $\p$ to the formula~\eqref{vlasov_integrated_form_appendix} of the solution $u$ to obtain
\[
 \p u(x,p)  = \int_{-\infty}^0 \left[\frac{\p f}{\p x^\alpha}(\gamma_{(x,p)} (s), \dot\gamma_{(x,p)} (s) )\p \gamma_{(x,p)}^\alpha(s) +\frac{\p f}{\p p^\alpha}(\gamma_{(x,p)} (s), \dot\gamma_{(x,p)} (s) )\p \dot{\gamma}_{(x,p)}^\alpha(s)\right]ds.
\]
Since $\frac{\p f}{\p x^\alpha}$ and $\frac{\p f}{\p p^\alpha}$ have the same properties as $f$\tbl{, and the smooth coefficients $\p \gamma_{(x,p)}^\alpha$ and $\p \dot{\gamma}_{(x,p)}^\alpha$ are uniformly bounded for $(x,p)\in Z\subset \OVS M$, we may apply the proof above to show that
\[
 \norm{u}_{C^{1} (Z)}\le c_{1,K,Z}\norm{f}_{C^{1}( \OVS M )}.
\]
}
The proof for $k\geq 2$ is similar.
\end{proof}


By using the solution formula~\eqref{vlasov_integrated_form_appendix},
\[
 u(x,p)  = \int_{-\infty}^0 f(\gamma_{(x,p)} (s), \dot\gamma_{(x,p)} (s) )ds, 
\]
in the proof of Theorem~\ref{vlasov-exist}, we have the following result for Cauchy problems for the equation $\XX u =f$ restricted to $\SP M$ and $L^+ M$. We denote by $L^+\mathcal{C}^{\pm}$ the bundle of future-directed lightlike vectors in the future $\mathcal{C}^+$ or past $\mathcal{C}^-$ of a Cauchy surface $\mathcal{C}$.
\begin{corollary}
 Assume as in Theorem~\ref{vlasov-exist} and adopt its notation. \tbl{Then for the compact set $K\subset \OVS \mathcal{C}^+$,}  the Cauchy problems
 \begin{alignat}{2}\label{vlasov11_timelike}
\mathcal{X}u(x,p)&= f(x,p) \quad &&\text{on} \quad \SP M \nonumber \\
u(x,p)&= 0 \quad &&\text{on} \quad  \SP \mathcal{C}^-,
\end{alignat}
and
\begin{alignat}{2}\label{vlasov11_lightlike}
\mathcal{X}u(x,p)&= f(x,p) \quad &&\text{on} \quad L^+ M \nonumber \\
u(x,p)&= 0 \quad &&\text{on} \quad  L^+ \mathcal{C}^-,
\end{alignat}
have continuous solution operators $\mathcal{X}^{-1}: C^k_K(\SP M)\to C^k(\SP M)$ and $\mathcal{X}_L^{-1}:C^k_K(L^+ M)\to C^k(L^+ M)$ respectively. 
\end{corollary}
Note that we slightly abused notation by denoting by $\mathcal{X}^{-1}$ the solution operator 
to both Cauchy problems~\eqref{vlasov11} and~\eqref{vlasov11_timelike}. Here $C_K^k(\SP M)$ and $C_K^k(L^+M)$ are defined similarly as $C_K^k(\OVS M)$. Since $\SP M$ and $L^+ M$ are manifolds without boundary, we are able to use standard results to extend the problems~\eqref{vlasov11_timelike} and~\eqref{vlasov11_lightlike} for a class of distributional sources $f$.

\begin{lemma}\label{vlasov_ext_conormal}
Assume that $(M, g)$ is a globally hyperbolic $C^\infty$-Lorentzian manifold. Let $\mathcal{C}$ be a Cauchy surface of $(M,g)$
 
 \noindent \textbf{(1)} The solution operator $\mathcal{X}^{-1}$ to the Cauchy problem~\eqref{vlasov11_timelike} on $\SP M$ has a unique continuous extension to 
$f\in \{h\in \mathcal{D}'(\SP M): WF(h)\cap N^*(\SP \mathcal{C})=\emptyset, \tbl{\ h=0 \text{ in } \SP\mathcal{C}^-}\}$. 
If $S$ is a submanifold of $\SP \mathcal{C}^+$, $f\in I^l(\SP M; N^*S)$, $l\in \R$, 
then we 
have that $u=\IX f$ satisfies $\chi u\in I^{l-1/4}(\SP M; N^*K_S)$ for any $\chi\in C_c^\infty(\SP M)$ with $\text{supp}(\chi)\subset\subset \SP M\setminus{S}$. 

 \noindent \textbf{(2)} The solution operator $\mathcal{X}_L^{-1}$ to the Cauchy problem~\eqref{vlasov11_lightlike} on $L^+M$ has a unique continuous extension 
to $\{h\in \mathcal{D}'(L^+ M): WF(h)\cap N^*(\SP \mathcal{C})=\emptyset, \tbl{\ h=0 \text{ in } L^+\mathcal{C}^-}\}$. If $S$ is a submanifold of $L^+ \mathcal{C}^+$, $f\in I^l(L^+M; N^*S)$, $l\in \R$, 
then we 
have that $u=\IX f$ satisfies $\chi u \in I^{l-1/4}(L^+M; N^*K_S)$ for any $\chi\in C_c^\infty(L^+ M)$ with $\text{supp}(\chi)\subset\subset L^+ M\setminus{S}$. 
\end{lemma}
\begin{proof}
Let us first consider the solution operator $\IX$ to~\eqref{vlasov11_timelike}. We will refer to~\cite[Theorem 5.1.6]{duistermaat2010fourier}. To do that, we consider $\SP M$ as 
$\R\times \SP \mathcal{C}$ by using the flowout parametrization $\R\times \SP \mathcal{C} \to \SP M$ given by
\begin{equation*}
   (s, (x,p))\mapsto  \dot\gamma_{(x,p)}(s), \quad s\in \R, \ (x,p) \in  \SP \mathcal{C}. 
\end{equation*}
Also, by reviewing the proof of Lemma~\ref{uuiisok}, we conclude that $\XX$ is strictly hyperbolic with respect to $\SP \mathcal{C}$. Then, by~\cite[Theorem 5.1.6]{duistermaat2010fourier}, the problem~\eqref{vlasov11_timelike} has a unique solution for $f\in \{h\in \mathcal{D}'(\SP M): WF(h)\cap N^*(\SP \mathcal{C})=\emptyset, \tbl{\ h=0 \text{ in } \SP\mathcal{C}^-}\}$. By~\cite[Remarks after Theorem 5.1.6]{duistermaat2010fourier} the solution operator $\IX$ to~\eqref{vlasov11_timelike} extends continuously to $\{h\in \mathcal{D}'(\SP M): WF(h)\cap N^*(\SP \mathcal{C})=\emptyset, \tbl{\ h=0 \text{ in } \SP\mathcal{C}^-}\}$. 

If $S\subset \PP^+ \mathcal{C}^+$ and $f\in I^l(\SP M; N^*S)$, then we have $f\in \{h\in \mathcal{D}'(\SP M): WF(h)\cap N^*(\SP \mathcal{C})=\emptyset\}$, because 
%
\[
WF(f) \cap N^* (\SP \mathcal{C} )  \subset N^*S \cap N^* (\SP \mathcal{C}) = \emptyset. 
\]

By using the definitions of the sets $C_0$ and $R_0$, which appear in~\cite[Theorem 5.1.6]{duistermaat2010fourier}, we obtain
\begin{align*}
 C_0\circ R_0&=\{(( x,p \; \xi), (z,w \;  \lambda) ) \in T^*(\SP M) \times T^*(\SP M) \setminus \{ 0\}: (x,p\; \xi) \text{ on the bicharacteristic } \\ 
&\quad \text{ strip through } \lambda \in T_{(z,w)}^*\SP M \text{ with } (z,w)\in \SP \mathcal{C}
\ \text{and} \ \sigma_{-i\XX} ( z,w \;  \lambda ) = 0\}.
\end{align*}
Here $\sigma_{-i\XX}$ is the principal symbol of $\XX$, see~\eqref{principal_symbol_of_X}. 
Let $\chi_1\in C_c^\infty(\SP M)$ be such that $\text{supp}(\chi_1)\subset\subset \SP M\setminus S$. We choose $\chi_2\in C_c^\infty(\SP M)$ such that  $\chi_2$ equals $1$ on a neighborhood of $S$ and $\supp{\chi_2}\subset \SP \mathcal{C}^+$  and such that $\text{supp}(\chi_1)\cap \text{supp}(\chi_2)=\emptyset$. 
Let us denote $A=\chi_1\mathcal{I}$ and $B=\chi_2\mathcal{I}$,
where $\mathcal{I}$ is the identity operator. If we consider $A$ and $B$ as pseudodifferential operators of class $\Psi^0_{1,0}(\SP M)$, we have that
 $(WF(A) \times WF(B) )\cap[ \text{diag}\s (T^*(\PP^+ M))  \cup (C_0 \circ R_0) ] = \emptyset$. 
We write
\[
 \chi_1 u = \chi_1\IX \chi_2 f+\chi_1\IX (1-\chi_2) f=A\IX B f, 
\]
where we used $(1-\chi_2) f=0$ so that $\IX (1-\chi_2)f=0$.
By~\cite[Theorem 5.1.6]{duistermaat2010fourier}, we have that 
\[
A\IX B\in I^{-1/4} (\PP^+ M,\PP^+ M; \Lambda_\XX). 
\]

The flowout of the conormal bundle over $S$ under $\XX$ is the conormal bundle of the geodesic flowout of $S$. That is
\[
 \Lambda_\mathcal{X} \circ N^*S  =   N^* K_{S}.
\]
Finally, by applying~\cite[Theorem 2.4.1, Theorem 4.2.2]{duistermaat2010fourier}, we have 
\begin{equation}\label{prprp1}
\chi_1 v\in I^{l - 1/4 }(\PP^+ M; N^* K_{S}).
\end{equation}
Renaming $\chi_1$ as $\chi$ concludes the proof of \textbf{(1)}. The proof of \textbf{(2)} is a similar application of~\cite[Theorem 5.1.6]{duistermaat2010fourier} by using the flowout parametrization for $L^+M$ given by $(s, (x,p))\mapsto  \dot\gamma_{(x,p)}(s)$, $(x,p)\in L^+ \mathcal{C}$ and $s\in \R$.
\end{proof}

%
%

Next we prove that the Boltzmann equation has unique small solutions for small enough sources. Before that, we give an estimate regarding the collision operator in the following lemma. Following our convention of this section, the integral in the statement of the lemma over a geodesic parameter is interpreted to be over the largest interval of the form $(-T,0]$, $T>0$, where the geodesic exists. 
\begin{lemma}\label{lislem}
Let $(M,g)$ be a globally hyperbolic Lorentzian manifold and let  $\COLOP$ be a collision operator with an admissible collision kernel $A : \Sigma \rightarrow \R$ in the sense of Definition~\ref{good-kernels}.  
 Then there exists a constant $C_A>0$ such that
\[
\left| \int_{-\infty}^0 \COLOP [v,u]( \gamma_{(x,p)}(s),\dot\gamma_{(x,p)}(s)) ds \right| \leq C_A \|u\|_{C(\OVS M )}\|v\|_{C(\OVS M)}
\]
for every $(x,p) \in \OVS M$ and $u,v \in C_b(\OVS M )$.
\end{lemma}
\begin{proof}
Let us define a compact set $K_{\pi}:=\pi(\supp{A})$. 
Let $e$ be some auxiliary Riemannian metric on $M$. Let us denote
\[
 X=SK_{\pi}\cap \OVS K_{\pi}
\]
the bundle of future directed causal (with respect to $g$) vectors who have unit length in the Riemannian metric $e$. Since $X$ is a closed subset of the compact set $SK_{\pi}$, we have that $X$ is compact.
By Lemma~\ref{param_length_lemma}, we have that there exists the maximum 
\[
l_0 = \max \big\{ \ell(y,q): (y,q
)\in X \},
\]
where $\ell(x,p)=\sup \{s\geq 0: \gamma_{(x,p)}(-s)\in K_{\pi} \}$. 

Let us define another compact set $\mathcal{K}$ as 
\begin{equation*}\label{copa}
\mathcal{K}:=\big\{ (y,r) \in \OVS M: 
(y,r)=(\gamma_{(x,q)} (s) ,  \dot\gamma_{(x,q)} (s)), \ s\in [-\ell(x,q),0], \ (x,q)\in X
\big\}.
\end{equation*}
To see that $\mathcal{K}$ is compact, note that it is the image of the compact set $\{(s,(x,q)): s\in [-\ell(x,q),0], \ (x,q)\in X \}$ under the geodesic flow. The set $\{(s,(x,q)): s\in [-\ell(x,q),0], \ (x,q)\in X \}$ is compact since it is bounded by Lemma~\ref{param_length_lemma} and closed by the upper semi-continuity of $\ell$.  
Since the collision kernel is admissible, the function 
\[
\lambda  \mapsto F_{x,p}(\lambda) 
= \| A (x, \lambda p , \ccdot, \ccdot, \ccdot ) \|_{L^1(\Sigma_{x,\lambda p})},
\]
is by assumption continuously differentiable in $\lambda$ and attains its minimum value zero at $\lambda=0$. 
Thus, for any $(x,p)\in \OVS M$, we have that
\[
\lambda^{-1} F_{x,p}(\lambda) \longrightarrow \frac{d}{d \lambda}\Big|_{\lambda=0}F_{x,p}(\lambda)=0 
\]
as $\lambda \rightarrow 0$.  
Since the continuous function $(x,p)\mapsto \frac{d}{d \lambda}\big|_{\lambda=0}F_{x,p}(\lambda)$ on the compact set $\mathcal{K}$ is uniformly continuous, there is a constant $\lambda_0 >0$ such that
\begin{equation}\label{conseq_of_cond4}
\lambda^{-1} F_{y,r}(\lambda) 
\leq  1
\end{equation}
for  $0 < \lambda < \lambda_0$ and  for $(y,r)$ in the compact set $\mathcal{K}$.

%
%
Let $(x,p)\in \OVS M$ and write $(x,p)=(x,\lambda q)$. Recall from Lemma~\ref{param_length_lemma} that $\lambda \ell(x,p) =  \ell (x,\lambda^{-1}p)=\ell (x,q)$. We have that 
\begin{equation}\label{lisayhtalo}
\begin{split}
& \antticomm{\bigg|}  \int_{-\infty}^0 \COLOP [ u, v  ]   \big( \gamma_{(x,p)} (s) ,   \dot\gamma_{(x,p)} (s)  \big) ds \antticomm{\bigg|} \\
&=\antticomm{\bigg|} \int_{-\ell(x,p)}^0 \COLOP [ u, v  ]   \big( \gamma_{(x,\lambda q)} (s) ,   \dot\gamma_{(x,\lambda q)} (s)  \big) ds \antticomm{\bigg|} \\
&=\antticomm{\bigg|}  \int_{-\ell(x,p)}^0 \COLOP [ u, v  ]   \big( \gamma_{(x,q)} (\lambda s) ,  \lambda \dot\gamma_{(x,q)} (\lambda s)  \big) ds \antticomm{\bigg|} \\
&= \frac{1}{\lambda}\antticomm{\bigg|}  \int_{-\lambda\ell(x,p)}^0 \COLOP [ u, v  ]   \big( \gamma_{(x,q)} (s') ,  \lambda \dot\gamma_{(x,q)} (s')  \big) ds'\antticomm{\bigg|}  \\
&=\frac{1}{\lambda} \antticomm{\bigg|} \int_{-\ell(x,q)}^0 \COLOP [ u, v  ]   \big( \gamma_{(x,q)} (s') ,  \lambda \dot\gamma_{(x,q)} (s')  \big) ds' \antticomm{\bigg|} \\
&\leq 2 \| u\|_{C(\OVS M)}\| v\|_{C(\OVS M)} \frac{1}{\lambda} \int_{-l_0}^0  \|A(\gamma_{(x,q)} (s') ,  \lambda \dot\gamma_{(x,q)} (s'), \ccdot,\ccdot,\ccdot)\|_{L^1(\Sigma_{x,p})}  ds' \\
%
  &  \leq \begin{cases}
 2 l_0 \| u \|_{C(\OVS M)} \| v\|_{C(\OVS M)}  & \quad \text{if} \quad \lambda< \lambda_0 \\
  2\lambda_0^{-1}l_0  \| u \|_{C(\OVS M)}\| v\|_{C(\OVS M)}  &\quad \text{if} \quad \lambda  \geq \lambda_0 \\
  \end{cases} \\
\end{split}
\end{equation}
Here, for $\lambda \geq \lambda_0$, we used the condition (4) of the assumptions in the definition of an admissible kernel. For $\lambda\leq \lambda_0$ we used~\eqref{conseq_of_cond4}. We also did a change of the variable in the integration  as $s'=\lambda s$. This proves the claim. \qedhere

\end{proof}

\begin{reptheorem}{boltz-exist}

Let $(M, g)$ be a globally hyperbolic $C^\infty$-Lorentzian manifold of dimension $n\geq 3$. Let also $\mathcal{C}$ be a Cauchy surface of $M$ and $K\subset \OVS \mathcal{C}^+$ be compact. 
Assume that $A: \Sigma \to \mathbb{R}$ is an admissible collision kernel in the sense of Definition \ref{good-kernels}.
\antticomm{Moreover, assume that $\pi (\text{supp}A) \subset \mathcal{C}^+$. }

There are open neighbourhoods $B_1 \subset C_K(\OVS \mathcal{C}^+ )$ and $B_2\subset C_b( \OVS  M )$ of the respective origins such that if $f \in B_1$, the Cauchy problem
\begin{alignat}{2}\label{boltz11}
\mathcal{X}u(x,p)-\COLOP[u,u](x,p)&= f(x,p) \quad && \text{ on } 
\OVS M \nonumber \\
u(x,p)&= 0 \quad &&\text{ on } \OVS \mathcal{C}^-
\end{alignat}
has a unique solution $u\in B_2$. There is a constant $c_{A,K}>0$ such that
\[ \norm{u}_{C( \OVS M )}\le c_{A,K}\norm{f}_{C(\OVS M)}.\]
\end{reptheorem}

\begin{proof}
We integrate the equation~\ref{boltz11} along the flow of $\XX$ in $TM$ and then use the implicit function theorem in Banach spaces for the resulting equation. (Integrating the equation~\eqref{boltz11} avoids some technicalities regarding Banach spaces, which there would be in the application of the implicit function theorem without the integration.) 
We define the mapping
\[
F :  C_K(\OVS \mathcal{C}^+ ) \times C_b( \OVS M )    \rightarrow C_b( \OVS M )   , 
\]
by
\begin{equation}\label{integrated_bman}
F ( f, u )  (x,p) = u (x,p) - \int_{-\infty}^0 \COLOP [ u, u  ]  \big( \gamma_{(x,p)} (s) ,   \dot\gamma_{(x,p)} (s)  \big) ds  -  \int_{-\infty}^0 f (\gamma_{(x,p)} (s) ,  \dot\gamma_{(x,p)}(s)) ds.
\end{equation}
 Let us denote
\[
 Z:=\pi[\text{supp}(A)]\cup \pi[\text{supp}(f)] \antticomm{\subset \mathcal{C}^+ }    
\]                                             
where $\pi$ is the canonical projection.
\antticomm{In geodesically incomplete geometry, the line integrals above are interpreted as integrals over the associated lower half $(-T_1,0]$ of the inextendible domain $(-T_1,T_2)$, $T_1,T_2 \in (0,\infty]$ of the geodesic  $\gamma_{(x,p)}$. The tails of the integrals are zero.   
Indeed, as shown in  Lemma~\ref{param_length_lemma}, the integrals in the definition of $F$ above contribute only over  the bounded  interval $[-\ell(x,p),0] \subset (-T_1,0]$, where $\ell(x,p)=\max \{s\geq 0: \gamma_{(x,p)}(-s)\in Z \}  
  < \infty $ is the exit time inside the inextendible domain. 
}
In combination with Lemma \ref{lislem}, we have that $F$ is well-defined.

\antticomm{
Using $\gamma_{(\gamma_{x,p}(t),\dot\gamma_{x,p}(t))} (s) = \gamma_{x,p} (t+s)$ we compute
\[
\begin{split}
&\mathcal{X} \left(  \int_{-\infty}^0 \COLOP [ u, u  ]  \big( \gamma_{(\cdot,\cdot)} (s) ,   \dot\gamma_{(\cdot,\cdot)} (s)  \big) ds \right) (x,p) 
=\partial_t\left(  \int_{-\infty}^0 \COLOP [ u, u  ]  \big( \gamma_{(x,p)} (t+s) ,   \dot\gamma_{(x,p)} (t+s)  \big) ds \right) \Big|_{t=0} 
\\
&=\partial_t\left(  \int_{-\infty}^t \COLOP [ u, u  ]  \big( \gamma_{(x,p)} (s) ,   \dot\gamma_{(x,p)} (s)  \big) ds \right) \Big|_{t=0}
=  \COLOP [ u, u  ]  \big( \gamma_{(x,p)} (0) ,   \dot\gamma_{(x,p)} (0)  \big) =  \COLOP [ u, u  ](x,p) .\\
\end{split}
\]
The same argument yields also that
\[
\mathcal{X} \left( \int_{-\infty}^0 f (\gamma_{(\cdot,\cdot)} (s) ,  \dot\gamma_{(\cdot,\cdot)}(s)) ds \right)(x,p) = f(x,p) 
\]
Hence, $F(u,f) = 0$ implies that $u$ satisfies the first equation in \eqref{boltz11}. 

The second equation in \eqref{boltz11}, that is, the zero initial condition follows from the causality. Indeed, if  $(x,p)$ is a causal, future-pointing vector in the lower half $\mathcal{C}^-$, then points in $\{ \gamma_{x,p}(-s) :  s\geq 0 \}$ are in the causal past of $x$
and therefore also lie on the lower half. In particular, such points do not belong to $Z \subset \mathcal{C}^+$. 
As vectors with base-points outside $Z$ do not contribute to the collision term nor $f$, the  
initial condition for $u$ with $F(u,f) = 0$ follows by applying this to the integrals above.  
}


We apply the implicit function theorem for Banach spaces (see e.g.~\cite[Theorem 9.6]{Renardy}) to $F$ to obtain a solution $u$ if the source $f\in C_K(\OVS \mathcal{C}^+ )$ is small enough. 
First note that $F(0,0)=0$. Additionally, observe that for $u,\,v\in C_b( \OVS M  ) $ we have 
\begin{equation}\label{Q_binomial}
 \COLOP [ u+v, u +v ] =  \COLOP [ u, u ]  + \COLOP [v, u ]  + \COLOP [ u, v ]  + \COLOP [ v, v ]
\end{equation}
since $\COLOP$ is linear in both of its arguments.

Next, we argue that $F$ is continuously Frech\'et differentiable (in the sense of~\cite[Definition 9.2]{Renardy}). Let $u,\,v\in C_b( \OVS M )$ and $f,\,h\in C_K(\OVS \mathcal{C}^+ ) $. 
We have that 
\[
 F(f+h,u+v)-F(f,u) - L(f,u)(h,v)=- \int_{-\infty}^0 \COLOP [ v, v  ]  \big( \gamma_{(x,p)} (s) ,   \dot\gamma_{(x,p)} (s)  \big) ds,
\]

where 
\begin{align*}
L(f,u)(h,v) &:=  v (x,p) -  \int_{-\infty}^0 h (\gamma_{(x,p)} (s) ,  \dot\gamma_{(x,p)}(s)) ds\\
&\qquad - \int_{-\infty}^0 \COLOP [ u, v ]  \big( \gamma_{(x,p)} (s) ,   \dot\gamma_{(x,p)} (s)  \big) ds -\int_{-\infty}^0 \COLOP [ v, u ]  \big( \gamma_{(x,p)} (s) ,   \dot\gamma_{(x,p)} (s)  \big) ds. 
\end{align*}
It thus follows from Lemma~\ref{lislem} that
\[
\| F(f+h,u+v)-F(f,u) - L(f,u)(h,v) \|_{C_b(\OVS M )}^2 \leq  C_A \|v\|_{C_b(\OVS M)}^2. 
\]
We conclude that the Frech\'et derivative of $F$ at $(f,u)$ is given by $DF(f,u)(h,v) = L(f,u)(h,v) $. We have that $DF(f,u)$ is continuous 
\[
 DF(f,u): C_K(\OVS \mathcal{C}^+ )\times  C_b(  \OVS M  ) \to C_b(  \OVS M  ),
\]
by Lemma~\ref{lislem}. 
%
%
Finally, note that the Frech\'et differential in the second variable of $F$ at $(0,0)$
\[
 DF_2(0,0):C_b(  \OVS M  ) \to C_b(  \OVS M  ),
\]
given by $DF_2(0,0)=DF(0,0)(0,\ccdot)$, is just the identity map.
%

By the implicit function theorem in Banach spaces there exist open neighbourhoods $B_1\subset C_K(\OVS \mathcal{C}^+ )$ and $B_2\subset C_b(  \OVS M  )$ of the respective origins  
and a continuously (Frech\'et) differentiable map $T: V\to U$ such that for $f\in B_1$, the function $u=T(f)\in B_2$ is the unique solution to $F(f,u)=0$. 
Further, since $T$ is continuously differentiable, there exists $c_{A,K}>0$ such that 
\[
\norm{u}_{C_b(\OVS M ) }\le c_{A,K}\norm{f}_{C_K(\OVS M)}.
\]
This concludes the proof.
\end{proof}

Next we show that the source-to-solution mapping of the Boltzmann equation can be used to compute the source-to-solution mappings of the first and second linearizations of the Boltzmann equation.

\begin{replemma}{deriv}
Assume as in Theorem \ref{boltz-exist} and adopt its notation. Let $\Phi:B_1 \to B_2 \subset C_b(\OVS M)$, $B_1\subset C_K(\OVS \mathcal{C}^+ )$, be the source-to-solution map of the Boltzmann equation. 

The map $\Phi$ is twice Frech\'et differentiable at the origin of $C_K(\OVS \mathcal{C}^+ )$. If $f,\, h\in B_1$, then we have:

\begin{enumerate}
\item The first Frech\'et derivative $\Phi'$ of the source-to-solution map $\Phi$ at the origin satisfies 
\[ \Phi'(0;f)
= \Phi^{L}(f),\]
where $\Phi^L$ is the source-to-solution map of the Vlasov equation~\eqref{vlasov1}. 
\item The second Frech\'et derivative $\Phi''$ of the source-to-solution map $\Phi$ at the origin satisfies 
\[ \Phi''(0;f,h)
= \Phi^{2L}(f,h),\]
where $\Phi^{2L}(f,h)\in C(\OVS M)$ is the unique solution to  the equation 
\begin{alignat}{2}\label{second_lin_bman_appendix}
\mathcal{X}\Phi^{2L}(f,h)&= \COLOP[\Phi^L(f),\Phi^L(h)] + \COLOP[\Phi^L(h),\Phi^L(f)], \quad &&\text{ on }
\OVS M, \nonumber \\
\Phi^{2L}(f,h)&= 0, \quad && \text{ on }\OVS \mathcal{C}^-.
\end{alignat}
\end{enumerate}

\end{replemma}

\begin{proof}[Proof of Lemma \ref{deriv}]
%

{\it Proof of (1).} We adopt the notation of Theorem~\ref{boltz-exist}. Let $f\in C_K(\OVS \mathcal{C}^+ )$ and let $f_0\in B_1$. Then by Theorem~\ref{boltz-exist} there exists a neighbourhood $B_2$ of the origin in $C_b(\OVS M)$ and $\eps_0>0$ such that for all $\eps_0>\eps>0$ the problem 
\begin{align}\label{intergrate_bman}
\mathcal{X}u_{\epsilon}-\COLOP[u_{\epsilon},u_{\epsilon}]&= \epsilon f \quad \text{on } \OVS M \\
u_{\epsilon}&= 0 \quad \text{on } \OVS \mathcal{C}^-, \nonumber
\end{align}
has a unique solution $u_{\epsilon}\in B_2$ satisfying $\norm{u_\eps}_{C_b(\OVS M)}\leq c_{A,K}\s \eps \norm{f}_{C_K(\OVS M)}$. 
Let us define functions $r_\eps\in C_b(\OVS M)$, for $\eps<\eps_0$, by
\[
 u_\eps=\eps v_0 +r_\eps,
\]
where $v_0\in C_b(\OVS M)$ solves
%
%
\begin{align}\label{intergrate_vlas}
\XX v_{0}&= f \quad \text{on } \OVS M\\
v_{0}&= 0 \quad \text{on } \OVS \mathcal{C}^-. \nonumber
\end{align}
We show that $r_\eps=\mathcal{O}(\eps^2)$ in $C_b(\OVS M)$. To show this, we first calculate
\[
 \XX r_\eps=\XX(u_\eps -\eps v_0)=\COLOP[u_{\epsilon},u_{\epsilon}]+\eps f -\eps \XX v_0=\COLOP[u_{\epsilon},u_{\epsilon}].
\]
We integrate this equation along the flow of $\XX$ to obtain
\[
 r_\eps(x,p)=\int_{-\infty}^0\COLOP[u_{\epsilon},u_{\epsilon}]\big( \gamma_{(x,p)} (s) ,   \dot\gamma_{(x,p)} (s)  \big) ds. 
\]
(As before, the integral is actually over a finite interval since $u_\eps$ vanishes in $\pi^{-1}(\mathcal{C}^-)$.)
By Lemma~\ref{lislem}, we have that the right hand side is at most
\[
 C_1\norm{u_\eps}^2_{C_b(\OVS M))}.
\]
Since $\norm{u_\eps}_{C_b(\OVS M)}\leq C\eps\norm{f}_{C_K( \OVS M )}$ by Theorem~\ref{boltz-exist}, we have that $r_\eps=\mathcal{O}(\eps^2)$ in $C_b(\OVS M)$ as claimed.
%
%
%
%
%
Consequently, we have that 
 \[  
 \lim_{\epsilon\to0}\frac{\Phi(\epsilon f)-\Phi(0)}{\epsilon} = \lim_{\epsilon\to0}\frac{u_{\epsilon}-0}{\epsilon}=\lim_{\epsilon\to0}\frac{\eps v_0 +r_\eps}{\epsilon}=v_0=\Phi^L(f), 
 \]
 where the limit is in $C_b( \OVS M )$. This proves Part (1).
%

{\it Proof of (2).}  Let $f$, $u_\eps$ and $v_0$ be as before. We first prove that 
\[
 u_\eps = \eps v_0 + \eps^2w + \mathcal{O}(\eps^3)
\]
in $C_b(\OVS M)$, where $w$ is the unique solution to
%
\begin{align}\label{integrated_second_lin}
\XX w &= \COLOP[v_0,v_0] 
\quad \text{ on } \OVS M\\
w &= 0, \quad \qquad \ \  \text{ on }\OVS \mathcal{C}^- \nonumber.
\end{align}
A unique solution to~\eqref{integrated_second_lin} exists by using the formula~\eqref{vlasov_integrated_form_appendix} and noting the $\pi(\text{supp}A)$ is compact.
To show this, we define $R_\eps\in C_b(\OVS M)$ by
\begin{equation}\label{Reps}
 u_\eps=\eps v_0-\eps^2w + R_\eps.
\end{equation}
To show that $R_\eps=\mathcal{O}(\eps^3)$ in ${C_b(\OVS M)}$ we apply $\XX$ to the equation~\ref{Reps} above. We have that
\begin{equation}\label{eqforReps}
 \XX R_\eps=\XX(u_\eps-\eps v_0-\eps^2w)=\COLOP[u_\eps,u_\eps]+\eps f-\eps\XX v_0 -\eps^2\COLOP[v_0,v_0] =\COLOP[u_\eps,u_\eps]-\eps^2\COLOP[v_0,v_0].
\end{equation}
Since the collision operator $\COLOP$ is linear in both of its arguments, we have that
\begin{align}\label{telescopic_Q}
 &\COLOP[u_\eps,u_\eps]-\eps^2\COLOP[v_0,v_0]= \COLOP[(u_\eps-\eps v_0),u_\eps]-\COLOP[\eps v_0,(u_\eps-\eps v_0)].
\end{align}

We integrate the equation~\eqref{eqforReps} for $R_\eps$ along the flow of $\XX$ to obtain 
\[
 R_\eps(x,p)
 \leq C_1\norm{u_\eps-\eps v_0}_{C_b(\OVS M)}\norm{u_\eps}_{C_b(\OVS M)}+C_1\norm{\eps v_0}_{C_b(\OVS M)}\norm{u_\eps-\eps v_0}_{C_b(\OVS M)}.
\]
Here we used~\eqref{telescopic_Q} and Lemma~\ref{lislem}.
By using the estimate $\norm{u_\eps-\eps v_0}_{C(\OVS M)}\leq C\eps^2$ from Part (1) of this lemma, and by using that $\norm{u_\eps}_{C_b(\OVS M)}\leq C\eps\norm{f}_{C_K( \OVS M )}$  and that $\norm{v_0}_{C_b(\OVS M)}\leq C_2\norm{f}_{C_K( \OVS M )}$ we obtain
\[
 \norm{R_\eps(x,p)}_{C_b(\OVS M)}\leq C_3\eps^3.
\]
We have shown that
\[
 u_\eps=\eps v_0-\eps^2w + \mathcal{O}(\eps^3).
\]
It follows that $\Phi$ is twice Frech\'et differentiable at the origin in $C_b(\OVS M)$.

Let $f,\, h\in C_K(\OVS \mathcal{C}^+)$. To prove Part (2) of the claim, we use ``polarization identity of differentiation'', which says that any function $F$, which is twice differentiable at $0$, satisfies
\[
  \frac{\p^2 }{\p \eps_1\eps_2}\Big|_{\eps_1=\eps_2=0}F(\eps_1f_1+\eps_2f_2)=\frac{1}{4}\frac{\p^2}{\p\eps^2}\Big|_{\eps=0}[F(\eps(f_1+f_2))-F(\eps(f_1-f_2))].
 \]
%
%
%
%
For $f\in B_1\subset C_K(\OVS \mathcal{C}^+ )$, we denote by $u_f$ the solution to the Boltzmann equation~\eqref{intergrate_bman} with source $f$. We also denote by $v_f$ the solution to the Vlasov equation~\eqref{intergrate_vlas} with source $f$, and we denote similarly for $w_f$, where $w_f$ solves~\eqref{integrated_second_lin} where $v_0$ is replaced by $v_f$. 

We need to show that
\begin{equation}\label{pol_id}
 \frac{1}{\eps_1\eps_2}\left[\Phi(\eps_1f_1+\eps_2f_2)-\Phi(\eps_2f_2)-\Phi(\eps_1f_1)+\Phi(0)\right]\longrightarrow w,
\end{equation}
as $\eps_1\to 0$ and $\eps_2\to 0$, where $w$ solves
\begin{alignat}{2}
\XX w &= \COLOP [ v_{f_1},  v_{f_2} ]+  \COLOP [ v_{f_2},  v_{f_1}], 
\quad &&\text{ on } \OVS M\\
w &= 0,\quad  &&\text{ on }\OVS \mathcal{C}^- \nonumber.
\end{alignat}
By using the polarization identity~\eqref{pol_id} and the expansion of $u_{\eps(f_1\pm f_2)}$ for $\eps$ small, which we have already proven, we obtain
\begin{align*}
\lim_{\eps_1,\, \eps_2\to 0}\frac{\Phi(\eps_1f_1+\eps_2f_2)-\Phi(\eps_2f_2)-\Phi(\eps_1f_1)+\Phi(0)}{\eps_1\eps_2} &= \frac{1}{4}\frac{\p^2}{\p\eps^2}\Big|_{\eps=0}[u_{\eps(f_1+f_2)}-u_{\eps(f_1-f_2)}] \\
&=\frac{1}{2}w_{f_1+f_2}-\frac{1}{2}w_{f_1-f_2}. 
\end{align*}
By denoting $w_{f_1-f_2}-w_{f_1+f_2}=2w$ and by using the linearity of $\COLOP$ in both of its argument, we finally have have that
\[
 \XX w=\frac{1}{2}\COLOP[v_{f_1+f_2},v_{f_1+f_2}]-\frac{1}{2}\COLOP[v_{f_1-f_2},v_{f_1-f_2}]= \COLOP[v_{f_1},v_{f_2}]+\COLOP[v_{f_2},v_{f_1}].
\]
Renaming $f_1$ and $f_2$ as $f$ and $h$ respectively proves the claim.

\end{proof}

%

%

\bibliographystyle{acm}
\bibliography{Sections/citations-boltzmann}

\end{document}